\documentclass{amsart}          

\usepackage{amsmath,amsthm}     
\usepackage{amssymb,bbm}            
\usepackage{euscript}           
\usepackage{array,enumerate,calc,graphicx,xcolor}
\usepackage[matrix,arrow,curve,frame]{xy}    
\usepackage[colorlinks]{hyperref}
\usepackage{stmaryrd}
\usepackage{tikz-cd}

\makeatletter
\newcommand*\bigcdot{\mathpalette\bigcdot@{.7}}
\newcommand*\bigcdot@[2]{\mathbin{\vcenter{\hbox{\scalebox{#2}{$\m@th#1\bullet$}}}}}
\makeatother

\usepackage{tikz}
\usetikzlibrary{matrix,arrows,decorations.pathmorphing}

\xymatrixcolsep{1.9pc}                          
\xymatrixrowsep{1.9pc}
\newdir{ >}{{}*!/-5pt/\dir{>}}                  

\newtheorem{defn}[subsection]{Definition}
\newtheorem{prop}[subsection]{Proposition}

\newtheorem{cor}[subsection]{Corollary}
\newtheorem{lemma}[subsection]{Lemma}

\theoremstyle{definition}  
\newtheorem{example}[subsection]{Example}

\newtheorem{remark}[subsection]{Remark}

  {\end{list}}

\newcommand{\dfn}{\textbf} 

\newcommand{\mdfn}[1]{\dfn{\mathversion{bold}#1}} 

\newcommand{\tens}              {\otimes}               

\newcommand{\iso}               {\cong}

\newcommand{\cat}{\EuScript}    
\newcommand{\cB}{{\cat B}}      

\newcommand{\cD}{{\cat D}}

\newcommand{\Mod}{\text{Mod}}

\newcommand{\field}[1]  {\mathbb #1} 

\newcommand{\D}          {\field D}

\newcommand{\R}         {\field R}

\newcommand{\Z}         {\field Z}
\newcommand{\C}         {\field C}
\newcommand{\Q}         {\field Q}

\DeclareMathOperator*{\im}{Im}

\DeclareMathOperator{\Hom}{Hom}
\DeclareMathOperator{\uHom}{\underline{Hom}}
\DeclareMathOperator{\uExt}{\underline{Ext}}

\DeclareMathOperator{\uTor}{\underline{Tor}}
\DeclareMathOperator{\Ext}{Ext}
\DeclareMathOperator{\Tor}{Tor}

\DeclareMathOperator{\tr}{tr}

\newcommand{\ra}{\rightarrow}                   
\newcommand{\lra}{\longrightarrow}              
\newcommand{\lla}{\longleftarrow}               
\newcommand{\llra}[1]{\stackrel{#1}{\lra}}      
\newcommand{\llla}[1]{\stackrel{#1}{\lla}}      

\newcommand{\inc}{\hookrightarrow}              

\newcommand{\blank}{-}                          
\newcommand{\id}{id}                            
\newcommand{\und}{\underline}

\newcommand{\bdot}{\bullet}

\newcommand{\he}{\simeq}
\newcommand{\pt}{pt}

\numberwithin{equation}{section}

\newcommand{\Ab}{\cat Ab}

\newcommand{\bbox}{\Box}

\newcommand{\MMack}[2]{\xymatrix{
{#1} \ar@(ul,dl)[] \ar@<0.5ex>[r] & {#2}\ar@<0.5ex>[l]}}

\newcommand{\dbox}{\bbox \!\!\!\!\bbox}  

\newcommand{\MMod}{\!-\!\Mod}

\newcommand{\mZ}{\underline{\smash{\Z}}}
\newcommand{\mQ}{\underline{\smash{\Q}}}

\DeclareMathOperator{\Res}{Res}
\DeclareMathOperator{\Ind}{Ind}

\DeclareMathOperator{\lcm}{lcm}

\newcommand{\uH}{\underline{H}}

\DeclareMathOperator{\Div}{Div}
\newcommand{\ucD}{\underline{\cD}}
\newcommand{\bbone}{\mathbf{1}}
\newcommand{\ts}{\chi}

\newcommand{\mInd}{\underline{\Ind}}
\newcommand{\mRes}{\underline{\Res}}

\newcommand{\lamu}[1]{\lambda_{\und{#1}}}
\DeclareMathOperator{\tors}{tors}
\DeclareMathOperator{\len}{length}

\numberwithin{equation}{subsection}

\newcommand{\cln}{\colon\!}
\newcommand{\lur}[1]{\langle \underline{#1}\rangle}

\begin{document}
\title[Bredon Equivariant cohomology for cyclic groups]{The Bredon equivariant cohomology of a point for cyclic groups}

\author{Daniel Dugger}
\address{Department of Mathematics, University of Oregon, Eugene, OR 97403}
\author{Christy Hazel}
\address{Department of Mathematics, Grinnell College, Grinnell, IA 50112}

\begin{abstract}
We study the $RO(G)$-graded Bredon cohomology of a point in the case where $G$ is a cyclic group of odd order, expanding on the information provided by previous studies such as \cite{Z} and \cite{BD}.   Our methods center on the purely algebraic aspects of this matter, which interpret it as the ``stable homotopy groups of spheres'' problem for the derived category of modules over the constant-coefficient Mackey ring.
\end{abstract}

\maketitle

\tableofcontents

\section{Introduction}
\label{se:intro}

Let $G$ be a finite group and let $\mZ_G$ (or just $\mZ$) denote the constant-coefficient Mackey functor with value $\Z$.  The $RO(G)$-graded Bredon cohomology theory $H^\star(\blank;\mZ)$ has been much-studied in recent years, with interest getting a strong boost after its role (for $G=C_8$) in the solution of the Kervaire invariant one problem \cite{HHR1}.  Despite the notable interest in the theory, there are still many mysteries even around the ground ring $H^\star(\pt;\mZ)$.  Recent papers that focus on computing this ground ring for various groups $G$ include \cite{AMR}, \cite{BD}, \cite{BG1}, \cite{BG2}, \cite{DV}, \cite{G1}, \cite{HK}, \cite{KL}, \cite{Li}, \cite{Lu}, \cite{Ya1}, \cite{Ya2}, \cite{Z}, and there are undoubtedly many others.  
But while there are a host of papers on this problem, they sometimes offer only partial information or else information that is so complicated that it can be unclear how to make use of it.   

Our motivation in the present work was to take another look at this problem, but from a very concrete perspective.  The computation of $H^\star(\pt;\mZ)$ can be regarded as a purely algebraic problem: in the derived category $\cD(\mZ_G)$ of $\mZ_G$-modules one wants to write down all the invertible objects and compute all maps between them.  We wanted to understand what parts of this are easy and what parts are hard, and why the hard parts are hard.  And as a starting point we were satisfied to focus on the case where $G$ is a cyclic group.  

To jump to the end, the final results of our investigation are mixed.  We still cannot give a complete description of the ring $H^\star(\pt;\mZ)$ except when the cyclic group $G$ is very small (where ``small'' means that the order has at most three divisors, and so is either a prime or the square of a prime).   Nevertheless, we can give quite a bit of information about these rings. And we think we can give an overall 
guide to these rings that is a bit more user-friendly than perhaps has been available before.
\medskip

To begin to state the results in more detail, let $G=C_n$ be the cyclic group of order $n$, with a fixed generator $t$.  Let $\lambda(d)$ be $\R^2=\C$ with the action of $C_n$ where $t$ acts as multiplication by $e^{2\pi i\cdot d/n}$.  
When $n$ is odd the representations $1$ and $\lambda(d)$ for $1\leq d\leq \tfrac{n-1}{2}$ give a complete set of irreducibles for $C_n$.
When $n$ is even the sign representation on $\R$ must be added to the list.  For the rest of this paper we assume that $n$ is odd; the methods and results can all be adapted for $n$ even, but things are complicated enough without trying to deal with the two cases simultaneously.  

So for $n$ odd we have that $RO(C_n)=\Z^{\frac{n+1}{2}}$, and this is the group that grades our cohomology theory.  It is known that in fact there are several ``extra periodicities'' present in this setting, and that it is sufficient to restrict to the subgroup generated by the $\lambda(d)$ where $d|n$.  We will not explain this fully at the moment (see Remark~\ref{re:cells} for a discussion), but from the algebraic perspective we adopt below, this is a non-issue since the $\lambda(d)$ where $d\nmid n$ do not even appear.  

When $d|n$ we set $\lambda_d=\lambda(\tfrac{n}{d})$.  This re-indexing is a little annoying at first, but becomes extremely convenient.  Note that $\lambda_1$ is the rank $2$ trivial representation.

Let $\Div_n$ denote the set of positive divisors of $n$. In the free abelian group $\Z\langle \Div_n\rangle$ it will be useful to have the symbol $\lambda_b$---in addition to its above meaning---also denote the generator of $\Div_n$ corresponding to $b\in \Div_n$.  Let $\D_n=(\Z\langle 1\rangle \oplus\Z\langle \Div_n\rangle)/(1+1=\lambda_1)$, and observe that we can regard $\D_n$ as a subgroup of $RO(C_n)$.  It will sometimes be helpful to adopt the notation $\lambda_0=1\in \D_n$, just as a convenience.  The group $\D_n$ will replace $RO(C_n)$ as the preferred grading for our cohomology theory.   When we write $H^\star(\blank;\mZ)$ the intention is that $\star\in \D_n$, and we will reserve the asterisk $*$ for integer gradings.  

To talk about the derived category $\cD(\mZ)$ we first need some language and tools from the category of $\mZ$-modules. These were developed at length in \cite{DHone}, and here we will just give a very brief summary (more detail is in Section~\ref{se:Zmod} below).  Write $\Theta_d$ for the orbit $C_n/(t^d)$. Some basic objects in the category of $\mZ$-modules are:

\vspace{0.1in}
\begin{tabular}{|l|l|}
\hline
$F_d=F_{\Theta_d}$ & free $\mZ$-module on a generator in spot $\Theta_d$ (note that $F_1=\mZ$). \\
$I_d\subseteq \mZ$ & ideal generated by $1_d\in \mZ(\Theta_d)$.\\
$\mZ/I_d$ & the quotient module.  \\
$I_d/I_e$ & quotient, defined when $d\mid e$. \\ 
\hline
\end{tabular}
\vspace{0.1in}

\noindent
The category of $\mZ$-modules is often regarded as the equivariant analog of the category of abelian groups, but unlike the latter there is not a simple list of finitely-generated indecomposables.  The above list of basic objects is not anywhere close to being comprehensive, they just happen to be a convenient class with which to get started. 

The category of $\mZ$-modules has a closed symmetric monoidal structure, the tensor being the box product $\bbox$ and the co-tensor denoted $\uHom(\blank,\blank)$. 

For each $d|n$ we define $S^{\lambda_d}$ to be the length two chain complex
\[ 0 \lra F_d \llra{\id-Rt} F_d \llra{Rp} F_1 \lra 0.
\]
(see Section~\ref{se:Zmod} for our conventions about maps).  This is really the reduced cellular chain complex for an evident cell structure on the one-point compactification of the representation $\lambda_d$, but we can ignore the geometric origin for now and just take it as an algebraic object. See Remark~\ref{re:cells} for an explanation of the geometry. 
For a sequence $\und{b}=(b_1,\ldots,b_s)$ where each $b_i|n$, define
\[ S^{\lamu{b}}=S^{\lambda_{b_1}}\bbox \cdots \bbox S^{\lambda_{b_s}}.
\]
We will identify the symbol $\lamu{b}$ with the associated representation $\bigoplus_i \lambda_{b_i}$ when convenient.

With only a bit of trouble one can show algebraically that $S^{\lambda_d}$ is an invertible object in $\cD(\mZ)$ 
and that an inverse is $\uHom(S^{\lambda_d},\mZ)$
(see Proposition~\ref{pr:sphere-invertible}).  Define $S^{-\lambda_d}$ to be this complex, and for a sequence $\und{b}$ as above define
\[ S^{-\lamu{b}}=S^{-\lambda_{b_1}}\bbox \cdots \bbox S^{-\lambda_{b_s}}.
\]

To extend to a $\D_n$-grading,
fix once and for all the order $\leq$ of the divisors of $n$ and for $v=\sum_{d\neq 1} v_d\lambda_d\in \D_n$ define
\[ S^v=S^{v_0}\bbox \Bigl (\bbox_{v_d>0} (S^{\lambda_d})^{\bbox(v_d)}\Bigr ) \bbox \Bigl ( \bbox_{v_d<0} (S^{-\lambda_d})^{\bbox(-v_d)}
\Bigr )
\]
where the box product factors inside the parentheses are ordered using $\leq$.  For $v,w\in \D_n$ it is not true that $S^{v+w}$ {\it equals\/} $S^v\bbox S^w$, but one can specify a canonical isomorphism in the derived category  between them (see Section~\ref{se:D_stable} for details).

Given a $\beta\in \Div_n$ written as $\beta=m_0\lambda_0+\sum_i m_i \lambda_{b_i}$ with all $b_i> 1$ and $m_i\in \Z$, define $\dim \beta=m_0+2\sum_i m_i$.  Call this the \dfn{geometric dimension} of $\beta$.  This is just the real dimension of the associated representation.  Let us also define the \dfn{fixed dimension} of $\beta$ to be $m_0$. 

For $v\in \D_n$ one has the following fundamental connection between Bredon cohomology and the category $\cD(\mZ)$: 
\[ H^v(\pt;\mZ)=\cD(\mZ)(\mZ,S^v).
\]
Note that $\cD(\mZ)(\blank,\blank)$ denotes maps in $\cD(\mZ)$.  
Said more globally, the Bredon cohomology ring of a point is
\[ H^\star(\pt;\mZ)=H^\star(\pt)=H^\star=\bigoplus_{v\in \D_n} \cD(\mZ)(\mZ,S^v).
\]
The first three expressions will be used interchangeably.
The ring structure comes from the box product of maps.  Note that it is important here that $\D_n$ is free, otherwise there can be issues around associativity (see \cite{D} for a complete discussion).  
As an application of \cite[Proposition 1.2(3)]{D} one finds that $H^\star(\pt)$ satisfies the skew-commutative rule
\begin{equation}
\label{eq:skew-comm}
x\cdot y=(-1)^{(\text{fixed-dim. of $|x|$})\cdot (\text{fixed-dim. of $|y|$})} y\cdot x.
\end{equation}
See Section~\ref{se:D_stable} for details, as well as Corollary~\ref{co:comm-intro} below for a further, but non-obvious, simplification.

So now at last we see the purely algebraic problem in front of us: understanding $H^\star(\pt;\mZ)$ comes down to computing homotopy classes of maps $\mZ\ra S^{v}$ for $v\in \D_n$ and understanding how they multiply.  It is the ``stable homotopy groups of spheres'' problem but in the setting of $\cD(\mZ)$. \medskip

The category of $\mZ$-modules is enriched over itself, and this passes to the derived category $\cD(\mZ)$.  For chain complexes $X$ and $Y$ we write $\ucD(\mZ)(X,Y)$ for the $\mZ$-module function object, and we write
\[ \uH^v(X)=\ucD(X,S^v), \qquad \uH^v=\uH^v(\pt)=\uH^v(\mZ)=\ucD(\mZ,S^v).
\]
In words one often refers to this as ``Mackey-functor-valued cohomology''.  

The collection of groups $H^{\lamu{b}+k}(\pt)$ where $\und{b}$ is any sequence of divisors and $k\in \Z$ is called the \dfn{positive cone} inside of $H^\star(\pt)$.  It is easy to see that these groups vanish for $k<-\dim \lamu{b}$ and $k>0$, so that the nonzero groups lie in a wedge (or cone).  The positive cone is a subring of $H^\star(\pt)$.  
The collection of groups $H^{k-\lamu{b}}(\pt)$ where some $b_i>1$ is called the \dfn{negative cone}. Here they vanish for $k>\dim \lamu{b}$ or $k\leq 0$.   This is also a subring of $H^\star(\pt)$.  The positive and negative cones together are called the \dfn{regular region} of $H^\star(\pt)$.  Note that the regular region is {\it not\/} a subring, though.  

The homogeneous parts of $H^\star(\pt)$ that are outside the regular region are called \dfn{irregular}.  The general situation to keep in mind is that the positive cone is a Noetherian ring that in practice might be somewhat accessible, the negative cone is a non-Noetherian ring that can be difficult to describe but in some ways is not much more complicated than the positive cone, whereas the irregular region is a netherworld that seems to largely defy attempts at a complete description.  Often papers that claim to compute $H^\star_G$ leave out the irregular region altogether.  Note that when $G=C_\ell$, $\ell$ a prime, there is no irregular region in the $\D_\ell$-grading, making this case particularly simple.  \bigskip

Now we have set the stage and can talk about results.  
The complex $S^{\lambda_d}$ comes equipped with canonical maps 
\[ \mZ \llra{a_d} S^{\lambda_d}, \qquad \Sigma^2 \mZ \llra{u_d} S^{\lambda_d}.
\]
As elements of the cohomology ring these are $a_d\in H^{\lambda_d}(\pt)$ and $u_d\in H^{\lambda_d-2}(\pt)$; we will also write $|a_d|=\lambda_d$ and $|u_d|=\lambda_d-2$ for the degrees.    It is easy to compute the homology of the complex $S^{\lambda_d}$ by hand, revealing that
\[ \uH^{\lambda_d-i}(\pt)=\cD(\mZ,S^{\lambda_d-i})=\cD(\Sigma^i\mZ,S^{\lambda_d})=
\uH_i(S^{\lambda_d})\iso \begin{cases}
\mZ & \text{$i=2$}, \\
0 & \text{$i=1$ or $i\geq 3$}, \\
\mZ/I_d & \text{$i=0$}.
\end{cases}
\]
Here $u_d$ and $a_d$ are the generators in their respective degrees.   [Warning: A topologist might expect to see the {\it reduced} homology of the sphere appear in the above formula, but we remind the reader that our $S^{\lambda_d}$ is not the geometric sphere but rather its {\it reduced} cellular chain complex.]

The Mackey functors for $\uH^{\star}(\pt)$ in the regular region are likewise
\[ \uH^{\lambda_{d_1}+\cdots+\lambda_{d_s}-i}(\pt)=
\uH_i(S^{\lambda_{d_1}}\bbox \cdots \bbox S^{\lambda_{d_s}})
\]
and
\[ 
\uH^{i-\lambda_{d_1}-\cdots-\lambda_{d_s}}(\pt)=
\uH_{(-i)}(S^{-\lambda_{d_1}-\cdots-\lambda_{d_s}})
\iso \uH^i(S^{\lambda_{d_1}}\bbox \cdots \bbox S^{\lambda_{d_s}}    ). 
\]
For a general degree in the irregular region we need to compute
\[ \uH^{i+\lambda_{\und{d}}-\lambda_{\und{c}}}(\pt)= 
\ucD( S^{\lambda_{\und{c}}}, \Sigma^i S^{\lambda_{\und{d}}}).
\]
It is clear from these descriptions that all of the $\mZ$-modules in $\uH^\star(\pt)$ are finitely-generated.  
Moreover, in all of these instances one has K\"unneth or Universal-Coefficient spectral sequences that can be used (see \cite{Z} or \cite[Section 6]{DHone}), but these tend to be replete with nontrivial differentials that need to be detected. The issues goes away after rationalization, and one readily finds the following (see Proposition~\ref{pr:rationalization}):

\begin{prop}
\label{pr:rat1}
The ring $\uH^\star(\pt)\tens \Q$ is $\mQ[u_d, u_d^{-1} \,:\, d|n, d\neq 1]$.
\end{prop}

Since $|u_d|=\lambda_d-2$ and therefore $\dim |u_d|=0$ for all $d$,
the following result is immediate from Proposition~\ref{pr:rat1}:

\begin{cor}
\label{co:rat1}
$\uH^\beta(\pt)$ is torsion whenever $\dim \beta\neq 0$.
\end{cor}

In degrees where the rationalization map $\uH^\star(\pt)\ra \uH^\star(\pt)\tens \Q$ is injective, we can choose to name elements by their image.  This will be useful on multiple occasions.  In a moment we will give a complete description of the image of this map, but it will be helpful to first address other matters.
We also mention that both Proposition~\ref{pr:rat1} and Corollary~\ref{co:rat1} are well-known, but they are useful starting-off points for our story.

It turns out that there are some extra periodicities in $\uH^\star(\pt)$ that are not immediately obvious: for any divisors $c$ and $d$ of $n$ one has 
\[ S^{\lambda_c} \bbox S^{\lambda_d} \he S^{\lambda_{(c,d)}}\bbox S^{\lambda_{[c,d]}}
\]
where $(c,d)$ is the gcd and $[c,d]$ is the lcm.  (This can be compared, metaphysically, to the non-canonical isomorphism of abelian groups $\Z/c\oplus \Z/d \iso \Z/(c,d)\oplus \Z/[c,d]$.)  See Corollary~\ref{co:basic_equiv}.  Said slightly differently, the ring $H^\star(\pt)$ has a system of units
\[ \chi_{c,d}\in H^{\lambda_{(c,d)}+\lambda_{[c,d]}-\lambda_c-\lambda_d}
\]
for every pair of divisors $c$ and $d$.  The $\mZ$-module $\uH^{\lambda_{(c,d)}+\lambda_{[c,d]}-\lambda_c-\lambda_d}$ is isomorphic to $\mZ$, so this is a degree where the rationalization map is injective and we can adopt the notation
\[ \chi_{c,d}=\tfrac{u_{(c,d)}u_{[c,d]}}{u_cu_d}.
\]
The $\chi_{c,d}$ and their inverses, together with $\pm 1$, turn out to be the only homogeneous units in $H^\star(\pt)$. See Proposition~\ref{pr:units}.  

\begin{remark}
These extra periodicities allow one to reduce the grading group from $\D_n$ down to something much smaller.  For example, if $\ell_1$ and $\ell_2$ are primes then any $\lambda_{\ell_1^{e_1}\ell_2^{e_2}}$ that appears in a degree can be replaced by $\lambda_{\ell_1^{e_1}}+\lambda_{\ell_2^{e_2}}-\lambda_1$ to obtain an isomorphic group.  So up to isomorphism all computations are reduced to degrees that are formal sums of $\lambda_d$ where $d$ is a prime power.  For $n=\prod \ell_i^{e_i}$ this reduces the usual $\Z^{\varphi(n)}$-grading (where $\varphi(n)$ the Euler totient function) to a $\Z^{(1+
\sum e_i)}$-grading, which is typically substantially smaller.  (One can accomplish something similar via $\ell$-localization, which we discuss in a moment.)
\end{remark}

Note that the ``move'' $x,y\mapsto (x,y),[x,y]$ replaces an arbitrary pair of integers with a pair where the first divides the second.  If $x|y$ in the original pair, the move doesn't change anything.  Define a \dfn{divisor string} to be  a sequence of positive integers where each one divides the next.
Starting with any sequence of positive integers $\und{c}=(c_1,\ldots,c_s)$, applying the above move iteratively (it doesn't matter what order) eventually stabilizes at a new sequence $(\und{c};1),(\und{c};2),\ldots,(\und{c};s)$ that is a divisor string.  We call this the \mdfn{divisor string associated to $\und{c}$}.  One can show that $(\und{c};r)$ is the gcd of the $r$-at-a-time lcms of elements of the original sequence.  So at the two ends we have $(\und{c};1)=(c_1,\ldots,c_s)$ and $(\und{c};s)=[c_1,\ldots,c_s]$.  Alternatively, $(\und{c};r)$ is the product over all primes $\ell$ of the $r$th smallest power of $\ell$ that appears in the prime factorizations of $c_1,\ldots,c_s$.

Iterative multiplication by the $\chi$-units allows us to transform any group $H^{\lambda_{\und{c}}-\lambda_{\und{d}}+k}$ into the isomorphic group $H^{\lambda_{\und{c}'}-\lambda_{\und{d}'}+k}$ where $\und{c}'$ and $\und{d}'$ are the divisor strings associated to $\und{c}$ and $\und{d}$.  Consequently, the problem of understanding the groups $H^{\star}(\pt)$ reduces to the case where the positive and negative components of $\star$ are each a divisor string.\bigskip

The complex $S^{\lambda_d}$ is concentrated in degrees $0$ through $2$ and has three free modules appearing.  The box product $S^{\lambda_{d_1}}\bbox \cdots \bbox S^{\lambda_{d_s}}$ is concentrated in degrees $0$ through $2s$ and involves $3^s$ box products of free modules.  It turns out---and this is not at all obvious---that there
is a much smaller model: $S^{\lambda_{d_1}+\cdots+\lambda_{d_s}}$ is quasi-isomorphic to the complex
\begin{equation}
\label{eq:linear-sphere}
 0 \ra F_{D_s} \llra{\id-Rt} F_{D_s} \llra{I\pi R\pi} F_{D_{s-2}} \llra{\id-Rt} F_{D_{s-2}} \llra{I\pi R\pi} \cdots \lra F_{D_1} \llra{\id-Rt} F_{D_1} \llra{R\pi} \mZ \ra 0
\end{equation}
with the terms in homological degrees $0$ through $2s$ (the differentials are described in Section~\ref{se:smallmodels}).
Here $D_i=(\und{d};i)$ coming from the associated divisor string to the original sequence.  See Proposition~\ref{pr:spheres-small-model} for this quasi-isomorphism.
These ``linear'' models for the spheres are much easier to work with than the large box product complexes.
\bigskip

It is easy to compute the homology groups of the above linear complexes, yielding the following:

\begin{prop}
\label{pr:positive-cone-MF}
The homology modules in the positive cone are given by:
\[ \uH^{\lambda_{d_1}+\dots+\lambda_{d_s}-k}=\uH_k(S^{\lambda_{d_1}+\cdots+\lambda_{d_s}})\iso \begin{cases}
\mZ & \text{$k=2s$},\\
0 & \text{$k$ odd or $k\geq 2s+1$},\\
\mZ/I_{D_{1+\frac{k}{2}}} & \text{$k$ even and $0\leq k\leq 2s-2$.}
\end{cases}
\]
\end{prop}

The indexing on the $D$'s is a little unfortunate, but the point is that one gets precisely the $D$-sequence, in order, in the even degrees from $0$ up through $2s-2$.   

One can also completely describe the positive cone as a ring:

\begin{prop}
\label{pr:positive-cone-intro}
The positive cone is the quotient of the graded Mackey ring $\mZ[a_d,u_d: d|n, 1<d]$ by the relations 
\begin{itemize}
\item $(p_{\Theta_d\ra \Theta_1})^*(a_d)=0$ (Euler class relation),
\item $\tfrac{d}{(d,e)}a_du_e=\tfrac{e}{(d,e)}a_eu_d$ for all $d,e$ (gold relations).
\end{itemize}
\end{prop}

A similar description of the positive cone was first obtained by Basu-Dey \cite{BD}.  

Note that the information in the above two propositions is complementary.  Just given the information in Proposition~\ref{pr:positive-cone-intro} it is not at all immediately clear what the $\mZ$-modules actually are in each degree, but Proposition~\ref{pr:positive-cone-MF} identifies them.

Every homogeneous component of the positive cone is generated by products of $a$-classes and $u$-classes (abbreviated as ``$au$-classes'' from now on).  Although the groups are cyclic, one cannot typically find a single $au$-class that is a generator.  When $n$ is a prime power this is always possible, but not in general.

\begin{remark}
The order of the above exposition is convenient in terms of motivating the ideas and the results of the computations, but in the body of the paper it sometimes turns out to be easier to prove things via a different path.  So for example, in the paper we first prove Proposition~\ref{pr:positive-cone-intro} and then later deduce the linear models for the spheres from that, which is the opposite of the way we have told the story here.  
\end{remark}

To pass to the negative cone one uses the functor $\uHom(\blank,\mZ)$.  Applying this to our linear models for the spheres $S^{\lambda_{\und{d}}}$ gives linear models for $S^{-\lambda_{\und{d}}}$.  One can compute the homology by hand, but it also immediately comes out of the universal coefficient spectral sequence (which collapses here):

\begin{prop} 
\label{pr:neg-cone-intro}
The homology modules in the negative cone are given by
\[ \uH^{k-\sum_i\lambda_{d_i}}=\uH_{-k}(
S^{-\lambda_{d_1}-\cdots-\lambda_{d_s}})\iso
\begin{cases}
I_{D_s} & \text{$k=2s$}, \\
0 & \text{$k\neq 2s$ even, or $k>2s$, or $k<0$} \\
\mZ/I_{D_{\frac{k-1}{2}}} & \text{$3\leq k\leq 2s-1$}.
\end{cases}
\]
\end{prop}

Again, note that the main point is that one gets the associated divisor string $D_*$, this time appearing in odd degrees from $3$ up to $2s-1$, with the final term $D_s$ appearing in a slightly different guise in degree $2s$.\bigskip

Just as every homogeneous component of the positive cone was generated by $au$-classes,  one can give a similar (but more complex) description of the homogeneous elements of the negative cone.  The catch is that this now requires {\it division\/} by $au$-classes instead of multiplication.  

First, for every sequence of divisors $d_1,\ldots,d_s$ we have the generator
\[ \frac{[d_1,\ldots,d_s]}{u_{d_1}\cdots u_{d_s}} \in H^{2s-\lambda_{d_1}-\cdots-\lambda_{d_s}}(\pt)
\]
(recall that $[d_1,\ldots,d_s]$ denotes the lcm).
This generates the $\Z$ in the $\Theta_1$-component of the Mackey functor which by Proposition~\ref{pr:neg-cone-intro} is isomorphic to $I_{[d_1,\ldots,d_s]}$.  Once again, this is a degree where the rationalization map is injective and so we can name classes by their image.  Note that the
$u_i$ are not invertible,  but the notation says that there is a  unique class in this degree with the property that multiplying it by $u_{d_1}\cdots u_{d_s}$ gives $[d_1,\ldots,d_s]$. We will call these ``$\frac{?}{u}$-classes''. 
They give a complete description for the non-torsion part of the negative cone.  

Next, for every $b|n$ there is a special class $\gamma_b\in H^{3-2\lambda_b}(\pt)\iso \Z/b$ that generates this group.  This is obtained as a kind of ``equivariant Bockstein'', in the following sense.  As explained in \cite{DHone}, $\Ext^3(\mZ/I_b,\mZ)\iso \Z/b$ and one can write down a canonical generator $B_b$ using the standard resolution of $\mZ/I_b$.  In our derived category this is a map $B_b\colon \mZ/I_b \ra \Sigma^3 \mZ$.  We also have a canonical map $\bbone_b\colon S^{\lambda_b}\ra \mZ/I_b$, yielding $S^{\lambda_b+\lambda_b}\ra \mZ/I_b\bbox \mZ/I_b =\mZ/I_b$; this is an element $\bbone_b\bbox \bbone_b\in [S^{2b},\mZ/I_b]$.   The element $\gamma_b$ is the composition $B_b\circ (\bbone_b\bbox \bbone_b)$.  This description via the Bockstein is useful for uncovering some of the properties of these classes; see Propositions~\ref{pr:Gamma-ident} and \ref{pr:u-gamma}. 

Finally, for every divisor string $b_1,\ldots,b_s$ we have elements
\begin{equation}
\label{eq:intro-gamma-class}
 \frac{a_{b_j}\gamma_{b_j}}{u_{b_1}^{i_1}\cdots u_{b_{j-1}}^{i_{j-1}} a_{b_{j}}^{i_{j}}\cdots a_{b_s}^{i_s}  }
 \in 
 H^{(3+2(i_1+\cdots+i_{j-1})-\lambda_{b_j}-\sum_k i_k\lambda_{b_k})}\iso \Z/b_j
\end{equation}
where the exponents in the denominator are non-negative and at least one of the exponents on the $a$-classes is nonzero.
Note that the notation here is very dicey---these are torsion elements and so map to zero in the rationalization,
and the product $a_b\gamma_b$ is itself actually zero.
So the fraction only makes sense as a piece of notation, not as a true fraction.  It does have the property  that products involving these classes are internally consistent with the notation, as long as one is careful about what the allowable rules of manipulation are.  We show this in Proposition~\ref{pr:neg-cone-products}.

We will refer to any expression as in (\ref{eq:intro-gamma-class}) as a ``negative $\gamma$-class''. 
Note that $\gamma_b=\tfrac{a_b\gamma_b}{a_b}$ and so $\gamma_b$ is itself a ``negative $\gamma$-class'' according to this definition.

While (\ref{eq:intro-gamma-class}) looks intimidating, what it says in the end is that we can take an expression $a_b\gamma_b$ and divide it by any number of $u$-classes with indices forming a divisor string ending in $b$, as well as any number of $a$-classes with indices forming a divisor string beginning with $b$, as long as we are dividing by at least one $a$-class.    

\begin{prop}
The negative cone is additively generated by the negative $\gamma$-classes and the $\frac{?}{u}$-classes in degrees involving a divisor string, and in other degrees by the products of those with $\chi$-classes.  
\end{prop}

Multiplying two $\gamma$-classes is easy: the product of two such classes always lands in a degree where the torsion group is zero, and hence the product must be zero.  There are partial rules for multiplying $\gamma$-classes with certain $au$-classes and with the $\frac{?}{u}$-classes.  The rules are too lengthy to recount here in the introduction, but see Propositions~\ref{pr:integral-edge-products} and \ref{pr:neg-cone-products}.

For the interested reader we mention the existence of Section~\ref{se:magic}, entitled  ``Black magic for the negative cone'', which has some tricks for remembering how to multiply elements in this region.  But as introductions should be kept civil and genteel, we say no more on this here. 
\medskip

Next we turn to groups in the irregular region.  These have the form $H^{\lambda_{\und{d}}-\lambda_{\und{c}}+k}(\pt)$ where neither $\und{d}$ nor $\und{c}$ consists entirely of $1$s.  This group is isomorphic to the set of chain homotopy classes
 $[S^{\lambda_{\und{c}}},\Sigma^k S^{\lambda_{\und{d}}}]$.
Given the simple description of the spheres from (\ref{eq:linear-sphere}) one might expect that these groups could be calculated directly, with ease.  In some sense that is the case---for any particular values of $\und{c}$ and $\und{d}$, the group can be calculated algorithmically.  Unfortunately, the answers don't seem to admit a simple formula when the length of the sequences is at least three, even when working locally at a prime.  We investigate these matters in Section~\ref{se:irregular} and give some partial information, but the state of our knowledge is unsatisfying here.  

The following proposition gives a starting point for reducing certain degrees to simpler ones:

\begin{prop} 
\label{pr:irregular}
Let $\und{c}=(c_1,\ldots,c_q)$ and $\und{d}=(d_1,\ldots,d_s)$ be divisor strings (possibly empty).
\begin{enumerate}[(a)]
\item Suppose there is at least one $c_i$ and that $k\geq 4$. Then
\[ \uH^{\lamu{d}-\lamu{c}+k} \llra{u_{c_1}} \uH^{\lamu{d}-\lamu{c'}+(k-2)}
\]
is an isomorphism, where $c'$ denotes the sequence $\und{c}$ with $c_1$ removed.  
\item Suppose there is at least one $d_i$ and that $k<0$.  Then
\[ \uH^{\lamu{d'}-\lamu{c}+k+2} \llra{u_{d_1}} \uH^{\lamu{d}-\lamu{c}+k}
\]
is an isomorphism, where $d'$ denotes the sequence $\und{d}$ with $d_1$ removed.
\item If $k$ is odd then $H^{\lamu{d}-\lamu{c}+k}$ is generated by products of $(au)_{\und{d}}$-classes and negative $\gamma_{\und{c}}$-classes.  
\item Suppose that $q>s$ (so there are more $c$'s than $d$'s).   Then there is an isomorphism
\[ \xymatrixcolsep{4.6pc}\xymatrix{
H^{\lamu{d}-\lamu{c}}\ar[r]^-{\cdot a_{c_{s+1}}\cdots a_{c_q}}  & \tors H^{\lamu{d}-\lamu{c'}} 
}
\]
where $\und{c}'=(c_1,\ldots,c_s)$ and $\tors$ denotes the torsion subgroup. 
\item Suppose that $q<s$ (so there are more $d$'s than $c$'s).  There there is an exact sequence
\[ \xymatrixcolsep{3.5pc}\xymatrix{
0 \ar[r] & \Z \ar[r] & H^{\lamu{d'}-\lamu{c}} \ar[r]^-{\cdot a_{d_{q+1}\cdots a_{d_s}}} & H^{\lamu{d}-\lamu{c}} \ar[r] & 0
}
\]
where $\und{d}'=(d_1,\ldots,d_q)$.  
\end{enumerate}
\end{prop}

Part (c) of the above result actually leads to the following (see Corollary~\ref{co:comm} for the proof):

\begin{cor}
\label{co:comm-intro}
The ring $H^\star(\pt)$ is commutative (not just graded-commutative): that is, $xy=yx$ for all $x$ and $y$.    
\end{cor}

\begin{remark}
Starting in parts (c)--(e) of Proposition~\ref{pr:irregular} the reader will notice a shift as we stop talking about the $\mZ$-modules $\uH^\star$ and instead talk mostly about the groups $H^\star$.  There is a reason for this: unlike the positive and negative cones, the Mackey functors in the irregular region can be quite complicated and there is no known simple description of them in terms of basic building blocks.  This issue was probably first observed in \cite{BG3}.    
\end{remark}

Parts (a) and (b) of Proposition~\ref{pr:irregular} 
allow us to reduce the actual calculations to $H^{\lamu{d}-\lamu{c}+k}$ where $0\leq k\leq 3$.  Any degree with a $k$-value outside this range can be reduced to one inside this range.  Parts (c), (d), and (e) say that in terms of naming elements we can further reduce to the groups $H^{\lamu{d}-\lamu{c}}$ 
where $\und{c}$ and $\und{d}$ have the same length.
 
To demonstrate some of this material let us consider the following extended example.
Let $(c_1,c_2)$ and $(d_1,d_2)$ be divisor strings.  Figure~\ref{fig:gens} lists generators for all of the irregular groups in $H^\star(\pt)$ where the degrees involve at most two $\lambda$-classes in each of the positive and negative parts.

To read this table, first recall that $S^{\lamu{c}}$ is modeled by a complex concentrated in degrees $0$ through $2\cdot \len(c)$ (and likewise for $S^{\lamu{d}}$).  From this it follows that
$H^{\lamu{d}-\lamu{c}+k}=[S^{\lamu{c}},\Sigma^k S^{\lamu{d}}]$ automatically vanishes when $k<-2\len(d)$ or when  $k>2\len(c)$. In fact, one can strengthen the first inequality to $k<2-2\len(d)$ using Proposition~\ref{pr:irregular}(b) and induction on $\len(d)$. The blank spaces in the table are these regions (and the regions outside the scope of the table are blank for the same reason).  

In the cases where $k$ is odd, we simply wrote down all possible products of $au_{\und{d}}$-classes and negative $\gamma_{\und{c}}$-classes, keeping in mind what we know about the positive and negative cones. There was always just at most one generator except for $k=1$ in the last column, where we used the gold relation to recognize that one of the classes was a multiple of the other.

For the $k=4$ row one uses Proposition~\ref{pr:irregular}(a) to reduce (via multiplying by $u_{c_1}$) to a group of the general type in row $k=2$. This is almost the group that is a knight's move two steps up and one left in the table, but with $c_2$ replacing $c_1$.  For $k=-2$ one uses Proposition~\ref{pr:irregular}(b) to make a similar reduction that involves division by $u_{d_1}$.

\begin{figure}[ht]
\!\!\!\!\!\begin{tabular}{|c||c|c|c|c|}
\hline
$k$ & $\!H^{\lambda_{d_1}-\lambda_{c_1}+k}$ & $\!H^{\lambda_{d_1}-\lambda_{c_1}-\lambda_{c_2}+k}$ & $\!H^{\lambda_{d_1}+\lambda_{d_2}-\lambda_{c_1}+k}$ & $\!H^{\lambda_{d_1}+\lambda_{d_2}-\lambda_{c_1}-\lambda_{c_2}+k}$ \\[0.1in]
\hline\hline
$-2$ &  &  & {$u_{d_1} \tfrac{c_1}{(c_1,d_2)}\tfrac{u_{d_2}}{u_{c_1}}$}  & $0$   \\[0.1in] \hline
$-1$ &  & & 0 & $u_{d_1}u_{d_2}\tfrac{a_{c_1}\gamma_{c_1}}{a_{c_2}}$ \\[0.1in] \hline
$0$ & \fbox{$\tfrac{c_1}{(c_1,d_1)} \tfrac{u_{d_1}}{u_{c_1}}$} & $0$ & $\tfrac{c_1}{(c_1,d_1)}\tfrac{u_{d_1}}{u_{c_1}}\cdot a_{d_2}$ & \fbox{[H]} \\[0.1in] \hline
$1$ & 0 & $u_{d_1} \tfrac{a_{c_1}\gamma_{c_1}}{a_{c_2}}$ & $0$ & $ u_{d_1}a_{d_2} \tfrac{a_{c_1}\gamma_{c_1}}{a_{c_2}}$ \\[0.1in] 
\hline
$2$ & $a_{d_1}\tfrac{c_1}{u_{c_1}}$ & \fbox{$\tfrac{c_1c_2}{(c_1c_2,d_1)}\cdot \tfrac{u_{d_1}}{u_{c_1}u_{c_2}}$} & $a_{d_1}a_{d_2}\tfrac{c_1}{u_{c_1}}$ & $a_{d_2}\tfrac{c_1c_2}{(c_1c_2,d_1)}\tfrac{u_{d_1}}{u_{c_1}u_{c_2}}$ \\[0.1in]
\hline
$3$ &  & $a_{d_1}\tfrac{a_{c_1}\gamma_{c_1}}{a_{c_2}}$
&  & $a_{d_1}a_{d_2}\tfrac{a_{c_1}\gamma_{c_1}}{a_{c_2}}$ \\[0.1in]
\hline
$4$ &  & $a_{d_1}\tfrac{c_2}{u_{c_1}u_{c_2}}$ &  & $a_{d_1}a_{d_2}\tfrac{c_2}{u_{c_1}u_{c_2}}$\\[0.1in]
\hline
\end{tabular}
\caption{Generators for low-degree cohomology groups in the irregular region}
\label{fig:gens}
\end{figure}

For $k=2$ we can interpret the $+2$ as a $\lambda_1$, essentially adding a $1$ to the start of the $d$-string. Then Proposition~\ref{pr:irregular}(e) takes care of the first, third, and fourth columns. (Note here this reduces from column $x$ to column $x-2$, e.g. we move from column $4$ to column $2$ via division by $a_{d_2}$. Also note $H^{2-\lambda_{c_1}}$ is in the regular region---in fact, in the negative cone---and is generated by $c_1/u_{c_1}$, but this can also be seen in the $d_1=1$, $k=0$ entry of the first column.)

The boxed terms are the ones where the number of $c$'s is equal to the number of $d$'s (keeping in mind that $\lambda_1=2$ and therefore we can regard a $+2$ as a $d$-term when convenient) and where the group cannot be reduced to another one via the techniques of Proposition~\ref{pr:irregular}.  These groups have to be computed directly, as is done in Section~\ref{se:L-maps}.  The two boxes in the first two columns turn out to be copies of $\Z$, whereas the one in the last column---filled in only with an H for ``Hard''---is a direct sum of $\Z$ and a cyclic torsion term.  See Examples~\ref{ex:H-onefold} and \ref{ex:H-twofold}, as well as  well as Proposition~\ref{pr:Sb_to_Sc}.

Notice that all of the boxes in the table with identified entries are actually cyclic groups.  This is somewhat by luck, though.  As the length of the $\und{c}$ and $\und{d}$ sequences grows, the number of summands tends to increase.  In fact as mentioned above, the group in the box labelled [H] is not always cyclic, depending on the values of $\und{c}$ and $\und{d}$.  
As soon as one non-cyclic group appears, it propagates into more complicated degrees via the reductions outlined in Proposition~\ref{pr:irregular}.

Note as well that we have not given the orders for the cyclic groups in the table.  They can be computed, but the formulas get more and more complex.  See Examples~\ref{ex:H-onefold} and \ref{ex:H-twofold} for specifics.\vspace{0.2in}

As one final piece of information we return to the rationalization map 
\[ H^\star(\pt)\ra H^\star(\pt)\tens \Q=\Q[u_d,u_d^{-1}\,:\, d|n, d\neq 1].
\]
What is the image?  We can completely calculate it, but the answer is not easy to state.  First note that any homogeneous element in the image will be a multiple of a class $\frac{u_{d_1}\cdots u_{d_s}}{u_{c_1}\cdots u_{c_k}}$.  
Given that $u_1=1$ we can add ones to either the $\und{c}$ or $\und{d}$ sequence as needed in order to make them have the same length.  Further, the multiples in question will not change as we multiply by the $\chi$-units, and so we might as well assume that $\und{c}$ and $\und{d}$ are divisor strings.

Given divisor strings $\und{c}$ and $\und{d}$ of the same length, form the array
\[ \xymatrixrowsep{0.7pc}\xymatrix{ & d_1 & d_2 & \cdots & d_s \\
1 & c_1 & c_2 & \cdots & c_s.}
\]
Define an \dfn{allowable path} to be one that moves along the entries of the above array, starting at $1$ and having the property that each step is either right-one (R), up-right (UR), or down (D).
For each allowable path define the associated integer to be the product of all the terms encountered along the path.  

\begin{prop}
\label{pr:integral-intro}
Let $\und{c}$ and $\und{d}$ be divisor strings of the same length $s$.  Then the image of $H^{\lambda_{\und{d}}-\lambda_{\und{c}}}(\pt)$ in $H^\star(\pt)\tens \Q$ consists of all integral multiples of $M\cdot \tfrac{u_{\und{d}}}{u_{\und{c}}}$ where $M$ is given by
\[ M= \frac{(\text{gcd of all integers associated to allowable paths that end at $c_s$})}{(\text{gcd of all integers associated to allowable paths that end at either $c_s$ or $d_s$} )}.
\]
When $\und{c}$ and $\und{d}$ are not divisor strings, the answer is the same as one gets by replacing each with their associated divisor string.  
\end{prop}

See Proposition~\ref{pr:integral-classes} for this result.  We again warn that knowing a certain multiple $M\frac{u_{\und{d}}}{u_{\und{c}}}$ is in the image of the rationalization map does not necessarily give us a canonical element in $H^\star(\pt)$ that deserves to be called by this fraction.  If $H^\star(\pt)$ has torsion in this degree then multiple elements can map to the same fraction.\\

We complete this introductory survey by saying a little about $\ell$-adic localization, where $\ell$ is a prime.  To the extent possible we have tried in this introduction to describe $H^\star_{C_n}$ without working one prime at a time, and that is a reasonable first goal; as a related phenomenon, the classical group cohomology $H^*(C_n)$ is easier to describe as a global object than by patching together the $\ell$-local pieces.  But at the end of the day $\ell$-localization does help simplify matters, though perhaps not drastically so.  

The first important phenomenon one encounters is that for each $d$ there is a class $\frac{u_d}{u_{d(\ell)}}\in H^{\lambda_d-\lambda_{d(\ell)}}$ 
that becomes a unit after tensoring with $\Z_{(\ell)}$, where $d(\ell)$ is the largest power of $\ell$ that divides $d$. So after $\ell$-localization one picks up extra periodicities in the $\uH^\star$ that allow one to restrict the grading to
$\D_{d(\ell)}\subseteq \D_n$.  Note that the underlying objects are still chain complexes of $\mZ_{C_n}$-modules, though.

The second reduction involves a passage from $\mZ_{C_n}$-modules to $\mZ_{C_{n(\ell)}}$-modules.  This is most conveniently done via the change-of-group functor $\mRes_f$ where $f\colon C_n\ra C_{n(\ell)}$ is the quotient map.  This functor takes the complex $S^{\lambda_{d(\ell)}}$ as defined over $\mZ_{C_{n(\ell)}}$ to the similarly-named-but-different complex $S^{\lambda_{d(\ell)}}$ of $\mZ_{C_n}$-modules, and it also preserves box products.  Using standard properties of $\mRes_f$ and its right adjoint $\mInd_f$ one obtains canonical isomorphisms
\[ H^{\lamu{d}-\lamu{c}+k}_{C(n(\ell))} \iso H^{\lamu{d}-\lamu{c}+k}_{C_n}
\]
when all of the $c_i$ and $d_i$ are powers of $\ell$.  Putting everything together, we have that the ring $H^\star_{C_n}\tens \Z_{(\ell)}$ is essentially just $H^\star_{C_{n(\ell)}}\tens \Z_{(\ell)}$, the ``essentially'' because the former is $\D_n$-graded and the latter is $\D_{n(\ell)}$-graded.  This reduction does not, for the most part, solve the underlying problem; the computations over $C_{n(\ell)}$ are still difficult in the irregular region.  But what this reduction {\it does\/} do is reduce to a setting where everything is a divisor string, so that the $\chi$-classes are out of the story.  

For discussion of these and related matters, see Section~\ref{se:change}.\bigskip

To close this introduction, let us just say that $H^\star_{C_n}$ is a complicated structure that appears to have not yet given up all of its secrets.  One could wonder, for example, whether there is a combinatorial approach such as in Proposition~\ref{pr:integral-intro} that would give the isomorphism type of $H^{\lamu{d}-\lamu{c}}$.  There is more to be figured out here.  

\subsection{Notation and terminology}
Our results require some detailed work with the category of $\mZ_{C_n}$-modules, and unfortunately there is no universal set of conventions for describing objects and maps in this setting.  The prequel paper \cite{DHone} attempts to carefully set up the language and foundational results we need, and the reader is referred there for this basic material.  However, we give a very brief summary in Section~\ref{se:Zmod} below.

We fix once and for all an odd positive integer $n$, and $C_n$ is the cyclic group of order $n$ with a specified generator $t$.  The symbol $\ell$ always denotes a prime, and $n(\ell)$ is the largest power of $\ell$ that divides $n$.  Also $n(\hat{\ell})=\frac{n}{n(\ell)}$.  Given two integers $a$ and $b$, the gcd and lcm are denoted $(a,b)$ and $[a,b]$, respectively.

\subsection{Acknowledgments}  
Some of the work on this paper was done while the first author was a visitor at the Institute for Computational and Experimental Research in Mathematics, which is supported by NSF grant DMS-1929284.  Said author is grateful to ICERM and Brown University for providing an excellent working environment.  


\section{$\mZ$-modules and their derived category}
\label{se:derived}

In this section we briefly review the foundations of $\mZ$-modules and then discuss some basic constructions in the associated derived category, setting up the objects of study for the rest of the paper.    

\subsection{Review of $\mZ$-modules}
\label{se:Zmod}
Our preferred language and tools for working with $\mZ$-modules are described in detail in \cite{DHone}.  Here we provide an ultra-quick review.

The orbit category of $C_n$ has objects $\Theta_d=C_n/C_{(n/d)}$.  The maps are generated by the projection maps $p=p_{\Theta_d \ra \Theta_e}$ when $e|d$, as well as the action maps $t=t_d\colon \Theta_d\ra\Theta_d$. Note that we will typically drop subscripts and denote any projection by $p$ and any action map by $t$.  One always has $tp=pt$. Maps $\Theta_d\ra \Theta_e$ exist only when $e|d$, in which case the maps are precisely $t^ip$ for $i=0,1,\ldots,e-1$.  
The special projection maps $\Theta_d\ra \Theta_1$ will usually be denoted $\pi$.  

Recall $\mZ$-modules are contravariant, additive functors on a category $\cB\Z$ made from spans. See \cite[Section 2.8]{DHone} for a construction of $\cB\Z$ and discussion on how it can be viewed as a quotient of the usual Burnside category of spans. Every map $f$ in the orbit category gives maps $Rf$ and $If$ in $\cB\Z$, inducing maps $(Rf)^*=f^*$ and $(If)^*=f_*$ on a $\mZ$-module.  In this language, a span $[S \llla{f} X \llra{g} T]$ corresponds to $Rf\circ Ig$.  

For every orbit $\Theta_d$ there is a corresponding free $\mZ$-module $F_{\Theta_d}$, which we will abbreviate to just $F_d$.  We typically assume this free module comes with a chosen basis, which is a specified generator $g_d\in F_d(\Theta_d)$.  Note that $F_{\Theta_1}=F_1=\mZ$ and we will usually write it as the latter.
Maps $F_{d}\ra F_{e}$ are in bijective correspondence with $\cB\Z(\Theta_d,\Theta_e)$.  In particular, a map $f\colon \Theta_d\ra \Theta_e$ gives a map $Rf\colon F_d\ra F_e$ and a map $If\colon F_e\ra F_d$.  On basis elements these maps send $g_d\mapsto f^*(g_e)$ and $g_e\mapsto f_*(g_d)$, respectively.

For each divisor $d$ the ideal $I_d\subseteq \mZ$ is the one that is generated by the element $1_d\in \mZ(\Theta_d)$.  The module $\mZ/I_d$ has a standard free resolution of the form
\begin{equation}
\label{eq:std-free}
\xymatrixcolsep{2.5pc}\xymatrix{
0 \ar[r] & \mZ \ar[r]^{I\pi} & F_d \ar[r]^{\id-Rt} & F_d \ar[r]^{R\pi} & \mZ \ar[r] & \mZ/I_d \ar[r] & 0.
}
\end{equation}
This is described in \cite[Proposition 4.18]{DHone}.

An important property that will be used over and over is that if $e|n$ and $M$ is a $\mZ$-module that is generated in spots $\Theta_d$ for $e|d$ then $M\bbox \mZ/I_e=0$.  See \cite[Proposition 4.4]{DHone}.
In particular:
\begin{equation}
\label{eq:box-zero}
\text{If $e|d$ then $F_d\bbox \,\mZ/I_e=0$ and $I_d\bbox \mZ/I_e=0$.}  
\end{equation}

The category of $\mZ$-modules has an internal hom object $\uHom(\blank,\blank)$ that is right adjoint to the box product.  The functor $\uHom(\blank,\mZ)$ sends $F_d$ to $F_d$, but sends the map $Rf$ to $If$ and vice versa.  

We will sometimes make use of $\uTor$ and $\uExt$ calculations such as
\[ \uTor_i(\mZ/I_d,\mZ/I_e)\iso \uExt^i(\mZ/I_d,\mZ/I_e)\iso 
\begin{cases}
\mZ/I_{(d,e)} & \text{if $i=0$ or $i=3$,}\\
0 & \text{otherwise.}
\end{cases}
\]
For these and others, see \cite[Section 4]{DHone}.  Also note that there is a canonical decomposition $\mZ/I_d=\bigoplus_{\ell|d} \tors_{\ell^\bdot} (\mZ/I_{d})$ into $\ell$-power torsion submodules (where $\ell$ always denotes a prime), and non-canonical isomorphisms $\tors_{\ell^\bdot} (\mZ/I_d)\iso \mZ/I_{d(\ell)}$.  

The following is proven in \cite[Proposition 4.15(f)]{DHone}:
\begin{equation}
\label{eq:ell-local-fact}
\text{\parbox{4.5in}{Let $M$ be a finitely-generated $\mZ$-module, and let $S\subseteq M(\Theta_1)$.  Fix $d\geq 1$.  Suppose for every prime $\ell$ that $M\tens \Z_{(\ell)}\iso \mZ/I_{d(\ell)}$ and that $M\tens \Z_{(\ell)}$ is generated by $S$.  Then $M\iso \mZ/I_d$ and $M$ is generated by $S$.}}
\end{equation}

One final miscellaneous fact: 
\begin{equation}
\label{eq:Z(e;d)}
\text{\parbox{4.5in}{There is (up to isomorphism) at most one $\mZ$-module $M$ that sits in a short exact sequence $0\ra I_e\ra M \ra \mZ/I_d \ra 0$ and has $M(\Theta_1)$ non-torsion.  This module exists precisely when $d|e$, and is denoted $\mZ(e;d)$.}  }
\end{equation}
For a proof of this fact and more detailed descriptions of the module, see \cite[Section 5.4]{DHone}.

\subsection{The derived category}
Let $\cD=\cD(\mZ)$ denote the derived category of $\mZ$-modules, i.e. the localization of the category of chain complexes where one inverts all quasi-isomorphisms.  
We let $\cD_{cpt}$ be the full subcategory consisting of the bounded, levelwise-free complexes (the compact objects).  Note that these derived categories are naturally enriched over $\mZ$-modules: for objects $X$ and $Y$ we write $\ucD(X,Y)$ for the $\mZ$-module
\[ \Theta\mapsto \cD(F_\Theta\bbox X,Y).
\]

If $M$ is a $\mZ$-module and $X$ is in $\cD$ we (as usual) write
\[ \uH_i(X;M)=\ucD(\Sigma^i\mZ,X\dbox M), \qquad
\uH^i(X;M)=\ucD(X,\Sigma^i M)
\]
where $\dbox$ is the derived box product.  When $M=\mZ$ we often drop it from the coefficients.

If $X$ is a bounded, levelwise-free chain complex, let $X^*=\uHom(X,\mZ)$.  The functor $X\mapsto X^*$ preserves quasi-isomorphisms between bounded, levelwise-free chain complexes, and so may be regarded as a functor on $\cD_{cpt}$.   Observe that we have canonical isomorphisms $X\ra (X^*)^*$, and so we also have the resulting duality isomorphisms
\[ \cD(X,Y)\iso \cD(Y^*,X^*).
\]
Routine adjunction arguments show that this is the $\Theta_1$-component of a canonical isomorphism of $\mZ$-modules
\begin{equation}
\label{eq:duality}
\ucD(X,Y)\iso \ucD(Y^*,X^*).
\end{equation}

There are certain special objects in $\cD_{cpt}(\mZ)$ that come to us from topology.  For positive integers $b|n$  let $S^{\lambda_b}$ denote the complex
\[ \xymatrixcolsep{3.2pc}\xymatrix{
0 \ar[r] & F_{b} \ar[r]^{\id-Rt} & F_{b}\ar[r]^{R\pi} & \mZ \ar[r] & 0
}
\] 
where $\mZ$ is in degree $0$.  This is the reduced cellular chain complex for the 1-point compactification of the representation $\lambda_b$ (see Section~\ref{se:intro}).
That geometric sphere is also usually denoted by $S^{\lambda_b}$, but using the same symbol for both objects will usually not lead to any confusion.  

\begin{remark}
\label{re:cells}
We will never need the following, but let us briefly explain the underlying geometry here.  In the representation $\lambda_b$ draw the ray $r=\overrightarrow{01}$ as well as its conjugates $t^i r$ for $0\leq i < b$.  The geometric $S^{\lambda_b}$ has a cell structure consisting of the $0$-cells $0$ and $\infty$ (both of which are fixed and so have orbit type $\Theta_1$), the above rays as an equivariant $1$-cell of orbit type $\Theta_b$, and the sectors between the rays as an equivariant $2$-cell of type $\Theta_b$.  The boundary of the $2$-cell is $r-tr$, which is the $\id-Rt$ in the above complex.    The complex has only one $\Z$ instead of two in degree $0$ because it is the {\it reduced\/} cellular complex (the $0$-cell $\infty$ is quotiented to zero).

For $1\leq k< n$ one can similarly write down the reduced cellular chain complex for $S^{\lambda(k)}$.  The rays $t^i r$ correspond to the subgroup of $\Z/n$ generated by $k$, which is $(k,n)$.  As a $C_n$-set this is $\Theta_{\frac{n}{(k,n)}}$.  The set of two-cells that interlace between these rays also form the set $\Theta_{\frac{n}{(k,n)}}$.  
If we write $(k,n)=Ak+Bn$ with $A\geq 0$ then the boundary of the `first' $2$-cell is $r-t^Ar$, which leads to the associated reduced chain complex being the one along the top row of the following diagram:
\[ \xymatrixcolsep{4.9pc}\xymatrix{0 \ar[r] & F_{\frac{n}{(k,n)}} \ar[r]^{\id-Rt^A}  & F_{\frac{n}{(k,n)}} \ar[r]^{R\pi} & \mZ \ar[r] 
& 0 \\
0 \ar[r] & F_{\frac{n}{(k,n)}}\ar[u]^{\id} \ar[r]^{\id-Rt} & F_{\frac{n}{(k,n)}} \ar[r]^{R\pi}\ar[u]_{\id+Rt+\cdots+Rt^{A-1}} & \mZ \ar[u]_{\cdot A}\ar[r] & 0.
}
\]

We claim that the indicated map of chain complexes is a quasi-isomorphism, which explains why the $\lambda(k)$ parts of the $RO(C_n)$-grading are extraneous for $k\nmid n$.  To see this, observe that in the $F_{\frac{n}{(k,n)}}$ Mackey functor every element is fixed by $t^{\frac{n}{(k,n)}}$.  Since $1=A\cdot \frac{k}{(k,n)}+B\cdot \frac{n}{(k,n)}$ it follows that $t$ and $t^{A\cdot \frac{k}{(k,n)}}$ act the same on the Mackey functor.  It follows formally that the maps $\id-Rt$ and $\id-Rt^A$ have the same kernel and image---for the latter, use that \[\id-Rt=\id-(Rt^A)^{\frac{k}{(k,n)}}=(\id-Rt^A)\cdot (\id+Rt^A+\cdots Rt^{A\cdot (\frac{k}{(k,n)}-1)}).\]
So the homology modules of the above two complexes are the same in all degrees, and since $\uH_1=0$ the induced map is clearly an isomorphism in degrees $1$ and $2$.  In degree $0$ it is the multiplication-by-$A$ map on $\mZ/I_{\frac{n}{(k,n)}}$, but since $A$ is a unit modulo $\frac{n}{(k,n)}$ this is an isomorphism as well.

\end{remark}

Returning to our discussion of the derived category of $\mZ$,
note that $S^{\lambda_1}\he \Sigma^2\mZ$ because
the inclusion of the degree two piece gives a map $\Sigma^2\mZ\ra S^{\lambda_1}$ which is a quasi-isomorphism.
While $S^{\lambda_1}$ is not technically equal to $\Sigma^2\mZ$ as complexes, when working in the derived category we can identify them via this explicit quasi-isomorphism.  Also, let us define $S^{\lambda_0}$ to be $\mZ$ (concentrated in degree $0$).

Given positive integers $b_1,\ldots,b_s$ dividing $n$ we write
\[ S^{\lambda_{b_1}+\cdots+\lambda_{b_s}}=S^{\lambda_{b_1}}\bbox \cdots \bbox S^{\lambda_{b_s}}.
\]
Note that changing the order of the $b$'s results in a different object, though it will be the same up to isomorphism.  

We also define $S^{-\lambda_b}=(S^{\lambda_b})^*$ and
\[ S^{-\lambda_{b_1}-\cdots - \lambda_{b_s}}=S^{-\lambda_{b_1}}\bbox \cdots \bbox S^{-\lambda_{b_s}}.
\]
Observe that $S^{\lambda_{b_1}+\cdots+\lambda_{b_s}}$ is concentrated in degrees $0$ through $2s$, whereas $S^{-\lambda_{b_1}-\cdots-\lambda_{b_s}}$ is concentrated in degrees $-2s$ through $0$.  

The complex $S^{-\lambda_b}$ is given by 
\[
\xymatrixcolsep{3.2pc}\xymatrix{
0 \ar[r] &\mZ \ar[r]^{I\pi}& F_{b} \ar[r]^{\id-Rt} & F_{b}  \ar[r] & 0.
}
\]
Both of the complexes $S^{\lambda_b}$ and $S^{-\lambda_b}$ can be found inside our standard free resolution 
(\ref{eq:std-free}) of $\mZ/I_b$, a fact that can be encapsulated as the existence of the following 
cofiber sequences in $\cD(\mZ)$:

\[ S^{\lambda_b} \ra \mZ/I_b \ra \Sigma^3\mZ, \qquad
\mZ \ra \mZ/I_b \ra \Sigma^3 S^{-\lambda_b}. 
\]
These will be very useful to us.  Related to this, note the following easy computations:
\[ \uH_i(S^{\lambda_b})\iso \begin{cases}
\mZ & \text{if $i=2$}, \\
\mZ/I_b & \text{if $i=0$}, \\
0 & \text{otherwise}
\end{cases}
\qquad\text{and}\qquad
\uH_i(S^{-\lambda_{b}})\iso \begin{cases}
I_b & \text{if $i=-2$}, \\
0 & \text{otherwise.}
\end{cases}
\]
These follow immediately from the above cofiber sequences.  
In particular, observe that $S^{-\lambda_b}\he \Sigma^{-2}I_b$.  

Also note that if $a|b$ then (\ref{eq:box-zero}) says 
$F_b\bbox \mZ/I_a\iso 0$, and therefore
\begin{equation}
\label{eq:sphere-box}
S^{\lambda_b}\bbox \mZ/I_a\he \mZ/I_a \qquad\text{and}\qquad
S^{-\lambda_b}\bbox \mZ/I_a\he \mZ/I_a.
\end{equation}

Even without any motivation from topology, one would eventually discover the complexes $S^{\lambda_b}$ because they are invertible objects of $\cD(\mZ)$. In particular we have:

\begin{prop} 
\label{pr:sphere-invertible}
$S^{\lambda_b}\bbox S^{-\lambda_b}\he \mZ$.  
\end{prop}

\begin{proof}
Start with the cofiber sequence $\Sigma^{-3}\mZ \ra \Sigma^{-3} \mZ/I_b \ra S^{-\lambda_b}$ and box with $S^{\lambda_b}$ to get
\[ \Sigma^{-3}S^{\lambda_b} \ra \Sigma^{-3}\mZ/I_b \ra S^{\lambda_b}\bbox S^{-\lambda_b}
\]
(we have used (\ref{eq:sphere-box}) to simplify the middle term).
The map $S^{\lambda_b}\ra S^{\lambda_b}\bbox \mZ/I_b=\mZ/I_b$ is readily seen to be an isomorphism on $\uH_0$, and then it follows immediately from the above cofiber sequence that $\uH_i(S^{\lambda_b}\bbox S^{-\lambda_b})=0$ for $i\neq 0$ and
$\uH_0(S^{\lambda_b}\bbox S^{-\lambda_b})\iso \uH_{-1}(\Sigma^{-3}S^{\lambda_b})\iso \uH_2(S^{\lambda_b})\iso \mZ$.
This proves $S^{\lambda_b}\bbox S^{-\lambda_b}\he \mZ$, as desired.  
\end{proof}

The following result is often useful.  It says that as seen through the lens of boxing with a `large enough' free module, every sphere looks like an ordinary suspension of $\mZ$.

\begin{prop}
\label{pr:sphere-box-free}
Let $b_1,\ldots,b_s$ be divisors of $n$, set $b_0=0$, and let $m_0,\ldots,m_s\in \Z$.  Let $\beta=\sum_{i=0}^s m_i \lambda_{b_i}$.  
If $c|n$ and $b_i|c$ for all $i$, then $S^\beta \bbox F_c \he \Sigma^{\dim \beta}\mZ\,\bbox\, F_c=\Sigma^{\dim\beta}F_c$.
\end{prop}

\begin{proof}
It suffices to prove this when $\beta=\pm \lambda_{d}$ and $d|c$, and then to use the evident induction.  But $F_c$ is flat  (see \cite[Corollary 2.25]{DHone}) and so
\[ \uH_i(S^{\lambda_d}\bbox F_c)\iso \uH_i(S^{\lambda_d})\bbox F_c \iso \begin{cases}
\mZ\bbox F_c & \text{if $i=2$}, \\
\mZ/I_d\bbox F_c & \text{if $i=0$}, \\
0 & \text{otherwise.}
\end{cases}
\]
By (\ref{eq:box-zero}) the assumption $d|c$ implies that $\mZ/I_d\bbox F_c=0$, hence $S^{\lambda_d}\bbox F_c\he \Sigma^2 F_c$.  

The equivalence $S^{\lambda_d}\bbox F_c\he \Sigma^2 F_c$ immediately yields $\Sigma^{-2}F_c\he S^{-\lambda_d}\bbox F_c$, just by boxing with the inverses of the two spheres.
\end{proof}

The following result is similar to the previous one, but with a slightly different context.  Using the notation of Proposition~\ref{pr:sphere-box-free}, if $\ell$ is a prime define
\begin{equation}
\label{eq:dim_ell}
 \dim_\ell \beta=\sum_{\ell|b_i} m_i \dim \lambda_{b_i}, \qquad
\dim_{\hat{\ell}}\beta=\sum_{(\ell,b_i)=1} m_i \dim \lambda_{b_i}.
\end{equation}
Note that $\dim\beta=\dim_\ell \beta + \dim_{\hat{\ell}}\beta$.

\begin{prop}
\label{pr:sphere-box-Z/I}
Let $b_1,\ldots,b_s$ be divisors of $n$, set $b_0=0$, and let $m_0,\ldots,m_s\in \Z$.  Set $\beta=\sum_{i=0}^s m_i\lambda_{b_i}$.  For any prime $\ell$ we have $S^\beta \bbox \mZ/I_\ell\he \Sigma^{\dim_{\hat{\ell}}\beta} \mZ/I_\ell$.  
\end{prop}

\begin{proof}
As in Proposition~\ref{pr:sphere-box-free}, it suffices to prove $S^{\lambda_d}\bbox \mZ/I_\ell\he \Sigma^{\dim_{\hat{\ell}}\lambda_d} \mZ/I_\ell$ for any $d|n$.  If $\ell|d$ then $F_d\bbox \mZ/I_\ell=0$ and $\dim_{\hat{\ell}}\lambda_d=0$, therefore $S^{\lambda_d}\bbox \mZ/I_\ell \iso \mZ/I_\ell\iso \Sigma^{\dim_{\hat{\ell}}\lambda_d}\mZ/I_\ell$.
For the case $(\ell,d)=1$ start with the cofiber sequence \[\Sigma^2 \mZ \ra S^{\lambda_d}\ra \mZ/I_d.\]  Boxing with $\mZ/I_\ell$ and using that $\mZ/I_d\bbox \mZ/I_\ell=\mZ/I_{(d,\ell)}=0$ gives a quasi-isomorphism $\Sigma^2 \mZ/I_\ell \llra{\he} S^{\lambda_d}\bbox \mZ/I_\ell$, which is what we want since $\dim_{\hat{\ell}}\lambda_d=2$.
\end{proof}

\subsection{The stable homotopy ring for \mdfn{$\cD(\mZ)$}}
\label{se:D_stable}
Recall from the introduction the grading group $\D_n$ and the spheres $S^v$ for $v\in \D^n$, as well as the basic object of study
\[ H^\star=H^\star(\mZ)=\bigoplus_{v\in \D_n} \cD(\mZ,S^v).
\]
Let us briefly recall how this is a ring.
Any bounded complex of free $\mZ$-modules $A$ has canonical unit and co-unit maps \[\mZ\llra{\eta} \uHom(A,\mZ)\bbox A \quad \text{and} \quad A\bbox \uHom(A,\,Z)\llra{\hat{\eta}} \mZ\] expressing the fact that $A$ is dualizable in $\cD(\mZ)$ (these are the standard maps one has for chain complexes of dualizable objects in an abelian closed tensor category).  In particular, we have these maps for $A=S^{\lambda_d}$.  
As explained in \cite{D}, given $v,w\in \D_n$ one then gets a canonical isomorphism  $S^v\bbox S^w\iso S^{v+w}$ in $\cD(\mZ)$ by freely interchanging  factors except when they have the form $X$ and $X^*$, in which case we use the unit or co-unit to cancel them (depending on the order).  
Using these isomorphisms, the box product of maps $\mZ\ra S^v$ and $\mZ\ra S^w$ becomes a map $\mZ\ra S^{v+w}$, and it is shown in \cite{D} that this makes $H^\star$
into a $\D_n$-graded ring.   Additionally, by \cite[Proposition 1.2(3)]{D} there is a skew-commutativity rule that involves the elements $\tau_d=\tr(\id_{S^{\lambda_d}})$ and $\tau_0=\tr(\id_{S^1})$, there $\tr$ is the graded trace and takes values in $\Hom(\mZ,\mZ)=\Z$.  But observe that for each $d|n$ our chain complex model for $S^{\lambda_d}$ yields that 
\[ \tau_d=\tr(\id_{\mZ})-\tr(\id_{F_d})+\tr(\id_{F_d}) = \tr(\id_{\mZ})=1
\]
(one doesn't even need to know $\tr(\id_{F_d})$).  So the skew-commutativity rule only involves $\tau_0$ and therefore reduces to (\ref{eq:skew-comm}) as given in the introduction.  

\begin{remark}
\label{re:maps->H}
Given $v,w\in \D_n$ and a map $f\colon S^v\ra S^w$, there are two ways to get an associated element of $H^{w-v}$: one could apply $S^{-v}\bbox(\blank)$ or $(\blank)\bbox S^{-v}$.  In general one must be careful about this, and we refer to \cite[Section 6.2]{D} for a detailed discussion of the issues.  But the fact that all $\tau_d$ are equal to one implies by \cite[Proposition 6.11(a)]{D} that as long as $v$ and $w$ have fixed dimensions of the same parity the two methods give the same element of $H^\star$.  This fact, together with (\ref{eq:skew-comm}), allows us to be from now on very cavalier about most of these issues related to signs.
\end{remark}

\begin{remark}[The connections to topology]
Let $H\mZ$ denote the equivariant Eilenberg-MacLane spectrum. From the theory developed by Schwede and Shipley \cite{SS} there is a zig-zag of Quillen equivalences between Ch$(\mZ)$ and $H\mZ\MMod$, as explained in \cite[Corollary 5.2]{Z}. 
Using this equivalence, our $\uH^\star(\mZ)$ is a portion of the $RO(C_n)$-graded Bredon cohomology $\uH^\star_{C_n}(\pt;\mZ)$, where we regard $\D_n\subseteq RO(C_n)$.  By the observations in Remark~\ref{re:cells}, the two gradings have essentially the same information when $n$ is odd.   
 Thus, we often will refer to $\uH^\star(\mZ)$ as ``the cohomology of a point'', and will sometimes write it as $\uH^\star(\pt)$ or $\uH^\star(\pt;\mZ)$.  Note that when
$n$ is even there is also the sign representation in $RO(C_n)$, and we have not included that in our present development.
\end{remark}

We would like to compute the $\D_n$-graded Mackey ring $\uH^\star(\mZ)$, however this is a very complicated endeavor.  There is no known simple description of what the complete answer is.  
Instead we focus on the \dfn{regular region}, which we define to be the indices $m_0(1)+\sum_{b>1}m_b\lambda_b$ where all $m_b$ are nonnegative---called the \dfn{positive cone}---as well as the indices $m_0(1)-\sum_{b>1}m_b\lambda_b$ with all $m_b$ nonnegative and at least one nonzero---called the \dfn{negative cone}.   
The word `cone' is appropriate because of the following vanishing regions: when all $m_b\geq 0$ we have 
\begin{align*}
\uH^{m_0+\sum_{b>1}m_b\lambda_b}=0 &\quad\text{if $m_0>0$ or $m_0<-2\sum_{b>1}m_b$},\\
\uH^{m_0-\sum_{b>1}m_b\lambda_b}=0 &\quad \text{if $m_0>2\sum_{b>1}m_b$ or if $m_0<0$.}
\end{align*}
These follow at once from the fact that $S^{\sum_{b>1}m_b\lambda_b}$ is modeled by a complex concentrated in degrees $0$ through $2\sum_{b>1}m_b$.  

Let $a_{\lambda_b}$ (often just $a_b$ for short) denote the inclusion $\mZ\inc S^{\lambda_b}$ of the degree $0$ component.  As an element of $H^{\lambda_b}$ this is called the \mdfn{Euler class} of $\lambda_b$.  Since $\uH^{\lambda_b}=\uH_0(S^{\lambda_b})$ and the class $a_b$ is a generator, we have 
\begin{equation}
\label{eq:Euler}
(p_{\Theta_b\ra \Theta_1})^*(a_b)=0 \qquad\text{and}\qquad b\cdot a_b=0.
\end{equation}
These are called the \mdfn{Euler class relations}.  Note that the second relation actually follows from the first by applying $(p_{\Theta_b\ra \Theta_1})_*$.

Let $u_{\lambda_b}$ (or just $u_b$) denote the map $\Sigma^2 \mZ \ra S^{\lambda_b}$ that is $I\pi$ in degree $2$ (and zero elsewhere).  As an element of $H^{\lambda_b-2}$ this is often called an \dfn{orientation class}, as in \cite{HHR2}.
Note the cofiber sequence
\[ \Sigma^2\mZ \llra{u_b} S^{\lambda_b} \ra \mZ/I_b.
\]

\subsection{Rationalization}

Before exploring the integral story in detail, let us make some trivial observations about the rational story.  
Tensoring with $\Q$ annihilates the module $\mZ/I_b$, and therefore the map $u_b\colon \Sigma^2\mZ\ra S^{\lambda_b}$ is a rational equivalence.  Thus, we have the following:

\begin{prop}
\label{pr:rationalization}
$\uH^\star(\pt;\mQ)=\mQ[u_d, u_d^{-1}: d|n, d\neq 1]$.  Consequently, $\uH^\beta(\pt;\mZ)$ is torsion whenever $\dim \beta\neq 0$.
\end{prop}

\begin{proof}
Only the second statement requires comment.  The point is that $u_d$ lives in degree $\lambda_d-2$, which has dimension zero.  Therefore all the monomials in $u_d$ and $u_d^{-1}$ also live in degrees of dimension zero.  The groups $\uH^\beta(\pt;\mQ)$ vanish in all other degrees.
\end{proof}

Note that by \cite[Corollary 2.33]{DHone} we have
$\mQ_{C_n}\MMod \he \Q[C_n]\MMod$.
Recall that $\Q[C_n]\iso\Q[t]/(t^n-1)\iso \prod_{e|n, e<n} \Q[t]/(\Phi_e(t))$ where $\Phi_e$ is the $e$th cyclotomic polynomial and each $\Q[t]/(\Phi_e(t))$ is a field.  So the derived category $\cD(\mQ_{C_n})$ breaks up as a product of categories of graded vector spaces (over different fields), one factor for each nonmaximal divisor of $n$.  The objects $S^{\lambda_d}_\Q \he \Sigma^2 \mQ$ all live in the $e=1$ piece of this decomposition.  


\section{Initial computations}

We will spend the bulk of Section~\ref{se:point} computing the regular portion of $\uH^\star(\mZ)$.  But a basic calculation that needs to be done right off the bat is $[S^{\lambda_b},S^{\lambda_c}]$, which is the $\Theta_1$-component of $\uH^{\lambda_c-\lambda_b}$.  Although this is not in the regular region, understanding some key elements here will be a crucial part of our approach.    

The complexes are small enough that maps and chain homotopies can be constructed by hand.  Maps $S^{\lambda_b}\ra S^{\lambda_c}$ all have the form
\begin{equation}
\label{eq:chain-maps}
\xymatrix{
0 \ar[r] & F_{b} \ar[d]\ar[r]^{\id-Rt} & F_{b}\ar[r]^{R\pi}\ar[d] & \mZ \ar[r]\ar[d] & 0 \\
0 \ar[r] & F_{c} \ar[r]^{\id-Rt} & F_{c}\ar[r]^{R\pi} & \und{\Z} \ar[r] & 0 \\
}
\end{equation}
and as a precursor to constructing them we need to understand all maps $F_{b}\ra F_{c}$.  

\subsection{Constructing chain maps and homotopies}
Recall from Section~\ref{se:Zmod} that maps $F_b\ra F_c$ are in bijective correspondence with $\cB\Z(\Theta_b,\Theta_c)$.  The category $\cB\Z$ is described in detail in \cite{DHone}, and such maps are described in terms of spans $[\Theta_c \llla{f} S \llra{h} \Theta_b]$ in the orbit category of $C_n$ (and note, in particular, the right-to-left convention).  This span induces the map $F_b\ra F_c$ that sends $g_b\mapsto h_*(f^*g_c)$.
The following lemma identifies this set of maps:

\begin{lemma}
\label{le:BZ-lemma}\mbox{}\par
\begin{enumerate}[(1)]
\item If $S$ is a transitive $C_n$-set of size $d$, then any span $[\Theta_c \lla S \lra \Theta_b]$ is isomorphic to $[\Theta_c \llla{p}\Theta_d \llra{t^i\circ p} \Theta_b]$ for some $i$.
\item A span $[\Theta_c \llla{p} \Theta_d \llra{t^i\circ p} \Theta_b]$ is equal to $\tfrac{d}{[b,c]} [\Theta_c \llla{p} \Theta_{[b,c]} \llra{t^i\circ p}\Theta_b]$ in $\cB\Z(\Theta_b,\Theta_c)$.
\item 
The set $\cB\Z(\Theta_b,\Theta_c)$ is the free abelian group on the spans $[\Theta_c \llla{p} \Theta_{[b,c]} \llra{t^i\circ p} \Theta_b]$ for $i=0,1,\ldots,(b,c)-1$.  
\end{enumerate}
\end{lemma}

\begin{proof}
In (1) we can assume $S=\Theta_d$.  The maps $\Theta_d\ra \Theta_b$ and $\Theta_d\ra \Theta_c$ can both be written as $t^vp$ and $t^wp$ respectively, for some $v$ and $w$. The isomorphism $t^{-w}\colon \Theta_d\ra\Theta_d$ 
then allows us to rewrite the span into the desired form with $i=v-w$.

For (2) first note that such a span exists only if $b|d$ and $c|d$, or equivalently $[b,c]\bigl \lvert d$.  So the projection map from $\Theta_d$ to $\Theta_b$ (respectively, to $\Theta_c$) factors through $\Theta_{[b,c]}$.
One rewrites the span as the composite (right to left) of
\[ \xymatrixcolsep{1pc}\xymatrix{
& \Theta_{[b,c]} \ar[dr]^-=\ar[dl]_-{p} && \Theta_d \ar[dr]^=\ar[dl]_-p && \Theta_{d}\ar[dl]_-=\ar[dr]^-p && \Theta_{[b,c]}\ar[dr]^-{t^ip}\ar[dl]_-=\\
\Theta_c && \Theta_{[b,c]} && \Theta_d && \Theta_{[b,c]}&& \Theta_b.
}
\]
Then use that the composite of the two spans in the middle is $\tfrac{d}{[b,c]}\cdot \id$, which is a defining relation for the category $\cB\Z$.  

Moving on to (3), we can use the isomorphisms
\[ \cB\Z(\Theta_b,\Theta_c)\llra{\iso} \Hom_{\Z[C_n]}(\Z\langle \Theta_b\rangle,\Z\langle \Theta_c\rangle)\llra{\iso} \{x\in \Z\langle\Theta_c\rangle\,|\, t^bx=x\}
\]
where the second map sends a morphism to its value at the basepoint of $ \Theta_b$.  The group on the right is free (being a subgroup of $\Z\langle \Theta_c\rangle$). Write $g$ for the basepoint of $\Theta_c$, and observe if we let $A$ be the minimal sum  $g+t^bg+t^{2b}g+\cdots$ (stopping before the terms wrap around) then a basis is readily found to be
$A,tA,t^2A,\ldots$, where we again stop before the terms repeat. We just need to figure out the stopping points.
The subgroup $\langle b\rangle\subseteq \Z/c$ is equal to $\langle(b,c)\rangle$ so $A=\sum_{i=0}^{\frac{c}{(b,c)}-1} t^{(b,c)i}g$ and the generators are $A,tA,\ldots,t^{(b,c)-1}A$.  So $\cB\Z(\Theta_b,\Theta_c)\iso \Z^{(b,c)}$.   
The spans given in (3) are readily verified to map to these basis elements: precisely, $[\Theta_c\llla{p} \Theta_{[b,c]}\llra{t^ip}\Theta_b]$ maps to $t^{-i}A=t^{(b,c)-i}A$.  
\end{proof}

\begin{remark}
In the above result, there is some flexibility to where one puts the $t$ factors.  Note that $t^i\circ p=p\circ t^i$, and one can move the $t$ factors (suitably adjusted) from the left to the right map in the span.  But most of the time we will put them on the right map, as in the above lemma.   
\end{remark}

Lemma~\ref{le:BZ-lemma} tells us that any map $F_{b}\ra F_{c}$ has the form
\begin{equation}
\label{eq:free-map}
g_{b}\mapsto \sum_{i=0}^{(b,c)-1} m_i (t_*)^i p_*p^*(g_{c})
\end{equation}
for unique $m_i\in \Z$, where the first $p$ is $p_{(\Theta_{[b,c]}\ra \Theta_b)}$ and the second is $p_{(\Theta_{[b,c]}\ra \Theta_c)}$.  
Denote this map by $\lur{m}$ where $\und{m}$ is the tuple of integers $m_i$.

The map $\pi\colon \Theta_b\ra \Theta_1$ gives $R\pi\colon F_b\ra \mZ$, and the map $\pi\colon \Theta_c\ra \Theta_1$ gives $I\pi\colon \mZ\ra F_c$.  The composite $I\pi R\pi\colon F_b\ra F_c$  will appear often in our work.  In the above notation, it equals the map $\langle (1,1,\ldots,1)\rangle$.   This follows from the pullback diagram
\[ \xymatrixcolsep{3pc}\xymatrix{
\coprod_{i=0}^{(b,c)-1} \Theta_{[b,c]} \ar[d]_{\coprod_i p}\ar[r]^-{\coprod_i t^i p} & \Theta_b \ar[d]^\pi \\
\Theta_c\ar[r]_\pi & \Theta_1.
}
\]

\begin{prop}
\label{pr:lifting-lems-2}
\mbox{}\par
\begin{enumerate}[(a)]
\item The diagram
\[ \xymatrix{
& F_{d} \ar[d]^{\id-Rt} \\
F_{c}\ar[r]_{\langle \und{m}\rangle} \ar@{.>}[ur] & F_{d}
}
\]
admits a lifting as shown if and only if $\sum_i m_i=0$. 
\item A square
\[ \xymatrix{
F_{c}\ar[r]^f \ar[d]_{\id-Rt} & F_d \ar[d]^{\id-Rt} \\
F_c \ar[r]_{\langle{\und{m}}\rangle} & F_d
}
\]
commutes if and only if $f=\langle\und{m}\rangle+w \cdot I\pi R\pi$ for some $w\in \Z$.  
\end{enumerate}
\end{prop}

\begin{proof}
Both parts follow from the exact sequence 
\[ 0 \lra \mZ \llra{I\pi} F_d \llra{\id-Rt} F_d \llra{R\pi} \mZ
\]
from (\ref{eq:std-free}).
For part (a), a map $\langle\und{m}\rangle$ from $F_c$ into the degree $1$ piece of this complex lifts across $\id-Rt$ if and only if $Rp\circ \langle\und{m}\rangle$ is zero.  But this composite sends $g_c\mapsto (\sum_i m_i)\cdot 1_c$.

For (b), the commutativity is equivalent to the condition that
\[ (\id-Rt)(f-\langle \und{m}\rangle)=0 
\]
(this uses that $Rt$ commutes with everything).  Regarding $f-\langle \und{m}\rangle$ as a map into the degree two piece of the above complex, the above condition implies that it lifts across $I\pi$.  The lifting $F_c\ra \mZ$ will be a multiple of $R\pi$, so $f-\langle\und{m}\rangle$ is a multiple of $I\pi R\pi$.
\end{proof}

Now let us reconsider the construction of chain maps $S^{\lambda_b}\ra S^{\lambda_c}$ as in (\ref{eq:chain-maps}).  In degree one the map will be some $\lur{m}$, and then in degree two it must by Proposition~\ref{pr:lifting-lems-2}(b) be $\lur{m}+wI\pi R\pi$ for some $w\in \Z$.  The composite $R\pi\circ \lur{m}$ will send
\[ g_b\mapsto \sum_i m_i (t_*)^ip_*p^* g_c \mapsto \sum_i m_i (t_*)^i p_*p^* 1_c=\sum_i m_i\cdot \tfrac{[b,c]}{b}=\tfrac{[b,c]}{b}\cdot \sum_i m_i,
\]
where in the first equality we have used our knowledge of the pushforward, pullback, and action maps in $\mZ$.
So in order for the diagram to commute in degrees $0$ and $1$, the degree $0$ piece of the chain map must be multiplication by $\frac{[b,c]}{b}\sum_i m_i=\frac{c}{(b,c)}\sum_i m_i$.
Denote this map of complexes by $[\und{m},w]$.

\begin{prop}
\label{pr:Sb_to_Sc}
The map $[\und{m},w]$ is chain homotopic to zero if and only if $\sum_i m_i + (b,c)w=0$.  Consequently, we obtain an isomorphism $[S^{\lambda_b},S^{\lambda_c}]\ra \Z$ that sends $[\und{m},w]$ to $\sum_i m_i + (b,c)w$.  
\end{prop}

\begin{proof}
A null homotopy consists of maps $s_0\colon \mZ\ra F_c$ and $\lur{s_1}\colon F_b\ra F_c$.  The $s_0$ map will be $v I\pi$ for some $v\in \Z$.  The conditions for being a null homotopy are then found to be \[\tfrac{c}{(b,c)}\sum_i m_i=cv,~~ \lur{m}=vI\pi R\pi+(\id-Rt)\lur{s_1},~~ \text{and}~~\lur{m}+wI\pi R\pi=\lur{s_1}(\id-Rt).\]  Using that $\lur{s_1}(\id-Rt)=(\id-Rt)\lur{s_1}$, those terms can be eliminated from the last two equations to give $v=-w$. The first equation can also be simplified to $\sum_im_i=v(b,c)$.  This gives the implication in one direction.  For the converse, if $\sum_i m_i+(b,c)w=0$ then set $v=-w$ to get $s_0$. The map $\lur{m}-vI\pi R\pi$ then lifts through $\id-Rt$ by Proposition~\ref{pr:lifting-lems-2}(a) to give $s_1$.  
\end{proof}

Note that we will pursue more complicated analyses of this sort in Section~\ref{se:L-maps}.

\subsection{More \mdfn{$u$}-classes}
\label{se:more-u}
Let $u_{[c:b]}$ denote any map $S^{\lambda_b}\ra S^{\lambda_c}$ 
that is sent to $1$ under the isomorphism from Proposition~\ref{pr:Sb_to_Sc}.  A convenient choice is often the map $[(1,0,0,\ldots,0),0]$. This is shown below.
\[
\xymatrixcolsep{3.2pc}\xymatrix{S^{\lambda_b}:\ar[d]^{u_{[c:b]}}&
0 \ar[r] & F_{b} \ar[d]_{g_b\mapsto p_*p^*g_c} \ar[r]^{\id-Rt} & F_{b}\ar[r]^{R\pi}\ar[d]^{g_b\mapsto p_*p^*g_c} & \mZ \ar[r]\ar[d]^{\cdot \frac{[b,c]}{b}} & 0 \\
S^{\lambda_c}:&0 \ar[r] & F_{c} \ar[r]^{\id-Rt} & F_{c}\ar[r]^{R\pi} & \und{\Z} \ar[r] & 0. \\
}
\]
Observe $u_{[c:c]}$ is the identity. When $b=1$ a more convenient choice is $[(0,\ldots,0),1]$, and so then $u_{[c:1]}$ is readily seen to be the map $u_c$ defined previously.

We will freely regard $u_{[c\cln b]}$ both as a map in $\cD(S^{\lambda_b},S^{\lambda_c})$ and as an element of $H^{\lambda_c-\lambda_b}$.  See Remark~\ref{re:maps->H} for why this is unambiguous.

Here are some useful properties of these $u$-classes:

\begin{prop}\label{pr:u-props}
Let $b,c,d$ be divisors of $n$.  Then
\begin{enumerate}[(1)]
\item $u_{[c:b]}\cdot a_b = \tfrac{c}{(b,c)} a_c$,
\item $u_{[d:c]}\cdot u_{[c:b]}= \tfrac{[b,c,d]}{[b,d]}\cdot \tfrac{(b,d)}{(b,c,d)}\cdot u_{[d:b]}$,
\item $u_{[d:c]}\cdot u_c = \tfrac{c}{(c,d)}\cdot u_d$,
\item $u_{[d:c]}u_{[c:d]}=\tfrac{[c,d]}{(c,d)}$.
\end{enumerate}
\end{prop}

\begin{proof}
For (1) recall $a_b$ is just the inclusion of the degree $0$ part of the complex $S^{\lambda_b}$. The map $u_{[c:b]}$ is multiplication by $\tfrac{c}{(b,c)}$ on the degree $0$ piece, so this completes this part.  

Next observe that (3) and (4) are just the special cases of (2) when $b=1$ and $b=d$, respectively. So all that remains is to prove (2).

We consider the composite $u_{[d:c]}\circ u_{[c:b]}$, where we choose representatives for both maps that are of the form $[(1,0,\ldots,0),0]$.  The composite will then be of the form $[\und{m},0]$ for a certain $\und{m}$ that we must compute. By Proposition~\ref{pr:Sb_to_Sc} the map $[\und{m},0]$ is chain homotopic to $(\sum_i m_i)\cdot u_{[d:b]}$.  So it comes down to computing the sum of the $m_i$ integers. This will require us to carefully analyze certain pullbacks in the orbit category.

The degree $1$ part of the map $u_{[c:b]}$ sends $g_{\Theta_b}$ to $p_*p^*(g_{\Theta_c})$, where we pull back and push forward along the maps
\[ \Theta_b \llla{p} \Theta_{[b,c]} \llra{p} \Theta_{c}
\]
(note that the two $p$ maps are slightly different).  Likewise, in degree $1$ the map $u_{[d:c]}$ sends $g_{\Theta_c}$ to $p_*p^*(g_{\Theta_d})$, where the $p$ maps are different from each other and also different from the previous ones.  Consequently, the composite $u_{[d:c]}\circ u_{[c:b]}$ sends
\[ g_{\Theta_b} \mapsto p_*p^*p_*p^*(g_{\Theta_d})
\]
where the pushings and pullings are done along the bottom-to-right edge of the diagram (so not using the upper left corner)
\[\xymatrix{
\coprod_i \sigma^i \Theta_{[b,c,d]} \ar[r]^-{\amalg_i t^i p}\ar[d]_-{p} & \Theta_{[b,c]} \ar[r]^p\ar[d]^p & \Theta_b  \\
\Theta_{[c,d]} \ar[r]^p\ar[d]^p & \Theta_c \\
\Theta_d.  
}
\]
In the upper left corner we have filled in the pullback of the square, where the $\sigma^i$ are just placeholders for the different summands. To understand this pullback, observe the product $\Theta_{[c,d]}\times \Theta_{[b,c]}$ decomposes into $\Theta_{[b,c,d]}$-orbits. There are $([b,c], [c,d])/c$ such orbits, corresponding to the embeddings $\Theta_{[b,c,d]}\inc \Theta_{[c,d]}\times \Theta_{[b,c]}$ given by $x\mapsto (px,t^ipx)$ where $i$ ranges over the elements of the quotient group $c\Z/([b,c],[c,d])\Z$ (the exact values of $i$ will not matter for what we need below).

By Lemma~\ref{le:BZ-lemma}, as an element of $\cB\Z(\Theta_b,\Theta_d)$ this composite is equal to
\[ \sum_{i} \tfrac{[b,c,d]}{[b,d]} [\Theta_d \llla{p} \Theta_{[b,d]} \llra{t^ip} \Theta_b],
\]
with the $i$ values as above.
The number of terms in the sum is $([b,c],[c,d])/c$, which also equals $(b,d)/(b,c,d)$.  
From this we can directly read off the desired $m_i$, and we find that
\[ \sum_i m_i = \tfrac{[b,c,d]}{[b,d]} \cdot \tfrac{(b,d)}{(b,c,d)}.
\]
\end{proof}

\begin{remark}
Proposition~\ref{pr:u-props}(3) implies that 
under the rationalization map $H^\star(\pt;\mZ)\ra H^\star(\pt;\mQ)$ the element $u_{[b:c]}$ maps to $\tfrac{c}{(b,c)}\frac{u_b}{u_c}$.  Because $H^{\lambda_b-\lambda_c}(\pt;\mZ)$ is torsion-free by Proposition~\ref{pr:Sb_to_Sc}, rationalization is injective in this degree and we can write $u_{[b\cln c]}=\tfrac{c}{(b,c)}\tfrac{u_b}{u_c}$.  But this notation can lead to possible confusion when one starts multiplying elements, as the products might take one into degrees where rationalization is no longer injective.  For this reason we often stick with the $u_{[b\cln c]}$ notation.  
\end{remark}

The so-called ``gold relation'' in the following result was first noted (for $n$ a prime power) in papers of Hill, Hopkins, and Ravenel; see for example \cite[Section 4.2]{HHR2} and \cite{HHR3}.  For general $n$ it appears in \cite{BD}.

\begin{cor}[The gold relation]
\label{co:gold}
Let $b$ and $c$ be divisors of $n$.  Then
\[ \tfrac{b}{(b,c)} a_bu_c=\tfrac{c}{(b,c)}a_cu_b.
\]
\end{cor}

\begin{proof}
Multiply equation (1) from Proposition~\ref{pr:u-props} by $u_b$ to get $\tfrac{c}{(b,c)} a_cu_b = u_{[c:b]} a_bu_b$. Then observe
\begin{align*}
u_{[c:b]} a_bu_b = u_{[c:b]}u_ba_b 
= \tfrac{b}{(b,c)} u_ca_b =\tfrac{b}{(b,c)} a_bu_c.
\end{align*}
The third equality uses Proposition~\ref{pr:u-props}(3).
\end{proof}

Proposition~\ref{pr:Sb_to_Sc} computes the value at $\Theta_1$ of $\und{H}^{\lambda_c-\lambda_b}$, but we can actually identify the entire Mackey functor.  Recall the $\mZ$-modules $\mZ(e;d)$ from (\ref{eq:Z(e;d)}).  

\begin{prop}
\label{pr:lambda_c-lambda_b}
The $\mZ$-module $\uH^{\lambda_c-\lambda_b}$ sits in a short exact sequence
\[ 0 \lra I_b \lra \uH^{\lambda_c-\lambda_b}\lra \mZ/I_{(b,c)}\lra 0
\]
and thus is isomorphic to $\mZ(b;(b,c))$.  
\end{prop}

\begin{proof}
Take the cofiber sequence $S^{\lambda_c}\ra \mZ/I_c\ra \Sigma^3\mZ$ and box with $S^{-\lambda_b}$ to get
\[ S^{\lambda_c-\lambda_b} \ra \mZ/I_c\bbox S^{-\lambda_b} \ra \Sigma^3 S^{-\lambda_b}.
\]
Recall that $S^{-\lambda_b}\he \Sigma^{-2} I_b$, and so $\uH_i$ of the middle term is $\uTor_{i+2}(\mZ/I_c,I_b)$.  By \cite[Proposition 4.27]{DHone} this $\uTor$-module vanishes except when $i=-2$ and $0$, and for $i=0$ it is isomorphic to $\mZ/I_{(b,c)}$.  So
the long exact homology sequence gives
\[ 0 \ra I_b \ra \uH_0(S^{\lambda_c-\lambda_b}) \ra \mZ/I_{(b,c)} \ra 0.
\]
Now use
$\uH^{\lambda_c-\lambda_b}\iso \uH_0(S^{\lambda_c-\lambda_b})$.

Proposition~\ref{pr:Sb_to_Sc} says that $\uH^{\lambda_c-\lambda_b}(\Theta_1)\iso \Z$.  So (\ref{eq:Z(e;d)}) yields that $\uH^{\lambda_c-\lambda_b}\iso \mZ(b;(b,c))$.
\end{proof}

\begin{remark}
It will turn out that $\uH^{\lambda_{c_1}+\cdots+\lambda_{c_s}-\lambda_{b_1}-\cdots-\lambda_{b_s}}$ always contains $I_{[b_1,\ldots,b_s]}$ as a sub-object.  It is tempting to speculate that this $\mZ$-module is always a form of $\Z$, but this is not so.  In particular, $H^{\lambda_{c_1}+\lambda_{c_2}-\lambda_{b_1}-\lambda_{b_2}}$ can contain torsion.  See Example~\ref{ex:H-twofold}.
\end{remark}

\begin{cor}
\label{co:lambda_b-lambda_c_2}
Suppose that $b$ and $c$ are divisors of $n$ and $c|b$.  Then $S^{\lambda_c-\lambda_b}\he \mZ(b;c)$.  
\end{cor}

\begin{proof}
Start with the cofiber sequence $S^{\lambda_c} \ra \mZ/I_c \ra \Sigma^3\mZ$ and box with $S^{-\lambda_b}$.  Since $F_{b}\bbox \mZ/I_c=0$ by (\ref{eq:box-zero}), we have $\mZ/I_c\bbox S^{-\lambda_b}\he \mZ/I_c$. Hence we have the cofiber sequence
\[ S^{\lambda_c-\lambda_b} \ra \mZ/I_c \ra \Sigma^3 S^{-\lambda_b}.
\]
Since $S^{-\lambda_b}\he \Sigma^{-2}I_b$,
the right term only has homology in degree $1$. The middle term only has homology in degree $0$, and so $S^{\lambda_c-\lambda_b}$ only has homology in degree $0$. That homology is $\mZ(b;c)$ by Proposition~\ref{pr:lambda_c-lambda_b}. 
\end{proof}

\begin{cor}
\label{co:basic_equiv}
Let $b$ and $c$ be divisors of $n$.  Then 
\[ S^{\lambda_b}\bbox S^{\lambda_c} \he S^{\lambda_{(b,c)}}\bbox S^{\lambda_{[b,c]}}.
\]
\end{cor}

\begin{proof}
We have
$S^{\lambda_{b}-\lambda_{[b,c]}}\he \mZ([b,c];b)\iso \mZ(c;(b,c))\he S^{\lambda_{(b,c)}-\lambda_{c}}$
using Corollary~\ref{co:lambda_b-lambda_c_2} (twice) and \cite[Proposition 5.7]{DHone} for the middle isomorphism.
\end{proof}

\subsection{The \mdfn{$\chi$}-classes}
Let $b$ and $c$ be divisors of $n$.  Corollary~\ref{co:basic_equiv} implies that
\[ \cD(S^{\lambda_b}\bbox S^{\lambda_c},S^{\lambda_{(b,c)}}\bbox S^{\lambda_{[b,c]}})\iso \Z,
\]
and it will be useful to fix a specific generator.  
The group $\uH_4(\blank)(\Theta_1)$ of the domain and target are isomorphic to $\Z$, with generators $u_bu_c$ and $u_{(b,c)}u_{[b,c]}$ respectively, so a given equivalence will send the first to plus or minus the second.  Let us fix $\ts_{b,c}$ to be the map such that
\begin{equation} 
\label{eq:t-u}
\ts_{b,c} u_bu_c=u_{(b,c)}u_{[b,c]}.
\end{equation}
Note that $\ts_{b,c}\in H^{\lambda_{(b,c)}+\lambda_{[b,c]}-\lambda_b-\lambda_c}$ and is a unit (because the associated map is a quasi-isomorphism).  Also $\ts_{b,c}=\ts_{c,b}$ as elements of $H^\star$---one way to see this is to use that the rationalization map is injective in this degree and that the $u$-classes all commute.

The element $u_{[[b,c]:b]}u_{[(b,c):c]}$ lies in $H^{\lambda_{(b,c)}+\lambda_{[b,c]}-\lambda_b-\lambda_c}$ and is therefore an integral multiple of $\ts_{b,c}$.  The exact multiple, as well as similar formulas, are in the next result.

\begin{prop} Let $b|n$ and $c|n$. Then  
\begin{enumerate}[(1)]
\item $u_{[[b,c]:b]}u_{[(b,c):c]}=\tfrac{c}{(b,c)}\cdot \ts_{b,c}$, and
\item $u_{[b:[b,c]]}u_{[c:(b,c)]}=\tfrac{c}{(b,c)}\cdot \ts_{b,c}^{-1}$.
\end{enumerate}
\end{prop}

\begin{proof}
For (1) it suffices to compute that
$u_{[[b,c]:b]}u_{[(b,c):c]}\cdot u_bu_c=\tfrac{c}{(b,c)}u_{[b,c]}u_{(b,c)}$ using Proposition~\ref{pr:u-props}(3).  The proof of (2) is similar.  
\end{proof}

It will be helpful to also know how $\ts_{b,c}$ interacts with the other $au$-classes.

\begin{prop}Let $b|n$ and $c|n$. Then  
\begin{enumerate}[(1)]
\item $\ts_{b,c}a_bu_c=\tfrac{c}{(b,c)} a_{[b,c]}u_{(b,c)}$,
\item $\ts_{b,c}a_ba_c=a_{[b,c]}a_{(b,c)}$.
\end{enumerate}
\end{prop}

\begin{proof}
We will give the argument for (1), and the one for (2) follows the same technique.  We start with
\begin{align*}
\tfrac{c}{(b,c)}\ts_{b,c}a_bu_c=u_{[[b,c]:b]}u_{[(b,c):c]}a_bu_c = \tfrac{[b,c]}{b}a_{[b,c]}\cdot \tfrac{c}{(b,c)}u_{(b,c)} = \tfrac{c^2}{(b,c)^2}a_{[b,c]}u_{(b,c)}
\end{align*}
where we are using Proposition~\ref{pr:u-props} in the second equality.   Likewise
\begin{align*}
\tfrac{b}{(b,c)}\ts_{b,c}a_bu_c=u_{[[b,c]:c]}u_{[(b,c):b]}a_bu_c = a_{(b,c)}\cdot u_{[b,c]}= \tfrac{[b,c]}{(b,c)}a_{[b,c]}u_{(b,c)} = \tfrac{bc}{(b,c)^2} a_{[b,c]}u_{(b,c)}
\end{align*}
where here we have used Proposition~\ref{pr:u-props} and the gold relation.  So these two equations show that the expression
\[ \ts_{b,c}a_bu_c-\tfrac{c}{(b,c)}a_{[b,c]}u_{(b,c)}
\] is annihilated by both $\tfrac{c}{(b,c)}$ and $\tfrac{b}{(b,c)}$.  Since these two integers are relatively prime, the expression must be zero.  
\end{proof}


\section{The cohomology of a point}
\label{se:point}

Our goal in this section is to show how the computations for $\uH^\star(\pt)$ in the regular region follow very simply from the homological foundations that have been previously developed.  There are two separate issues: computing the underlying $\mZ$-modules, and computing the ring structure.  The former is very easy, and the latter is reasonably simple for the positive cone. Computing the ring structure in the region of the negative cone is more dicey, as we will see.

\medskip

\subsection{Prelude: The lcm-gcd sequence}

The results of our computations will involve a slightly strange collection of numbers, that we introduce now:

\begin{defn} Let $b_1,\dots, b_s$ be a list of divisors of $n$. Write $\und{b}$ for the tuple $(b_1,\dots, b_s)$. For $1\leq j \leq s$ define $(\und{b};j)$ to be the greatest common divisor of all possible $j$-fold least common multiples of $b_1,\dots, b_s$. That is,
\[(\und{b};j)=\gcd\bigl \{\text{\emph{lcm}}(b_{i_1},\dots, b_{i_j})\mid 1\leq i_1<\cdots <i_j \leq s\bigr \}.\] 
\end{defn}
Observe $(\und{b};1)=(b_1,\dots,b_s)$ and $(\und{b};s)=[b_1,\dots, b_s]$. We also have $(\und{b};j)\mid (\und{b};j+1)$ for all $j$. Thus for a tuple of divisors $\und{b}=(b_1,\dots, b_s)$, the numbers $(\und{b};j)$ give a sequence of $s$ divisors of $n$ that starts with the gcd, ends with the lcm, and has the property that $(\und{b};j)\mid (\und{b};j+1)$.

For divisor tuples $\und{b}$ and $\und{b}'$ write $\und{b}\subseteq \und{b}'$ if each number that appears in $\und{b}$ also appears somewhere in $\und{b}'$. We have that if $\und{b}\subseteq \und{b}'$ then $(\und{b}';j)\mid (\und{b};j)$.

\begin{example}
As a concrete example take $n=180$ and
$b_1=15$, $b_2=9$, $b_3=18$. Then
\[
(\und{b};1)=(15,9,18)=3, \quad (\und{b};2)=([15,9],[15,18],[9,18])=(45,90,18)=9, \text{ and}\]
\[
(\und{b};3)=[15,9,18]=90.
\]
\end{example}

Define a \dfn{divisor string} of $n$ to be a sequence of positive integers $b_1,\ldots,b_s$ such that $b_i|b_{i+1}$ for $1\leq i\leq s-1$ and $b_s|n$.  
Note that for a divisor string one has
$(\und{b};j)=b_j$, for all $j$. This simple case will appear often in our computations.

Recall that when $\ell$ is prime $b(\ell)$ denotes the $\ell$-adic part of a positive integer $b$. For a tuple $\und{b}$ we will similarly write $\und{b}(\ell)$ for the tuple given by replacing each entry with its $\ell$-adic part. Taking the $\ell$-adic part commutes with gcd and lcm, and so one has 
\begin{equation}
\label{eq:ell-adic}
(\und{b};j)(\ell)=(\und{b}(\ell);j)
\end{equation}
for any collection $\und{b}$.  
Note that $\und{b}(\ell)$ will always be a divisor string after a suitable re-ordering.

\begin{remark}
In terms of prime factorizations, the gcd is given by taking the minimal exponent that appears on each prime while the lcm is given by taking the maximum exponent that appears on each prime. From (\ref{eq:ell-adic}) it readily follows that $(\und{b};j)$ is found by taking the $j$th smallest exponent for each prime.  This leads to the following interpretation:
\end{remark}

\begin{prop}
\label{pr:lcm-gcd-distill}
Consider the operation that replaces two integers with their gcd and lcm: $x,y\mapsto (x,y),[x,y]$.  Note that this replaces the original pair with a divisor string.  Applying this operation repeatedly (pairwise) to a sequence of positive integers $b_1,\ldots,b_s$, it is always possible to eventually obtain a divisor string; and the divisor string one obtains is always the same, namely $(\und{b};1),(\und{b};2),\ldots,(\und{b};s)$.
\end{prop}

\begin{proof}
Left to the reader.  
\end{proof}

It is sometimes convenient to know the following alternative characterization:

\begin{prop}
Let $b_1,\ldots,b_s$ be divisors of $n$.  Then $(\und{b};j)$ is the lcm of all-possible $(s+1-j)$-fold gcds of the elements of the sequence; that is,
\[ (\und{b};j)=\lcm \bigl\{ \gcd(b_{i_1},\ldots,b_{i_{s+1-j}}) \,|\, 1\leq i_1 < \cdots <i_{s+1-j}\leq s   \bigr \}
\]
\end{prop}

\begin{proof}
Left to the reader.
\end{proof}

These lcm-gcd sequences enter our work due to the following fundamental result:

\begin{prop}
\label{pr:lcm-gcd-spheres}
Let $b_1,\ldots,b_s$ be divisors of $n$.  Then there is a quasi-isomorphism
\[ S^{\lambda_{b_1}+\cdots+\lambda_{b_s}}\he S^{\lambda_{(\und{b};1)}+\cdots+\lambda_{(\und{b};s)}}.
\]
\end{prop}

\begin{proof}
Immediate from Proposition~\ref{pr:lcm-gcd-distill} and Corollary~\ref{co:basic_equiv}.
\end{proof}

\begin{remark}
For another (and in some ways more direct) explanation of why the numbers $(\und{b};j)$ arise in our calculations, see the proof of Proposition~\ref{pr:Q-Mackey} below.  These integers arise as a purely algebraic consequence of the gold relations together with the Euler class torsion relations.  
\end{remark}

\subsection{The positive cone}
Computing the positive cone of $\uH^\star(\pt)$ amounts to computing the homology modules $\uH_*(S^{\lambda_{b_1}}\bbox \cdots \bbox S^{\lambda_{b_k}})$ for any choice of divisors $b_1,\ldots,b_k$ of $n$.  One can tackle this inductively via the K\"unneth spectral sequence, but there are a host of nontrivial differentials that need to be detected.  We will
instead adopt a totally different approach using carefully chosen cofiber sequences, but this only works when the $b_i$'s form a divisor string.  The final twist is to use $\ell$-localization to reduce to that case.  \smallskip

To start suppose that $b$ and $\ell$ are divisors of $n$, where $\ell$ is prime.  
Recall the elements $u_{[b:b(\ell)]}\in H^{\lambda_b-\lambda_{b(\ell)}}$
and $u_{ [b(\ell):b] }\in H^{\lambda_{b(\ell)}-\lambda_b}$ constructed in Section~\ref{se:derived}.  Proposition~\ref{pr:u-props}(4) gives that their product is $\tfrac{b}{b(\ell)}$, which is a unit after $\ell$-localization.  The following result is immediate:

\begin{prop}
\label{pr:ell-local-1}
Let $b_1,\ldots,b_k$ be divisors of $n$.
Then the map $\bbox_i u_{[b_i:b_i(\ell)]}$ gives an $\ell$-local quasi-isomorphism 
$S^{\lambda_{b_1{(\ell)}}+\cdots + \lambda_{b_k{(\ell)}}} \ra
S^{\lambda_{b_1}+\cdots + \lambda_{b_k}}$.
Consequently, 
multiplication by $\prod_i u_{[b_i:b_i(\ell)]}$ gives an
$\ell$-isomorphism
\[ \uH^{\lambda_{b_1{(\ell)}}+\cdots + \lambda_{b_k{(\ell)}}} \ra
\uH^{\lambda_{b_1}+\cdots + \lambda_{b_k}}.
\]
\end{prop}

In this way we can reduce many computations to the case where the chosen  divisors of $n$  are all powers of a fixed prime $\ell$.  \smallskip

Let us say that a complex is \dfn{even} if its homology is concentrated in even degrees.  As a demonstration of the above methods, we now prove the following:

\begin{prop}
For any divisors $b_1,\ldots,b_k$ of $n$, the complex $S^{\lambda_{b_1}}\bbox \cdots \bbox S^{\lambda_{b_k}}$ is even.   
\end{prop}

\begin{proof} Fix a pair of divisors $b,c$.
Recall the cofiber sequence $\Sigma^2 \mZ \ra S^{\lambda_c}\ra \mZ/I_c$.  Boxing with $S^{\lambda_b}$ gives a cofiber sequence
\[ \Sigma^2 S^{\lambda_b} \ra S^{\lambda_b}\bbox S^{\lambda_c} \ra S^{\lambda_b}\bbox \mZ/I_c.
\]
If $c\mid b$ then $F_b\bbox \mZ/I_c=0$ by (\ref{eq:box-zero}) and therefore $S^{\lambda_b}\bbox \mZ/I_c\he \mZ/I_c$. Since $S^{\lambda_b}$ is even, it follows that $S^{\lambda_b}\bbox S^{\lambda_c}$ is also even in the case $c\mid b$.  The same argument shows by induction that if $\und{b}$ is a divisor string then $S^{\lambda_{b_1}}\bbox \cdots \bbox S^{\lambda_{b_k}}$ is even.  

In particular, if $\ell$ is a prime and we fix exponents $1\leq e_1\leq \cdots \leq e_k$, then the preceding paragraph shows that  $S^{\lambda_{\ell^{e_1}}}\bbox \cdots \bbox S^{\lambda_{\ell^{e_k}}}$ is even.

For arbitrary $b_1,\ldots,b_k$, it suffices to prove that $S^{\lambda_{b_1}}\bbox \cdots \bbox S^{\lambda_{b_k}}$ is even after $\ell$-localization, for every prime $\ell$.  But after $\ell$-localization we can replace each $b_i$ by $b_i{(\ell)}$ (Proposition~\ref{pr:ell-local-1}), and here we already know the result.  
\end{proof}

Recall the classes $a_{b_i}\in H_0(S^{\lambda_{b_i}})$ and $u_{b_i}\in H_2(S^{\lambda_{b_i}})$, for each $i$.  Taking products of such classes gives elements in $H_*(S^{\lambda_{b_1}}\bbox \cdots \bbox S^{\lambda_{b_k}})$.  Any such product will be referred to as an \mdfn{$au$-class}.  
Note that different $au$-classes can live in the same degree, e.g. $a_{b_1}u_{b_2}$ and $u_{b_1}a_{b_2}$. The first paragraph of the above proof gives much more than what was recorded.  It immediately yields the following:

\begin{prop}
\label{pr:l-case}
Suppose that $b_1,\ldots,b_{k}$ is a divisor string.  
Then for $0\leq i\leq k$ one has
\[ \uH_{2i}(S^{\lambda_{b_1}+\cdots+\lambda_{b_{k}}}) \iso \begin{cases} \mZ/I_{b_{i+1}} & \text{if $i\leq k-1$},\\
\mZ & \text{if $i=k$}
\end{cases}
\]
and the homology vanishes in all other degrees.  
Moreover, in each degree the given Mackey functor is generated by the collection of $au$-classes (lying in the $\Theta_1$ spot).  
\end{prop}

\begin{proof}
Left to the reader.
\end{proof}

\begin{cor}
\label{co:positive-cone-computation1}
Let $b_1,\ldots,b_k$ be divisors of $n$.  Then 
\[ \uH_{2i}(S^{\lambda_{b_1}}\bbox \cdots \bbox S^{\lambda_{b_k}})\iso
\begin{cases}
\mZ/I_{(\und{b};i+1)} & \text{if $i<k$}, \\
\mZ & \text{if $i=k$},
\end{cases}
\]
and all other homology Mackey functors are zero.  Moreover, in each degree the Mackey functor is generated by the $au$-classes (lying in the $\Theta_1$-spot).
\end{cor}

\begin{proof}
\cite[Proposition 4.15]{DHone} says that in order to prove a $\mZ$-module $M$ is isomorphic to $\mZ/I_d$ with set of generators $S\subseteq M(\Theta_1)$, it is sufficient to prove this after $\ell$-localization for every prime $\ell$ (with $d$ replaced by $d(\ell)$).  We will take this approach, considering the set of all $au$-classes in a given degree and making use of
(\ref{eq:ell-adic}).

Write $c_i=b_i(\ell)$.  Multiplication by $\prod_i u_{[b_i:c_i]}$ is an 
$\ell$-isomorphism, so it suffices to understand the homology of $S^{\lambda_{c_1}}\bbox \cdots \bbox S^{\lambda_{c_k}}$.  But here the $c_i$ are all powers of $\ell$, and so can be arranged in an order where each divides the next.  The result for this case is by Proposition~\ref{pr:l-case}.

There is one subtlety, which is that multiplication by $u_{[b_i:c_i]}$ sends $u_{c_i}$ to $u_{b_i}$ but  sends $a_{c_i}$ to an integer multiple of $a_{b_i}$ 
(see Proposition~\ref{pr:u-props}).  This integer factor is prime to $\ell$, though.  So the fact that the  homology of $S^{\lambda_{c_1}}\bbox\cdots \bbox S^{\lambda_{c_{k}}}$ is generated by the $au$-classes immediately implies the same for $S^{\lambda_{b_1}}\bbox \cdots \bbox S^{\lambda_{b_k}}$, after $\ell$-localization.  
\end{proof}

\begin{remark}
Note that the above result proves that $H_{2i}(S^{\lambda_{b_1}+\cdots+\lambda_{b_s}})$ is a cyclic group, but does not identify a specific generator.  The subset of monomials in the $a$'s and $u$'s collectively generate the group, but identifying a specific generator can be a little tricky.  In particular, it can require a linear combination of the $au$-monomials.
\end{remark}

\subsection{Computations for the negative cone}

The negative cone corresponds to the $\mZ$-modules 
\[ \uH_*(S^{-\lambda_{b_1}}\bbox\cdots \bbox S^{-\lambda_{b_k}}) \iso \uH_*( \uHom(S^{\lambda_{b_1}}\bbox \cdots \bbox S^{\lambda_{b_k}},\mZ)).
\]
This turns out to be an easy computation via the universal coefficient spectral sequence.

\begin{prop}
\label{pr:negative-cone}
\[ \uH^{i-\lambda_{b_1}-\cdots-\lambda_{b_k}}\iso\uH_{-i}(S^{-\lambda_{b_1}-\cdots-\lambda_{b_k}})\iso 
\begin{cases}
\mZ/I_{(\und{b};j)} &\text{if $i=2j+1$, $1\leq j \leq k-1$,}\\
I_{[\und{b}]} & \text{if $i=2k$},\\
0 & \text{otherwise.}
\end{cases}
\]
\end{prop}

\begin{proof}

We start by observing that our standard model for $S^{-\lambda_{b_i}}$ is concentrated in degrees $-2$ through $0$, and so $S^{-\lambda_{b_1}}\bbox\cdots \bbox S^{-\lambda_{b_k}}$ will be concentrated in degrees $-2k$ through $0$.  In particular, the homology outside of this range must vanish. 

If $A_*$ is a chain complex of $\mZ$-modules then the universal coefficient spectral sequence for $\uH_*(\uHom(A,\mZ))$ has $E_2$-term equal to $\uExt^{-p}(\uH_{-q}(A),\mZ)$ in the $(p,q)$ spot with differentials $d_r$ mapping $(p,q)$ to $(p-r,q+r-1)$ (using this indexing we have a left half-plane spectral sequence). See \cite[Section 6]{DHone}. 
Apply this to $A=S^{\lambda_{b_1}+\cdots+\lambda_{b_k}}$, using the homology computation from 
Corollary~\ref{co:positive-cone-computation1}.  The relevant $\uExt$-computations are from \cite[Proposition 4.23]{DHone}: $\uExt^i(\mZ/I_d,\mZ)$ is isomorphic to $\mZ/I_d$ when $i=3$, and zero otherwise. So
the $E_2$-term looks as in Figure~\ref{fig:Ext_spseq}.

\begin{figure}[ht]
\begin{tikzpicture}[scale=0.9]
\draw[->, thick] (0,4)--(9,4);
\draw[->, thick] (8,-1)--(8,4.5);
\foreach \y in {-1,0,1,2,3,4}
	{\draw (0,\y)--(9,\y);}
\foreach \y in {2,4,6,8}
	{
	\draw (\y,-1)--(\y,4);
	};
\draw (3.1,3.5) node{ $\mZ/I_{(\und{b};1)}$ };
\draw (3.1,1.5) node{ $\mZ/I_{(\und{b};2)}$ };
\draw (3.1,0.5) node{$\vdots$};
\draw (9.1,1.5) node{ $\scriptstyle{-2}$};
\draw (9.1,2.5) node{ $\scriptstyle{-1}$};
\draw (9.2,3.5) node{ $\scriptstyle{0}$};
\draw (9.2,0.59) node {$\scriptstyle{\vdots}$};
\draw (9.3,4) node {$\scriptstyle{p}$};
\draw (7.7,4.4) node {$\scriptstyle{q}$};
\draw (9,-0.6) node { $\mZ$ $\scriptstyle{(-2k)}$};
\end{tikzpicture}
\caption{The $E_2$-term $\uExt^{-p}(\uH_{-q}(S^{\lamu{b}}),\mZ)$}
\label{fig:Ext_spseq}
\end{figure}

There is a single possible differential, namely a $d_3$ from the $\mZ$ term to the $\mZ/I_{(b;k)}$ term.  Note that the $\mZ/I_{(\und{b};k)}$ term contributes to $\uH_{*}$ for $*=3+(2k-2)=2k+1$, which we know is zero by the first paragraph.  The only way this contribution can disappear is for the differential $\mZ\ra \mZ/I_{(\und{b};k)}$ to be a surjection, in which case the kernel is precisely $I_{(\und{b};k)}$.  This completes the calculation after noting that $(\und{b};k)$ is just the lcm $[\und{b}]$.
\end{proof}

\subsection{The ring structure on the positive cone}

At this point we know the $\mZ$-modules that make up the positive cone, and we know explicit generators---namely, the $au$-classes.  We also know some basic relations between the $au$-classes, namely the gold relation and the torsion relations satisfied by the $a$'s.  It essentially follows by algebra that these must be the only relations---if there were any others, we would get $\mZ$-modules smaller than what we actually have.  In this section we trace through the details of this argument.

\begin{defn}
\label{de:Q}
Let $b_1,\ldots,b_t$ be fixed, distinct divisors of $n$ with $b_i>1$ for all $i$. Let ${N}$ be the free abelian group on the ordered set $\{1,\lambda_{b_1},\ldots,\lambda_{b_t}\}$, so that elements of ${N}$ are identified with $(t+1)$-tuples of integers. 
Let $Q_n(\und{b})$ be the ${N}$-graded $\mZ$-algebra described as follows: it is the quotient of the polynomial ring $\mZ[a_1,\ldots,a_t,u_1,\ldots,u_t]$, where the $a$'s and $u$'s are generators in the $\Theta_1$ spot with $|a_j|=\lambda_{b_j}$ and $|u_j|=\lambda_{b_j}-2$, subject to the relations
\[ \xymatrixrowsep{0.4pc}\xymatrix{
\bigr (p_{\Theta_{b_j}\ra \Theta_1} \bigl)^*(a_j)=0 & (\text{Euler class relation}),\\
 \tfrac{b_i}{(b_i,b_j)} a_iu_j = \tfrac{b_j}{(b_i,b_j)} a_ju_i & (\text{gold relation})
}
\]
for all $i,j$.  
\end{defn}

Note that applying the transfer map $(p_{\Theta_{b_j}\ra \Theta_{1}})_*$ to the first relation yields $b_ja_j=0$.  In fact, observe that the Euler class relation is the only relation that affects degree $\lambda_{b_j}$.  So the component of $Q_n(\und{b})$ in degree $\lambda_{b_j}$ is the quotient of $\mZ$ (with generator $a_j$) by the relation $(p_{\Theta_j\ra \Theta_1})^*(a_j)=0$, which is precisely $\mZ/I_{b_j}$.

\begin{prop}
\label{pr:Q-Mackey}
Fix $n$ and $\und{b}$ as above, and write $Q=Q_n(\und{b})$.  Then for $\und{d}\in N$ the $\mZ$-module $Q^{\und{d}}$ is generated by the products of $a$'s and $u$'s that lie in degree $\und{d}$.  Moreover $Q^{\und{d}}$ is nonzero unless $\und{d}=-2r(1)+\sum_{i>0} e_i\lambda_{b_i}$ where $e_i\geq 0$ for all $i$, and $0\leq r\leq \sum_i e_i$.  
In this case if we let $B$ denote the sequence consisting of $b_1$ repeated $e_1$ times, followed by $b_2$ repeated $e_2$ times, and so forth, then one has
\[ Q^{\und{d}}\iso \begin{cases}
\mZ & \text{if $r=\sum_{i>0}e_i$},\\
\mZ/I_{(B;r+1)} & \text{otherwise}.
\end{cases}
\]
\end{prop}

\begin{remark}
The  above result is hard to parse, but we can rephrase it as follows.  The components $Q^{\und{d}}$ are nonzero only in degrees that contain a monomial $a_{i_1}\cdots a_{i_s} u_{i_{s+1}}\cdots u_{i_{s+r}}$ (where the indices are not required to be distinct).  The $\mZ$-module in this degree is isomorphic to $\mZ$ if $s=0$, and otherwise is isomorphic to $\mZ/I_{(B;r+1)}$ where $B=(b_{i_1},b_{i_2},\ldots,b_{i_{s+r}})$.  
\end{remark}

For the proof we will need the following easy lemma:

\begin{lemma}
\label{le:gcd-vanishing}
Let $M$ be a $\mZ$-module, and let $e_1,\ldots,e_s$ be divisors of $n$.  Assume $x\in M(\Theta_1)$ has the property that the restriction of $x$ to each $M(\Theta_{e_i})$ equals zero.  Then the restriction of $x$ to $M(\Theta_{(e_1,\ldots,e_s)})$ is also equal to zero.
\end{lemma}

\begin{proof}
Let $y=(p_{\Theta_{(e_1,\ldots,e_s)}\ra \Theta_1})^*(x)$ and set $z_i=(p_{\Theta_{e_i}\ra \Theta_{(e_1,\ldots,e_s)}})^*(y)$.  Our assumption is that $z_i=0$, and therefore
\[ 0 = (p_{\Theta_{e_i}\ra \Theta_{(e_1,\ldots,e_s)}})_*(z_i)=\tfrac{e_i}{(e_1,\ldots,e_s)}\cdot y.
\]
But the integers $\tfrac{e_i}{(e_1,\ldots,e_s)}$ are relatively prime, and so there is a linear combination of them that equals $1$.  Hence $y=0$.  
\end{proof}

\begin{proof}[Proof of Proposition~\ref{pr:Q-Mackey}]
By an ``$au$-monomial'' we simply mean a monomial in the $a$'s and $u$'s. It is clear that $Q$ is zero except in degrees that contain an $au$-monomial.  It is also clear that the degree containing a $u$-monomial $u_{i_1}\cdots u_{i_t}$ does not contain any other $au$-monomials, and hence the component in this degree is equal to $\mZ$.  

Let $\und{d}$ be the degree containing the monomial $m=a_{i_1}\cdots a_{i_s}u_{i_{s+1}}\cdots u_{i_{s+r}}$.  Clearly $Q^{\und{d}}$ is generated (as a $\mZ$-module) by all of the $au$-monomials that lie in this degree---these are precisely the ones obtained from $m$ by permuting the indices.  Consequently, the restriction maps in $Q^{\und{d}}$ are all surjections.  

Let $e=(B;r+1)$. Our goal is to show $Q^{\und{d}}\cong \mZ/I_e$. Observe the monomial $m$ restricts to zero in each of the spots $\Theta_{b_{i_1}},\ldots,\Theta_{b_{i_s}}$, because the respective $a$-classes do.  Any $(r+1)$-fold lcm $\Lambda$ of the integers $b_{i_1},\ldots,b_{i_{s+r}}$ will be a multiple of at least one of $b_{i_1},\ldots,b_{i_s}$, and hence $m$ will restrict to zero in spot $\Theta_{\Lambda}$.  Letting $\Lambda$ range over all of the $(r+1)$-fold lcms, Lemma~\ref{le:gcd-vanishing} shows that $m$ restricts to zero in spot $\Theta_e$.  The same analysis applies to any monomial obtained from $m$ by permuting the indices, and so all of those $au$-monomials restrict to zero at $\Theta_e$.  This proves that $Q^{\und{d}}(\Theta_e)=0$.  

We will prove that $Q^{\und{d}}$ is of the desired form using (\ref{eq:ell-local-fact}), which reduces matters to the $\ell$-local case.  So fix a prime $\ell$.  
We use the recognition criteria from \cite[Proposition 4.8]{DHone}: this says that $\mZ/I_{e(\ell)}$ is (up to isomorphism) the unique $\mZ$-module $M$ having the properties that $M(\Theta_1)\iso \Z/e(\ell)$, $M(\Theta_{e(\ell)})=0$, and all of the restriction maps in $M$ are surjections.  The last part has already been verified above.  For the second part, observe that
$p_*p^*\colon Q^{\und{d}}(\Theta_{e(\ell)})\ra Q^{\und{d}}(\Theta_e) \ra Q^{\und{d}}(\Theta_{e(\ell)})$ is multiplication by $e/e(\ell)$ and therefore an $\ell$-isomorphism, whereas the middle group has been shown to vanish.  So $Q^{\und{d}}_{(\ell)}(\Theta_{e(\ell)})=0$.

It remains to prove the claim about $Q^{\und{d}}_{(\ell)}(\Theta_1)$, which we know is generated by the $au$-monomials.
After $\ell$-localization the gold relations become
\[ \tfrac{b_i(\ell)}{\min\{b_i(\ell),b_j(\ell)\}}\cdot (\text{a unit in $\Z_{(\ell)}$})\cdot a_iu_j =
\tfrac{b_j(\ell)}{\min\{b_i(\ell),b_j(\ell)\}}\cdot (\text{a unit in $\Z_{(\ell)}$})\cdot a_ju_i.
\]
The first term on either the left or the right is equal to one, and so the relation says that one of $a_iu_j$ and $a_ju_i$ is a multiple of the other (after $\ell$-localization): precisely, if $b_i(\ell)\leq b_j(\ell)$ then $a_iu_j$ is a multiple of $a_ju_i$ (and vice versa for the other direction).  So the more primitive $a_?\cdot u_?$ product is the one with the index on the $u$ corresponding to the $b$ with the smaller $\ell$-adic value.      From this it follows at once that all of the $au$-monomials in degree $\und{d}$ are multiples of a single one, namely any monomial where the indices on the $u$'s correspond to the $b$'s with the smallest $\ell$-adic values.  The additive order of this class is then equal to the $(r+1)$st smallest value of the $b_i(\ell)$'s, which is precisely $e(\ell)$.  
So $Q^{\und{d}}(\Theta_1)$ is cyclic of this order (recall that this is all after $\ell$-localization). 

\cite[Proposition 4.8]{DHone} now shows that $Q_{(\ell)}^{\und{d}}\iso \mZ/I_{e(\ell)}$, and the desired result then follows immediately from (\ref{eq:ell-local-fact}).
\end{proof}

\begin{cor}
Let $b_1,\ldots,b_t$ be the collection of all divisors of $n$ that are larger than $1$.  Let $N$ and $Q_n(\und{b})$ be as in Definition~\ref{de:Q}.  
There is a unique map of $N$-graded $\mZ$-algebras $Q_n(\und{b})\ra \uH^\star(\pt)$ sending the formal $a$- and $u$-classes to the analogs in $H^\star$, and this map is an isomorphism for all degrees in $N$ such that the multiplicity of each $\lambda_b$ is non-negative for $b>0$.
\end{cor}

\begin{proof}
The existence and uniqueness of the map follows from the fact that the $a$- and $u$-classes in $H^\star$ satisfy the defining relations for $Q_n(\und{b})$: see (\ref{eq:Euler}) and Corollary~\ref{co:gold}.
The map is surjective in all termwise non-negative degrees because we have seen in Corollary~\ref{co:positive-cone-computation1} that in these degrees the positive cone is generated by the $au$-monomials.  

By Proposition~\ref{pr:Q-Mackey} and Corollary~\ref{co:positive-cone-computation1}, in each degree that we are considering the domain and codomain of our map belong to the set $\{\mZ,\mZ/I_d\}_{d|n}$ and are isomorphic.  For such modules, any surjective map between objects of the same isomorphism class must be an isomorphism (this is because the analogous property holds for cyclic abelian groups).
\end{proof}


\section{Small models for spheres}\label{se:smallmodels}

Fix $b_1,\ldots,b_s$, divisors of $n$.  The sphere $S^{\lambda_{b_1}+\cdots+\lambda_{b_k}}$ is modeled by the box product $S^{\lambda_{b_1}}\bbox \cdots \bbox S^{\lambda_{b_s}}$, which is a chain complex whose terms consist of sums of $3^s$ free $\mZ$-modules, ranging in homological degrees from $0$ through $2s$.  In this section we produce a much smaller model that consists of just $2s+1$ free summands, one in each of those degrees.
\vspace{0.2in}

Let $b_1,\ldots,b_s$ be a divisor string.
Define $L(\und{b})$ to be the following chain complex:
\[ \xymatrixcolsep{1.6pc}\xymatrix{
0 \ar[r] &
F_{b_s} \ar[r]^-{\id-Rt} & F_{b_s} \ar[r]^-{\Delta_s} &
F_{b_{s-1}} \ar[r]^-{\id-Rt} & F_{b_{s-1}} \ar[r]^-{\Delta_{s-1}} & \cdots \ar[r]^-{\id-Rt} & F_{b_1} \ar[r]^-{R\pi} & \mZ \ar[r] & 0.
}
\]
Here the nonzero terms are concentrated (reading left to right) in degrees $2s$ through $0$, and 
\[ \Delta_k=I\pi_{b_{k-1}} \circ R\pi_{b_k}
\]
where $\pi_{b_k}: \Theta_{b_k}\to pt$ is the usual projection.
One readily checks that $L(\und{b})$ is indeed a chain complex: use that $R\pi(\id-Rt)=R\pi-R(\pi t)=R\pi-R\pi=0$
and \[(\id-Rt)I\pi_{b_k}=(\id-It^{-1})I\pi_{b_k}=I\pi_{b_k}-I(t^{-1}\pi_{b_k})=I\pi_{b_k}-I\pi_{b_k}=0.\]  

When $b_1,\ldots,b_s$ is not a divisor string, we define $L(\und{b})$ to be the $L$-complex for the divisor string $(\und{b};1),(\und{b};2),\ldots,(\und{b};s)$.

\begin{prop}
\label{pr:spheres-small-model}
For any divisors $b_1,\ldots,b_s$ of $n$
there is a quasi-isomorphism
$S^{\lambda_{b_1}}\bbox \cdots \bbox S^{\lambda_{b_s}}\he L(\und{b})$.
\end{prop}

\begin{proof}
By Proposition~\ref{pr:lcm-gcd-spheres} we immediately reduce to the case where $\und{b}$ is a divisor string.  
The proof now proceeds by induction on $s$, the case $s=1$ being trivial.

Let $B=L(b_2,\ldots,b_s)$ and define the complex $\tilde{B}$ by the cofiber sequence
\[ \mZ \inc B \ra \tilde{B}.
\]
So the terms of $\tilde{B}$ are $F_{b_i}$ for $i\geq 1$, and consequently we have $\tilde{B}\bbox \mZ/I_{b_1}=0$ by (\ref{eq:box-zero}) (note that this uses that each $b_i$ is a multiple of $b_1$).  

Take the cofiber sequence $\Sigma^2 \mZ \ra S^{\lambda_{b_1}} \ra \mZ/I_{b_1}$ and box with $\tilde{B}$.  Since the third term is contractible, the map
\[ \tag{*} \Sigma^2 \tilde{B} \ra S^{\lambda_{b_1}}\bbox \tilde{B} 
\]
is  a quasi-isomorphism.

Now take the cofiber sequence $\mZ\inc B\ra \tilde{B}$ and box with $S^{\lambda_{b_1}}$ to give
\[ S^{\lambda_{b_1}} \ra S^{\lambda_{b_1}}\bbox B \ra S^{\lambda_{b_1}}\bbox \tilde{B}.
\]
Backtracking the cofiber sequence and using (*), together with the induction hypothesis that $B\he S^{\lambda_{b_2}+\cdots+\lambda_{b_s}}$, gives a cofiber sequence
\[ \Sigma \tilde{B} \llra{f} S^{\lambda_{b_1}}\ra S^{\lambda_1+\cdots+\lambda_{b_s}}.
\]
The map $f$ takes the form
\[ \xymatrix{
\cdots \ar[r] & F_{b_2}\ar[r] & F_{b_2} \ar[d]_{f}\ar[r] & 0 \\
\cdots \ar[r] & 0\ar[r] & F_{b_1} \ar[r]^{\id-Rt} & F_{b_1} \ar[r] & \mZ
}
\]
and so has a unique nonzero component.  Commutativity of the diagram and Proposition~\ref{pr:lifting-lems-2}(b) implies that  $f=w\cdot I\pi_{b_1} R\pi_{b_2}$ for some $w\in \Z$.
If we show that $w=\pm 1$ then the mapping cone of $f$ is clearly isomorphic to $L(b_1,\ldots,b_s)$, and we would be done.

Let $Cf$ be the mapping cone of $f$, and compute $\uH_2(Cf)(\Theta_1)$.  One readily sees that it is $\Z/(wb_2)$, generated by the class $p_*(g_{\Theta_{b_1}})$.  But we already know $Cf\he S^{\lambda_{b_1}+\cdots+\lambda_{b_s}}$ and we have previously calculated that $\uH_2$ of this sphere is $\mZ/I_{b_2}$.  The value of this $\mZ$-module at $\Theta_1$ is $\Z/b_2$, so we must have $w=\pm 1$. 
\end{proof}

\begin{remark}
When $\und{b}$ is a divisor string the complex $L(\und{b})$ is readily identified as a certain homotopy colimit.  So for example, Proposition~\ref{pr:spheres-small-model} implies that $S^{\lambda_{b_1}+\lambda_{b_2}+\lambda_{b_3}}$ is the homotopy colimit of the diagram
\[ \xymatrixcolsep{3.2pc}\xymatrix{\Sigma^4 S^{\lambda_{b_3}}
& \Sigma^4 \mZ \ar[l]_-{\Sigma^4 a_{b_3}} \ar[r]^-{\Sigma^2u_{b_2}} & \Sigma^2 S^{\lambda_{b_2}} & 
\Sigma^2\mZ \ar[l]_-{\Sigma^2a_{b_2}}\ar[r]^-{u_{b_1}} & S^{\lambda_{b_1}}.
}
\]
Or one can inductively describe $S^{\lambda_{b_1}+\cdots+\lambda_{b_k}}$ as the homotopy pushout
\[ \xymatrixcolsep{3.5pc}\xymatrix{ \Sigma^{2k-2}S^{\lambda_{b_k}} &
\Sigma^{2k-2}\mZ \ar[l]_-{\Sigma^{2k-2}a_{b_k}}\ar[r]^-{u_{b_1} \cdots u_{b_{k-1}} } & S^{\lambda_{b_1}+\cdots+\lambda_{b_{k-1}}.}
}
\]
These models are sometimes useful.  
\end{remark}

The dual $L(\und{b})^*$ is given by 
\[ \xymatrixcolsep{2pc}\xymatrix{
0 \ar[r] &\mZ \ar[r]^-{I\pi}& F_{b_1} \ar[r]^{\id-It}& F_{b_1} \ar[r]^{I\pi R\pi}&  F_{b_2} \ar[r]^{\id-It} & \cdots \ar[r]^{I\pi R\pi} & F_{b_s}\ar[r]^{\id-It} & F_{b_s} \ar[r] & 0.
}
\]
(Here we have dropped the subscripts from the $\pi$-maps, but they can be inferred from the (co)domain.) The terms are now concentrated (reading left to right) in degrees $0$ through $-2s$.

\begin{cor}
Given a list $b_1,\ldots,b_s$ of divisors of $n$, there is a quasi-isomorphism
$S^{-\lambda_{b_1}}\bbox \cdots \bbox S^{-\lambda_{b_s}}\he L(\und{b})^*$.
\end{cor}

\begin{proof}
Immediate from Proposition~\ref{pr:spheres-small-model}.
\end{proof}


\section{Working with the negative cone}

The negative cone is more challenging than the positive cone because it is inherently non-Noetherian: its elements are typically infinitely-divisible by some subset of the $a$- and $u$-classes. In this section we will develop a system for writing down explicit generators for the homogeneous pieces of the negative cone indexed by divisor strings, and we will determine most of the products that stay within the regular region.  \medskip

The negative cone naturally divides into two regions: the ``integral edge'' which has homogenous pieces $\uH^{2k-\lambda_{b_1}-\cdots - \lambda_{b_k}}\cong I_{[\und{b}]}$ (everything here is non-torsion), and the ``torsion cone'' which has homogeneous pieces $\uH^{(2j+1)-\lambda_{b_1}-\cdots - \lambda_{b_k}}\cong \mZ/I_{(\und{b};j)}$.  It is good to keep in mind the picture in Figure~\ref{fig:negative-cone}, which shows a small slice.  Note that the classes $\frac{b}{u_b^i}$ and $\frac{\gamma_b}{u_b^ia_b^j}$ will be explained below in Sections~\ref{se:integral-edge} and \ref{se:gamma}, respectively.

\begin{figure}[ht]
\begin{tikzpicture}[scale=0.6]
\foreach \y in {-3.5,-2.5,-1.5,-0.5,0.5}
	{\draw[dotted] (-5,\y)--(1,\y);}
\foreach \y in {-4.5,-3.5,-2.5,-1.5,-0.5,0.5}
	{
	\draw[dotted] (\y,-4)--(\y,1.1);
	};
\draw(0,0) node{$\scriptstyle{\square}$};
\draw(0.8,0) node{$\scriptstyle{\frac{b}{u_b}}$};
\draw(0.8,-1) node{$\scriptstyle{\frac{b}{(u_b)^2}}$};
\draw(0.8,-2) node{$\scriptstyle{\frac{b}{(u_b)^3}}$};
\draw(0.8,-3) node{$\scriptstyle{\frac{b}{(u_b)^4}}$};
\draw(0,-1) node{$\scriptstyle{\square}$};
\draw(0,-2) node{$\scriptstyle{\square}$};
\draw(0,-3) node{$\scriptstyle{\square}$};
\foreach \y in {-1,-2,-3}
     {\filldraw(-1,\y) circle(0.1); 
     };
\foreach \y in {-2,-3}
     {\filldraw(-2,\y) circle(0.1); 
     };
\filldraw(-3,-3) circle(0.1);
\draw (-1,-1)--(-1,-3.7);
\draw (-2,-2)--(-2,-3.7);
\draw(-3,-3)--(-3,-3.7);
\draw(-1,-1)--(-3.7,-3.7);
\draw(-1,-2)--(-2.7,-3.7);
\draw(-1,-3)--(-1.7,-3.7);
\draw (-2,0) node{$\scriptstyle{\gamma_b}$};
\draw[->] (-1.8,-0.2)--(-1.2,-0.8);
\draw (-4,0) node{$\scriptstyle{\frac{\gamma_b}{u_b}}$};
\draw[->] (-3.7,-0.2)--(-1.2,-1.8);
\draw(-4,-2) node{$\scriptstyle{\frac{\gamma_b}{a_b}}$};
\draw[->] (-3.7,-2)--(-2.2,-2);
\draw(-5,-3) node{$\scriptstyle{\frac{\gamma_b}{(a_b)^2}}$};
\draw[->] (-4.7,-3)--(-3.2,-3);
\draw(7,-1) node{${\square=\Z,\ \ \bullet=\Z/b}$};
\draw(7,-2) node{${\lvert\ =\cdot u_b,\ \ \diagup\ = \cdot a_b}$};
\foreach \y in {-0.2,-1.2,-2.2}
   {  \draw(0,\y+0.05)--(0,\y-0.65);
   };
\draw (0,-3.15)--(0,-3.7);
\end{tikzpicture}
\label{fig:negative-cone}
\caption{The slice $H^{k-m\lambda_b}$ of the negative cone}
\end{figure}
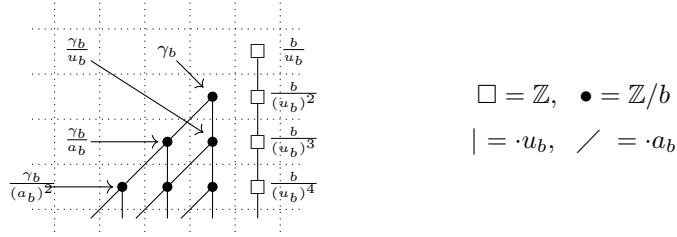

The picture of this small slice is misleading in one way that is worth mentioning: in general neither the negative cone nor the torsion cone will have a single `cone point'.
The top rim of the negative cone consists of the classes $\frac{b}{u_{b}}$ for all  $b|n$, and the top rim of the torsion cone consists of classes $\gamma_b$ and $\frac{a_b\gamma_b}{a_c}$ for $b|c|n$.  In both cases one should really picture a conical region descending from the class $1\in H^0$  but which has had a small bit of the top lopped off.

\subsection{The integral edge} 
\label{se:integral-edge}
We start by showing multiplication by $u$-classes gives an injection to $H^{0}$ and identify this image. 
\begin{lemma}\label{lem:uinj}
    Suppose $\und{b}=(b_1,\dots, b_k)$ is a sequence of divisors of $n$. Write $\lambda_{\und{b}}=\lambda_{b_1}+\dots+ \lambda_{b_k}$ and $u_{\und{b}}=u_{b_1}\cdots u_{b_k}$. Then multiplication by $u_{\und{b}}$ gives an injection
     \[
   \xymatrixrowsep{0.5pc}
    \xymatrix{
    \uH^{2k-\lambda_{\und{b}}}\ar@{^{(}->}[r]^-{u_{\und{b}}}\ar@{}[d]|*[@]{\rotatebox{270}{$\cong$}}& \uH^{0}\ar@{}[d]|*[@]{\rotatebox{270}{$\cong$}}\\
    I_{[\underline{b}]}& \underline{\Z}
    }
    \]
    and the image is $I_{[\und{b}]}$. In particular, in the $\Theta_1$-spot the image is generated by the least common multiple $[\und{b}]$.
\end{lemma}

\begin{proof}
Consider the map $\Sigma^{2k}\mZ \llra{u_{\und{b}}} S^{\lamu{b}}$ and let
$Q$ be the mapping cone. We get a long exact sequence
\[ \xymatrix{
0 & \uH^{2k+1}(S^{\lambda_{\und{b}}}) \ar@{=} [l]& \uH^{2k+1}(Q)\ar[l] & \mZ \ar[l] & \uH^{2k}(S^{\lamu{b}})\ar[l]_-{\cdot u_{\und{b}}} & \uH^{2k}(Q).\ar[l]
}
\]
We can model $S^{\lamu{b}}$ by the complex $L(\und{B})$ from Section~\ref{se:smallmodels}, where $\und{B}$ is the associated divisor string to $\und{b}$.  In this model the map $u_{\und{b}}$ (which is the generator for $\uH_{2k}(\blank)$) is the map $\Sigma^{2k}\mZ \ra L(\und{B})$ that equals $I\pi$ in degree $2k$.  The top three terms in the cofiber $Q$ then visibly coincide with the top of the standard resolution for $\mZ/I_{[\und{b}]}$.  Using that this resolution is self-dual, one immediately gets that $\uH^{2k+1}(Q)\iso \mZ/I_{[\und{b}]}$ and $\uH^{2k}(Q)=0$.  The proof is completed via the fact that any surjective map of $\mZ$-modules $\mZ\ra \mZ/I_{j}$ must have kernel exactly equal to $I_j$. 
\end{proof}

Suppose $g\in \uH^{2k-\lambda_{\und{b}}}(\Theta_1)$ is the generator mapping to $[\und{b}]$ from the above proof. It follows that this is a degree where the rationalization map $H^{\star}(pt)\to H^{\star}(pt)\otimes \Q$ is injective, and the generator $g$ maps to $\frac{[\und{b}]}{u_{\und{b}}}$. Thus we name this element by its image, writing
\[
g=\frac{[\und{b}]}{u_{\und{b}}}\in \uH^{2k-\lambda_{\und{b}}}(\Theta_1).
\]
This is the unique element of $H^\star$ having the property that multiplying it by $u_{\und{b}}$ gives $[\und{b}]$.
The next result immediately follows:

\begin{prop}
\label{pr:integral-edge-products}
Let $\und{b}=(b_1,\dots, b_k)$ be a sequence of divisors of $n$. For any $1\leq i \leq k$, write $\und{b}\setminus b_i$ for the tuple with $b_i$ removed. The following describes how to multiply any class from the regular portion with a class from the integral edge, assuming the product still lands in the regular portion:
\begin{enumerate}
    \item For any $1\leq i \leq k$, we have that
    \[
    u_{b_i}\cdot \frac{[\und{b}]}{u_{\und{b}}} = \frac{[\und{b}]}{[\und{b}\setminus b_i]} \frac{[\und{b}\setminus b_i]}{u_{\und{b}\setminus b_i}}.
    \]
    Thus multiplication by $u_{b_i}$ on $H^{2k-\lambda_{b_1}-\cdots-\lambda_{b_k}}$ is the map $\Z\to \Z$ given by multiplication by $\frac{[\und{b}]}{[\und{b}\setminus b_i]}$.
    \item  $a_{b_i}\cdot \tfrac{[\und{b}]}{u_{\und{b}}}=0$ for $1\leq i\leq k$.
    \item If $\und{c}=(c_1,\dots, c_j)$ is another sequence of divisors of $n$, then
    \[
    \dfrac{[\und{b}]}{u_{\und{b}}}\cdot \dfrac{[\und{c}]}{u_{\und{c}}} = \dfrac{[\und{b}][\und{c}]}{[\und{b}, \und{c}]}\cdot \dfrac{[\und{b}, \und{c}]}{u_{\und{b}}u_{\und{c}}}.
    \]
    \item If $\alpha$ is a homogeneous element of the torsion cone from a degree where the $\mZ$-module is $\mZ/I_c$ for some $c\lvert [\und{b}]$, then $\alpha \cdot \frac{[\und{b}]}{u_{\und{b}}}=0$.
\end{enumerate}
\end{prop}

\begin{proof}
Parts (1) and (3) are immediate from the defining properties of the $\frac{[\und{b}]}{u_{\und{b}}}$ classes.  Part (2) follows for degree reasons, as the group $H^{2k-\lambda_{\und{b}\setminus b_i}}$
equals zero by Proposition~\ref{pr:negative-cone}.
For (4), let $\alpha\in H^\beta(\pt)$.  Then $\alpha\cdot \tfrac{[\und{b}]}{u_{\und{b}}}$ is in the image of  the multiplication map
\[ \uH^\beta \bbox \uH^{2k-\lamu{b}} \llra{\mu} \uH^{2k+\beta-\lamu{b}}.
\]
But the domain is $\mZ/I_c \bbox I_{[\und{b}]}$, which is zero if $c$ divides $[\und{b}]$ by (\ref{eq:box-zero}).  
\end{proof}

\subsection{The gamma classes}
\label{se:gamma}
Our next goal is to understand as much as we can about
the torsion cone. Recall from Proposition~\ref{pr:negative-cone} that for each divisor $b$ we have $H^{3-2\lambda_b}\cong \Z/b$. We first introduce special generators that live in these various ``subtips'' of the negative cone, and then we show these generators are infinitely divisible by certain $a$- and $u$-classes.  This will allow us to conveniently describe many of the products.

Let $Y_b\ra \mZ/I_b$ be our standard free resolution from (\ref{eq:std-free}).  
Define $\Gamma_b$ to be the map in $\cD(\mZ/I_b,\Sigma^3\mZ)=\Ext^3(\mZ/I_b,\mZ)\iso \Z$ represented by the zig-zag
\[ \mZ/I_b \llla{\he} Y_b \lra \Sigma^3 \mZ
\]
where the right arrow is projection onto the component in degree $3$ (this map was called $B_b$ in the introduction).  Composition with $\Gamma_b$ gives a natural operation
\[ \uH^i(X;\mZ/I_b) \ra \uH^{i+3}(X;\mZ),
\]
which we will also call $\Gamma_b$ by abuse.
This can be thought of as a generalized Bockstein operator.  Note that following this procedure in the category $\Ab$ of abelian groups, using a chosen generator of $\Ext^1_{\Ab}(\Z/b,\Z)$, gives the usual integral Bockstein operator $H^i(X;\Z/b)\ra H^{i+1}(X;\Z)$.  

Every map $\mZ/I_c\ra \mZ/I_b$ is multiplication by $\tfrac{b}{(b,c)}\cdot e$ for some $e$.  The following proposition explains how these maps related to the $\Gamma$-operations:

\begin{prop}
\label{pr:Gamma-ident}
Let $E\colon \mZ/I_c\ra  \mZ/I_b$ be multiplication by $\tfrac{b}{(b,c)}e$.  Then for any $x\in H^i(X;\mZ/I_c)$ one has
\[ \Gamma_b(Ex)=\tfrac{c}{(b,c)}e\cdot \Gamma_c(x)
\]
as elements of $H^{3+i}(X;\mZ)$.  
\end{prop}

\begin{proof}
We need to show that the composite $\Gamma_b\circ E$ equals the given multiple of $\Gamma_c$.  This can be checked by lifting $E$ to a map between our standard resolutions:
\[ \xymatrix{
0 \ar[r] & \mZ \ar[r]^{I\pi}\ar[d]_{\cdot\frac{c}{(b,c)}e} & F_{c} \ar[r]^{\id-Rt}\ar[d]_{f} & F_{c} \ar[r]^{R\pi}\ar[d]_{f} & \mZ \ar[r]\ar[d]^{\cdot\frac{b}{(b,c)}e} & \mZ/I_c\ar[d]^E \ar[r] & 0 \\
0 \ar[r] & \und{\Z} \ar[r]^{I\pi} & F_{b} \ar[r]^{\id-Rt} & F_{b} \ar[r]^{R\pi} & \und{\Z} \ar[r] & \mZ/I_b \ar[r] & 0 \\
}
\]
where $f(g_{c})=e\cdot (p_{\Theta_{[b,c]}\ra \Theta_{c}})_*(p_{\Theta_{[b,c]}\ra \Theta_{b}})^*(g_{b})$.  From here the result is immediate.
\end{proof}

\begin{remark}
It is convenient to remember the identity from Proposition~\ref{pr:Gamma-ident} as
\[ \Gamma_b(kx)=\tfrac{c}{b}k\cdot \Gamma_c(x).
\]
This is fine as long as one remembers that $kx$ will only make sense as a generic map from $H^i(X;\mZ/I_c)$ to $H^i(X;\mZ/I_b)$ when $k$ is a multiple of $\frac{b}{(b,c)}$, and therefore the expression $\frac{ck}{b}$ is actually an integer. 
\end{remark}

We have a canonical map $S^{\lambda_b}\ra \mZ/I_b$ that quotients out by the terms in positive degrees and is the canonical projection in degree $0$.  This is an element of $H^0(S^{\lambda_b};\mZ/I_b)\iso \Z/b$, and it generates this group.  It is natural to denote it $\bbone_b$.  Taking box products gives us $\bbone_b\bbox \bbone_b \in H^0(S^{\lambda_b}\bbox S^{\lambda_b};\mZ/I_b)$, and then we define
\[ \gamma_b=\Gamma_b(\bbone_b\bbox \bbone_b) \in H^3(S^{\lambda_b}\bbox S^{\lambda_b})=H^{3-2\lambda_b}.
\]

Recall we have generators $u_{[b:c]}\in [S^{\lambda_c}, S^{\lambda_b}]=H^{\lambda_b-\lambda_c}$ that we defined in Section~\ref{se:more-u}. The next result gives some relations between products of these with the $\gamma$-classes:

\begin{prop}
\label{pr:u-gamma}
Let $b$ and $c$ be divisors of $n$. Then
\begin{enumerate}[(a)]
\item $u_{[b\cln c]}\gamma_b=u_{[c\cln b]}\gamma_c$, and
\item $u_{[b\cln c]}^2\gamma_b=\tfrac{[b,c]}{(b,c)}\cdot \gamma_c$.  In particular, if $b|c$ then $u_{[b\cln c]}^2\gamma_b=\tfrac{c}{b}\cdot \gamma_c$.  
\end{enumerate}
\end{prop}

\begin{proof}
Part (b) follows from (a) and the fact that $u_{[b\cln c]}u_{[c\cln b]}=\tfrac{[b,c]}{(b,c)}$ from 
Proposition~\ref{pr:u-props}.  For (a) we examine the following diagram:

\[ \xymatrixcolsep{3.2pc}\xymatrix{
S^{\lambda_c}\bbox S^{\lambda_c} \ar[r]^{\bbone_c\bbox \bbone_c} & \mZ/I_c\bbox \, \mZ/I_c \ar@{=}[r] & \mZ/I_c \ar[r]^{\Gamma_c} & \Sigma^3\mZ \\
S^{\lambda_b}\bbox S^{\lambda_c} \ar[r]^{\bbone_b\bbox \bbone_c} \ar[d]_{\id\,\bbox u_{[b\cln c]}}\ar[u]^{u_{[c\cln b]}\bbox \id} & \mZ/I_b\bbox\, \mZ/I_c \ar[u]_{\frac{c}{(b,c)}\bbox \id} \ar[d]^{\frac{b}{(b,c)}\bbox \id}
 \ar@{=}[r] & \mZ/I_{(b,c)} \ar[u]_{\frac{c}{(b,c)}}\ar[d]^{\frac{b}{(b,c)}}\ar[r]^{\Gamma_{(b,c)}}
 & \Sigma^3\mZ\ar@{=}[d]\ar@{=}[u]\\
S^{\lambda_b}\bbox S^{\lambda_b} \ar[r]_{\bbone_b\bbox \bbone_b} & \mZ/I_b\bbox\, \mZ/I_b \ar@{=}[r] & \mZ/I_b \ar[r]^{\Gamma_b} & \Sigma^3\mZ. 
}
\]
The composite from $S^{\lambda_b}\bbox S^{\lambda_c}$ up and across is equal to $u_{[c\cln b]}\gamma_c$, and the composite down and across is equal to $u_{[b\cln c]}\gamma_b$.  So it will suffice to show that the diagram commutes.

The squares in the left column commute because our chain model for $u_{[x\cln y]}$ is multiplication by $\frac{x}{(x,y)}$ in degree zero (see the definition in Section~\ref{se:more-u}).  The commutativity of the squares in the middle column is evident.  Finally, the two squares in the right column commute by Proposition~\ref{pr:Gamma-ident}.
\end{proof}

\begin{remark} Write $b=\ell_1^{e_1}\cdots \ell_s^{e_s}$ where the $\ell_i$ are distinct primes.  The numbers $b(\hat{\ell_i})$ are relatively prime, therefore we can write $1=\sum_i A_i b(\hat{\ell_i})$ for some $A_i\in \Z$.  Then
\[ \gamma_b=\sum_i A_i b(\hat{\ell_i}) \gamma_b = \sum_i A_i u_{[b(\ell_i):b]}^2\gamma_{b(\ell_i)}.
\]
So every $\gamma_b$ class can be expressed in terms of the corresponding $\gamma$-classes for primes powers.
\end{remark} 

\subsection{Naming classes in the negative torsion cone} We now have preferred generators $\gamma_b\in H^{3-2\lambda_b}$ for each divisor $b$. We next use these classes to pick preferred generators for the rest of the torsion cone. We do this by showing we can always use certain $a$- and $u$-classes to isomorphically move from any spot in the negative cone to a $\gamma$-class.

Let $\und{b}=(b_1,\dots, b_k)$ be a divisor string and let $1\leq s \leq k-1$. Keep in mind that the $b_i$ need not be distinct, and often will not be. Fix another divisor $c|n$ and consider the two maps
\begin{equation}
\label{eq:a-u}
\xymatrix{\uH^{?} \ar[r]^-{\cdot a_{c}} & \uH^{2s+1-\lamu{b}} & \uH^{??}\ar[l]_-{\cdot u_{c}}
}
\end{equation}
(we have suppressed the degrees of the domains only to heighten readability, as they are easy to work out).  
The following result gives a useful range in which these maps are isomorphisms:

\begin{prop}  
\label{pr:iso-regions}
Let $\und{b}$ and $c$ be as above.  
\begin{enumerate}[(a)]
\item Suppose that $c$ can be inserted into $\und{b}$ after $b_s$ in a way that makes a new divisor string.  Then the $\cdot a_c$ map in (\ref{eq:a-u}) is an isomorphism.
\item Suppose that $c$ can be inserted into $\und{b}$ before $b_s$ in a way that makes a new divisor string.  Then the $\cdot u_c$ map in (\ref{eq:a-u}) is an isomorphism.
\end{enumerate}
\end{prop}

We delay the proof for a moment, but notice that this gives a tool for propagating both generators and relations in $H^{2s+1-\lamu{b}}$ downwards in the negative cone, through the above isomorphisms.  

Iterated use of Proposition~\ref{pr:iso-regions} yields the following diagram. Note from Proposition~\ref{pr:negative-cone} that all of the groups are isomorphic to $\Z/b_s$.

\begin{equation}\label{eq:torsioncone}
\xymatrixcolsep{3.2pc}\xymatrix{
H^{2s+1-\lambda_{\und{b}}}\ar[r]^-{u_{b_1}\cdots u_{b_{s-1}}}_-{\cong} & H^{3-\lambda_{b_s}-\dots - \lambda_{b_{k}}} \ar[r]^-{a_{b_{s+2}}\cdots a_{b_{k}}}_-{\cong} & H^{3-\lambda_{b_s}-\lambda_{b_{s+1}}}& \\
&&\ar[u]^-\cong_-{a_{b_{s}}} H^{3-2\lambda_{b_s}-\lambda_{b_{s+1}}} \ar[r]^-{a_{b_{s+1}}}_-{\cong} & H^{3-2\lambda_{b_s}}.
}
\end{equation}
Taking the preimage of the generator $\gamma_{b_s}\in H^{3-2\lambda_{b_s}}$ across the isomorphisms, we get a generator of $H^{2s+1-\lambda_{\und{b}}}$ that we will denote
\begin{equation}
\label{eq:gamma-class}
\frac{a_{b_s}\gamma_{b_s}}{u_{b_1}\cdots u_{b_{s-1}}a_{b_{s+1}}\cdots a_{b_k}}\in H^{2s+1-\lambda_{\und{b}}}.
\end{equation}
To ease typography in what follows we will sometimes write $u_{\langle b_1,\ldots,b_{s-1}\rangle}$ as an abbreviation for the produect $u_{b_1}\cdots u_{b_{s-1}}$.  (While this doesn't seem like much of an abbreviation, it really does help when parsing complicated formulas.)

Three useful things to remember about the element in (\ref{eq:gamma-class}):
\begin{itemize}
\item The additive order is precisely $b_s$, which is the $\gamma_?$ index in the numerator.
\item The `fixed dimension' of the overall degree is equal to \[ 2\cdot(\text{\# of $u$-classes in denominator})+3.
\]
\item At least one $a$-class must appear in the denominator.  Thus, expressions like $\frac{a_{b_2}\gamma_{b_2}}{u_{b_1}}$ are not meaningful.  However, $\frac{\gamma_b}{u_b}$ is fine because it equals $\frac{a_b\gamma_b}{u_ba_b}$.  
\end{itemize}

\begin{example}
   Suppose we have three divisors $b,c,d$ with $b\mid c \mid d$ and we want to understand all of the groups $H^{*-\lambda_b-\lambda_c-\lambda_d}$. Then we have the following torsion generators
    \[
    \dfrac{a_b\gamma_b}{a_ca_d}\in H^{3-\lambda_b-\lambda_c-\lambda_d}\cong \Z/b \quad \text{and} \quad \dfrac{a_c\gamma_c}{u_ba_d}\in H^{5-\lambda_b-\lambda_c-\lambda_d}\cong \Z/c.
    \]
    The class $\frac{d}{u_bu_cu_d}$ generates $H^{6-\lambda_b-\lambda_c-\lambda_d}\iso \Z$, and $H^{*-\lambda_b-\lambda_c-\lambda_d}=0$ for $*\not\in \{3,5,6\}$. 
    
    When we just have two divisors $b\mid c$, then we have the single torsion generator
    \[
    \frac{a_b\gamma_b}{a_c}\in H^{3-\lambda_b-\lambda_c}\iso \Z/b
    \]
    as well as the nontorsion element $\tfrac{c}{u_bu_c}\in H^{4-\lambda_b-\lambda_c}\iso \Z$.  All other groups $H^{*-\lambda_b-\lambda_c}$ are zero.
\end{example}

To summarize, given a divisor string $\und{b}=(b_1,\dots, b_k)$ we get torsion generators in the negative cone by creating ``valid fractions'' with $\gamma$-, $a$-, and $u$-classes. There is a unique valid fraction for each $1\leq s<k$: the numerator is $a_{b_{s}}\gamma_{b_s}$ and the denominator is the product $u_{b_1}\cdots u_{b_{s-1}}\cdot a_{b_{s+1}}\cdots a_{b_k}$ (so we have $u$-classes for the divisors of $b_s$ and $a$-classes for the numbers divisible by $b_s$). Note that $s<k$, so every valid fraction must have an $a$-class in the denominator. Except we are allowed to ``reduce the fractions'' if we have $b_{s}=b_{s+1}$. That is, we can cancel the $a_{b_s}=a_{b_{s+1}}$ from the numerator and denominator, e.g. for $b\mid c$ we write
\[
\frac{a_c\gamma_c}{u_ba_c}=\frac{\gamma_c}{u_b}.
\]
\begin{remark}
In terms of division, this says we can infinitely divide $\gamma_c$ by $u_b$-classes with $b\mid c$ and $a_d$-classes with $c\mid d$. Also for each pair of divisors $c_1\mid c_2$, we have classes $\frac{a_{c_1}\gamma_{c_1}}{a_{c_2}}$ that are similarly infinitely divisible by $u_b$-classes with $b\mid {c_1}$ and $a_d$-classes with $c_2\mid d$. In both cases the resulting fractions generate the Mackey functor in that degree. Furthermore all degrees in which the negative torsion cone is nonzero can be generated this way.
\end{remark}

We now prove the various isomorphisms from Proposition~\ref{pr:iso-regions}.  These will be instances of the results in Proposition~\ref{pr:umulttorcone} and Proposition~\ref{pr:amulttorcone} below.  

\begin{prop}\label{pr:umulttorcone}
Suppose $\und{b}=(b_1,\dots, b_k)$ is a divisor string of $n$ and $1\leq s \leq k-1$. Then the map $\uH^{2s+1-\lambda_{\und{b}}}\to \uH^{2s-1-\lambda_{\und{b}}+\lambda_{b_i}}$ given by multiplication by $u_{b_i}$ is an isomorphism if $i<s$ and surjective if $i=s$.
\end{prop}
\begin{proof}
Consider the usual cofiber sequence $\Sigma^2 \mZ \overset{u_{b_i}}{\longrightarrow} S^{\lambda_{b_i}} \longrightarrow \mZ/I_{b_i}$. Applying $\uH^{2s+1-\lambda_{\und{b}}+\lambda_{b_i}}$ gives the exact sequence
\[
\xymatrix{
\uH^{2s+1-\lambda_{\und{b}}} \ar[r]^-{u_{b_i}}&\ar[r]\uH^{2s-1-\lambda_{\und{b}}+\lambda_{b_i}}&\uH^{2s+2-\lambda_{\und{b}}+\lambda_{b_i}}(\mZ/I_{b_i}).
}
\]
Write $j=2s+1$. We know by Proposition~\ref{pr:negative-cone} that $\uH^{j-2-\lambda_{\und{b}}+\lambda_{b_i}} \iso \mZ/I_r$ for an appropriate $r$, and hence the $\mZ$-module is generated in the $\Theta_1$-spot.   
So to show the map $u_{b_i}$ is surjective it suffices to show
surjectivity in the $\Theta_1$-spot. To that end we show the rightmost $\mZ$-module is zero in the $\Theta_1$-spot. Start by observing that
\begin{align*}
H^{j+1-\lambda_{\und{b}}+\lambda_{b_i}}(\mZ/I_{b_i})& = 
\cD(\Sigma^{-j-1}\mZ/I_{b_i}, S^{-\lambda_{\und{b\setminus b_i}}})\\
&\cong \cD((S^{-\lambda_{\und{b\setminus b_i}}})^*, (\Sigma^{-j-1}\mZ/I_{b_i})^*)\\ &\cong \cD(S^{\lambda_{\und{b\setminus b_i}}}, \Sigma^{j-2}\mZ/I_{b_i}),
\end{align*}
where the last isomorphism uses that $(\mZ/I_{b_i})^*\he \Sigma^{-3}\mZ/I_{b_i}$.  This group can now be computed by considering chain homotopy classes of maps $S^{\lambda_{\und{b\setminus b_i}}} \to \Sigma^{j-2}\mZ/I_{b_i}$. Using the linear models of spheres from Section~\ref{se:smallmodels}, these are maps of the form (noting $j-2=2s-1$)
\[
\xymatrix{
F_{{b_k}}\ar[r] & F_{{b_k}} \ar[r]& F_{{b_{k-1}}}\ar[r] & \dots\ar[r] & F_{{b_{s+1}}}\ar[r] & F_{{b_{s+1}}} \ar[r]\ar[d]& \dots\\
& & & & & \mZ/I_{b_{i}} &
}
\]
If $i\leq s$ then $b_i\mid b_{s+1}$ and so $\mZ/I_{b_{i}}(\Theta_{b_{s+1}})=0$. Thus all such chain maps are zero, and we conclude multiplication by $u_{b_i}$ is surjective as long as $b_i\lvert b_s$.

Finally note if $i<s$ then $(\und{b};s)=b_s=(\und{b\setminus b_{i}};s-1)$ (since $b_s$ is now the $(s-1)$-st smallest number in the string $\und{b\setminus b_{i}}$). Hence $\uH^{j-\lambda_{\und{b}}}\cong \uH^{j-2-(\lambda_{\und{b}}-\lambda_{b_i})}$, so the domain and codomain of our $\cdot u_{b_i}$ map are isomorphic. Since both of these Mackey functors are cyclic in each spot, surjectivity  implies injectivity. 
\end{proof}

Now we consider multiplication by $a_b$. Define $X_b$ to be the complex $F_{b}\overset{\id-Rt}{\longrightarrow} F_{b}$. Note $\Sigma X_b$ is exactly the cofiber of $a_b\colon \mZ \to S^{\lambda_b}$. To show multiplication by $a_b$ is injective/surjective in certain ranges, we show the cohomology of $X_b$ vanishes. We have the following lemma to help with this:

\begin{lemma}\label{le:suspxb}
For divisors $c\mid b \mid n$ we have $S^{\lambda_c}\bbox X_b\simeq \Sigma^2 X_b$.
\end{lemma}

\begin{proof}
Take the cofiber sequence $\Sigma^2 \mZ \overset{u_b}{\longrightarrow} S^{\lambda_c}\longrightarrow \mZ/I_c$ and box with $X_b$. This gives the sequence
\[
\Sigma^2 X_b \to S^{\lambda_c}\bbox X_b \to \mZ/I_c \bbox X_b.
\]
But $c\mid b$ implies $ \mZ/I_c \bbox X_b=0$ by (\ref{eq:box-zero}), and so the equivalence follows.
\end{proof}

\begin{prop}\label{pr:amulttorcone}
Suppose $\und{b}=(b_1,\dots, b_k)$ is a divisor string of $n$ and $1\leq s \leq k-1$. Then the map $\uH^{2s+1-\lambda_{\und{b}}}\to \uH^{2s+1-\lambda_{\und{b}}+\lambda_{b_i}}$ given by multiplication by $a_{b_i}$ is an isomorphism if $i\geq s+1$ and $s+1\neq k$.
\end{prop}

Note when $s=k-1$ multiplication by $a_{b_i}$ will be zero for degree reasons.

\begin{proof}[Proof of Proposition~\ref{pr:amulttorcone}]
Consider the cofiber sequence $\mZ\overset{a_{b_i}}{\longrightarrow} S^{\lambda_{b_i}}\longrightarrow \Sigma X_{b_i}$. Write $j=2s+1$. To show multiplication by $a_{b_i}$ is injective it suffices to show $H^{j-\lambda_{\und{b}}+\lambda_{b_i}}(\Sigma X_{b_i})=0$. Let $b'=(b_1,\dots, b_{s})$ and $b''=(b_{s+1},\dots, b_k)\setminus b_i$. We have that
\begin{align*}
H^{j-\lambda_{\und{b}}+\lambda_{b_i}}(\Sigma X_{b_i})&=[\Sigma^{1-j}X_{b_i}, S^{-\lambda_{\und{b}}+\lambda_{b_i}}] =[\Sigma^{1-j+\lambda_{\und{b}'}}X_{b_i}, S^{-\lambda_{\und{b}''}}].
\end{align*}
Every divisor in $\und{b}'$ divides $b_i$ so by Lemma~\ref{le:suspxb} we have that $\Sigma^{1-j+\lambda_{\und{b}'}}X_{b_i}\simeq \Sigma^{1-j+2s}X_{b_i}=X_{b_i}$. Thus we need to consider maps $X_{b_i}\ra S^{-\lambda_{\und{b}''}}$, which look like
\[
\xymatrix{
X_{b_i}: \ar[d] && F_{{b_i}} \ar[r]^-{\id-Rt} & F_{{b_i}}\ar[d] & &\\
S^{-\lambda_{\und{b}''}}: && & \und{\Z} \ar[r]^-{I\pi }& F_{{?}} \ar[r] & \dots \
}
\]
Observe the bottom complex is not just $\mZ$ because either $i>s+1$ and so $?=s+1$, or $i=s+1$ which implies $?=s+2$ (recall $s+1<k$ by assumption). The map $I\pi$ is injective so there are no nonzero chain maps. We conclude $[\Sigma^{1-j+\lambda_{\und{b}'}}X_{b_i}, S^{-\lambda_{\und{b}''}}]=0$. Thus multiplication by $a_{b_i}$ is injective. Finally note $\uH^{2s+1-\lambda_{\und{b}}}\cong \uH^{2s+1-\lambda_{\und{b}}+\lambda_{b_i}}\cong \mZ/I_{b_s}$ and so injectivity implies surjectivity.
\end{proof}

\subsection{Products in the negative cone}
We have so far produced three kinds of classes in $H^\star$:
\begin{itemize}
\item The $a$- and $u$- classes in the positive cone;
\item The classes $\tfrac{[\und{b}]}{u_{\und{b}}}$ on the integral edge of the negative cone;
\item The $\gamma$-classes in the torsion part of the negative cone.  
\end{itemize}
We have treated products of the first two types in Proposition~\ref{pr:integral-edge-products}.  
The following result lists some rules for products involving the third type.  Cases marked with a question mark are ones where we do not know a general rule.  

\begin{prop}
\label{pr:neg-cone-products}
Let $(b_1,\ldots,b_k)$ be a divisor string and $1\leq s<k$. 

\begin{enumerate}[(1)]
\item The product of two $\gamma$-classes is always zero.
\item $\frac{[\und{d}]}{u_{\und{d}}}\cdot \frac{a_{b_s}\gamma_{b_s}}{u_{\langle \ldots,b_{s-1}\rangle}a_{\langle b_{s+1},\ldots\rangle}}=
\begin{cases}
[\und{d}]\cdot \tfrac{a_{b_s}\gamma_{b_s}}{u_{\und{d}}u_{\langle \ldots,b_{s-1}\rangle}a_{\langle b_{s+1},\ldots\rangle}} & \text{if (**)},\\
0 & \text{if $b_s$ divides $[\und{d}]$},\\
?? & \text{other cases.}
\end{cases}
$

\noindent
where (**) is ``$\und{d}$ interlaces with $b_1,\ldots,b_{s-1},b_s$ to make a divisor string ending in $b_s$''.
\item  
$u_{c}\cdot \frac{a_{b_s}\gamma_{b_s}}{u_{\langle \ldots, b_{s-1}\rangle } a_{\langle b_{s+1},\ldots \rangle}}
= 
\begin{cases}
\tfrac{a_{b_s}\gamma_{b_s}}
{u_{\langle \ldots,\widehat{b}_i,\ldots b_{s-1}\rangle } a_{\langle b_{s+1},\ldots \rangle}} & \text{if $c=b_i$ for $i\leq s-1$}\\[0.1in]
\frac{b_i}{b_s} \cdot \frac{a_{b_{s-1}}\gamma_{{b_{s-1}}}}{u_{\langle \ldots, b_{s-2}\rangle}a_{\langle b_s,\ldots,\widehat{b}_i,\ldots\rangle}} & \text{if $c=b_i$ for $i\geq s+1$},\\
?? & \text{other cases}.
\end{cases}
$ \\[0.1in]
\item 
$a_{c}\cdot \frac{a_{b_s}\gamma_{b_s}}{u_{\langle \ldots b_{s-1}\rangle} a_{\langle b_{s+1},\ldots\rangle}}=
\begin{cases}
\tfrac{a_{b_s}\gamma_{b_s}}
{u_{\langle \ldots b_{s-1}\rangle} a_{\langle b_{s+1},\ldots,\widehat{b}_i,\ldots \rangle}}
& \text{if $c=b_i$ for $i\geq s+1$}, \\[0.1in]
\frac{b_{s+1}}{b_i} \cdot \frac{a_{b_{s+1}}\gamma_{{b_{s+1}}}}{u_{\langle \ldots,\widehat{b}_i,\ldots, b_s\rangle}
a_{\langle b_{s+2},\ldots \rangle}}
& \text{if $c=b_i$ for $i\leq s-1$}, \\
?? & \text{other cases.}
\end{cases}
$
\end{enumerate}
\end{prop}

The proof of Proposition~\ref{pr:neg-cone-products} breaks up into two pieces.  The two identities that involve scalar factors of the form $\tfrac{b_i}{b_j}$ are harder than the others and will require more tools than we have so far.  
But let us go ahead and take care of the easy parts:

\begin{proof}[Proof of Proposition~\ref{pr:neg-cone-products}, Part 1]
For (1), observe that the $\gamma$-classes are all in degrees where the fixed dimension is odd.  So the product of two of them lives in a degree where the fixed dimension is even.  By Proposition~\ref{pr:negative-cone} the groups of the negative cone in such degrees are either zero or torsion-free, and so the product must be zero.  

For (2), in the first case the class is uniquely determined by what happens when we multiply by $u_{\und{d}}u_{\langle b_1,\ldots,b_{s-1}\rangle}$, and one readily checks that the two sides of the equation give the same answer after such multiplication.  In the second case the result is by Proposition~\ref{pr:integral-edge-products}(4).  

The first cases in both (3) and (4) follow by reasoning similar to the previous paragraph, where one simply refers to the defining properties of the classes involved. 
\end{proof}

It remains to prove the `hard' cases of Proposition~\ref{pr:neg-cone-products}(3) and (4).  We pause to develop some preliminary results.  

Recall we have the classes $\gamma_b\in H^{3-2\lambda_b}$ and, when $b|c$, $\tfrac{a_b\gamma_b}{a_c}\in H^{3-\lambda_b-\lambda_c}$.  Call all of these classes the \dfn{subtips of the negative cone}.  Our next task is to establish how to move among these subtips using the $u$-classes.  

\begin{prop}\label{pr:u-subtips} Assume $b|c$. Then
\begin{enumerate}[(a)]
\item $u_{[b:c]}\gamma_b=\tfrac{a_b\gamma_b}{a_c}$,
\item $u_{[c:b]}\tfrac{a_b\gamma_b}{a_c}=\tfrac{c}{b}\gamma_b$,
\item $u_{[c:b]}\gamma_c=\tfrac{a_b\gamma_b}{a_c}$,
\item $u_{[b:c]}\tfrac{a_b\gamma_b}{a_c}=\tfrac{c}{b}\gamma_c$,
\item ${a_b^2}\cdot \tfrac{\gamma_b}{a_c^2}=\tfrac{c}{b} \gamma_c$.
\end{enumerate}
\end{prop}

The following diagram summarizes the above relations.

\[ \xymatrixcolsep{5pc}\xymatrix{
\gamma_b \ar@/^3ex/[r]^{u_{[b:c]}}& \tfrac{a_b\gamma_b}{a_c}\ar@/^3ex/[l]^{\fbox{$\scriptstyle{u_{[c:b]}}$}_{\frac{c}{b}}} \ar@/_3ex/[r]_{\fbox{$\scriptstyle{u_{[b:c]}}$}_{\frac{c}{b}}} & \gamma_c \ar@/_3ex/[l]_{u_{[c:b]}}
}
\]
The boxes indicate that an extra coefficient is involved, which is given outside the box.

\begin{proof}[Proof of Proposition~\ref{pr:u-subtips}]
First note that (a) implies (b) and (c) implies (d), both using that $u_{[b:c]}u_{[c:b]}=\tfrac{c}{b}$.  So we only need to prove (a) and (c).

For (a) we just compute that
\[ u_{[b:c]}\cdot \tfrac{\gamma_b}{a_b}\cdot a_c=u_{[b:c]}a_c\cdot \tfrac{\gamma_b}{a_b}=a_b\cdot \tfrac{\gamma_b}{a_b}=\gamma_b.
\]
Therefore $u_{[b:c]}\tfrac{\gamma_b}{a_b}=\tfrac{\gamma_b}{a_c}$, and mutiplying both sides by $a_b$ gives the desired relation.  

Part (c) is immediate from (a) and Proposition~\ref{pr:u-gamma}(a).

Finally, for (e) we just compute:
\[ a_b^2\cdot \tfrac{\gamma_b}{a_c^2}=a_b\cdot \tfrac{a_b\gamma_c}{a_c^2}=a_b\cdot \tfrac{ \left( \tfrac{a_b\gamma_c}{a_c}\right) }{a_c} = a_b\cdot \tfrac{u_{[c:b]}\gamma_c}{a_c}=a_bu_{[c:b]}\cdot \tfrac{\gamma_c}{a_c}=\tfrac{c}{b}a_c\cdot \tfrac{\gamma_c}{a_c}=\tfrac{c}{b}\gamma_c.
\]
\end{proof}

Here are two other useful facts:

\begin{prop} 
\label{pr:u-gamma-2}
Suppose that $b\lvert c$ and $c \lvert d$.  Then
\begin{enumerate}[(a)]
\item $u_{[b\cln c]}\cdot \frac{\gamma_d}{u_b}=\tfrac{c}{b}\cdot \tfrac{\gamma_d}{u_c}$,
\item  $u_{[d\cln b]}\cdot \tfrac{a_c\gamma_c}{a_d}=\tfrac{d}{c}\cdot \tfrac{a_b\gamma_b}{a_c}$.  
\end{enumerate}
\end{prop}

\begin{proof}
For (a) start with $u_c\cdot u_{[b\cln c]}\tfrac{\gamma_d}{u_b}=\tfrac{c}{b} u_b\cdot \tfrac{\gamma_d}{u_b}=\tfrac{c}{b}\gamma_d=u_c\cdot \tfrac{c}{b}\frac{\gamma_d}{u_c}$,
where the first equality is by Proposition~\ref{pr:u-props}.
The desired identity follows from the fact that we are in the range where multiplication by $u_c$ is an isomorphism.

For (b) we write
\[ u_{[d\cln b]}\tfrac{a_c\gamma_c}{a_d}=u_{[d\cln c]}u_{[c\cln b]}\cdot \tfrac{a_c\gamma_c}{a_d}=u_{[c\cln b]}\cdot \tfrac{d}{c}\gamma_c =\tfrac{d}{c}\cdot \tfrac{a_b\gamma_b}{a_c}.
\]
The first equality is by Proposition~\ref{pr:u-props}(2), the second by Proposition~\ref{pr:u-gamma}(b), and the third by Proposition~\ref{pr:u-gamma}(c).
\end{proof}

We now wrap up our unfinished proof:

\begin{proof}[Proof of Proposition~\ref{pr:neg-cone-products}, Part 2] First consider the remaining case of (3) and fix $s\leq i \leq k$. Let $\und{b}'=(b_1,\dots, b_{s-1})$ and $\und{b}''=(b_{s+1}, \dots, b_k)$. We have the diagram
\[
\xymatrixcolsep{5pc}\xymatrix{ 
H^{3-\lambda_{b_s}-
\lambda_{b_i}} \ar[r]^{u_{[b_i:b_{s-1}]}} & H^{3-\lambda_{b_{s-1}}-
\lambda_{b_s}}\\
H^{2s+1-\lambda_{\und{b}}} \ar[u]_{\cong}^{u_{\und{b'}}a_{\und{b''}\setminus b_{i}}} \ar[r]^{u_{b_i}} & H^{2s-1-\lambda_{\und{b}}+\lambda_{b_i}}\ar[u]_{u_{\und{b'}\setminus b_{s-1}}a_{\und{b''}\setminus b_{i}}}^{\cong}.
}
\]
This diagram commutes because $u_{[b_i:b_{s-1}]}u_{b_{s-1}}=\frac{b_{s-1}}{(b_{s-1},b_i)}u_{b_i} = u_{b_i}$ by Proposition~\ref{pr:u-props}. 
By Proposition~\ref{pr:u-gamma-2}(b) we have
\begin{equation}
\label{eq:temp}
u_{[b_i:b_{s-1}]}\cdot \tfrac{a_{b_s}\gamma_{b_s}}{a_{b_{i}}} = 
\tfrac{b_i}{b_s} \tfrac{a_{b_{s-1}}\gamma_{b_{s-1}}}{a_{b_s}}.
\end{equation}
The commutative square then gives us
\[
u_{b_i} \cdot \tfrac{a_{b_s}\gamma_{b_s}}{u_{\langle b_1,\ldots,b_{s-1}\rangle} \cdot a_{\langle b_{s+1},\ldots, b_i,\ldots b_k\rangle}} = \tfrac{b_i}{b_s} \cdot \tfrac{a_{b_{s-1}}\gamma_{{b_{s-1}}}}{u_{\langle b_1,\ldots,b_{s-2}\rangle} \cdot a_{\langle b_s,\ldots,
\hat{b_i},\ldots, {b_k}\rangle}}.
\]
because the vertical maps take the fractions on the two sides to the corresponding fractions that appear in (\ref{eq:temp}).

For the remaining case of (4), first suppose $2\leq s<k-1$ and fix $1\leq i < s$. Observe $a_{b_i}=u_{[b_i:b_s]}u_{[b_s:b_{s+1}]}a_{b_{s+1}}$ by Proposition~\ref{pr:u-props}. Compute as follows:

\begin{align*}
u_{[b_i\cln b_s]}u_{[b_s\cln b_{s+1}]}\tfrac{a_{b_s}\gamma_{b_s}}{u_{\langle b_1,\ldots,b_{s-1}\rangle} a_{\langle b_{s+1},\ldots,b_k\rangle}} &= u_{[b_i\cln b_s]}\cdot \tfrac{b_{s+1}}{b_s}\cdot \tfrac{\gamma_{b_{s+1}}}{u_{\langle b_1,\ldots,b_{s-1}\rangle}{a_{\langle b_{s+2},\ldots,b_k\rangle}}}
\\
&= \tfrac{b_{s+1} b_s}{b_s b_i}\cdot \tfrac{\gamma_{b_{s+1}}}{u_{\langle b_1,\ldots,\widehat{b}_i,\ldots,b_s\rangle} a_{\langle b_{s+2},\ldots,b_k\rangle}}.
\end{align*}
The two equalities are by Propositions~\ref{pr:u-subtips}(d) and \ref{pr:u-gamma-2}(a), respectively.  Now multiply both sides by $a_{b_{s+1}}$ to obtain
\[ a_{b_i}\cdot
\frac{a_{b_s}\gamma_{b_s}}{u_{\langle b_1,\ldots,b_{s-1}\rangle} a_{\langle b_{s+1},\ldots,b_k\rangle}} =
\frac{b_{s+1}}{b_i}\cdot 
\frac{a_{b_{s+1}}\gamma_{b_{s+1}}}{u_{\langle b_1,\ldots,\widehat{b}_i,\ldots,b_s\rangle} a_{\langle b_{s+2},\ldots,b_k\rangle}}
\]

Note if $s=k-1$, then multiplication by $a_{b_i}$ for any $1\leq i \leq k$ will be zero for degree reasons. Thus the final case to consider is $1\leq i=s\leq k-1$ and $k\geq 3$ (if $k=2$, the product would again vanish for degree reasons). Then we just use $a_{b_s}=a_{b_{s+1}}u_{[b_s:b_{s+1}]}$ as in the second step above to get the desired identity.
\end{proof}

\subsection{Black magic for the negative cone}
\label{se:magic}
Here we give a heuristic procedure for working with elements in the torsion part of the negative cone.  While one can verify that this works by translating it into our more rigorous treatment from previous sections, we stress that we do not know of an explanation for {\it why\/} it works.  

Introduce a formal symbol $\Omega$ of cohomological degree one.  That is, we will formally regard $\Omega$ as belonging to the group $H^1(\pt)$ even though this group is actually zero.  
For every $k\geq 2$, every divisor string $b_1,\ldots,b_k$, and every $1\leq s<k$ there is an element
\begin{equation}
\label{eq:Omega}
\tfrac{\Omega}{u_{b_1}\cdots u_{b_s} a_{b_{s+1}}\cdots a_{b_k}}\in H^{2s+1-\lambda_{b_1}-\cdots-\lambda_{b_k}}
\end{equation}
(note that the denominator must contain at least one $u$-class and at least one $a$-class), and these elements are subject to the following relations:

\begin{enumerate}[(1)]
\item The additive order of the above element is precisely $b_{s}$, and it generates the $\mZ$-module $\uH^{2s+1-\lambda_{\und{b}}}\iso \mZ/I_{b_{s}}$. 
\item {}[Inverse gold relation \#1] For $c|d$ one has
\[ u_d\cdot\tfrac{\Omega}{(u-\text{classes})u_ca_d(a-\text{classes})}=\tfrac{d}{c}\cdot \tfrac{\Omega}{(u-\text{classes})a_c(a-\text{classes})}.
\]
\item {}[Inverse gold relation \#2] For $c|d$ one has
\[
a_c\cdot \tfrac{\Omega}{(u-\text{classes})u_ca_d(a-\text{classes})}=\tfrac{d}{c} \cdot
\tfrac{\Omega}{(u-\text{classes})u_d(a-\text{classes})}.
\]
\end{enumerate}

To make sense of the inverse gold relations, the piece of the formula labelled ``$u$-classes'' is a product of $u_i$'s, and similarly for the ``$a$-classes'', with the only provision being that the terms in the denominator can be arranged
into the form of (\ref{eq:Omega}).  In each equation, the parenthesized terms are the same on the two sides of the equation---the only changes are to the unparenthesized pieces.  To see why these are called inverse gold relations, imagine taking the first factor on the left-hand-side and moving it into the denominator of the right-hand-side, then using the gold relation to compare the two denominators.  This is not permissible because the right-hand-side would then not be of the form (\ref{eq:Omega}), but this is rectified by rewriting the equations in the way given above.  

The translation into our rigorous development is
\[ \gamma_b=\tfrac{\Omega}{u_ba_b}.
\]
So for example, the class previously called $\tfrac{a_c\gamma_c}{u_ba_d}$ for $b|c|d$ would be written as $\tfrac{\Omega}{u_bu_ca_d}$ in the new notation.  
It is an informative exercise to take the relations developed in Proposition~\ref{pr:neg-cone-products}, rewrite them in the new notation, and check that they can all be derived from (1)--(3) above.  

\begin{remark}
It is tempting to introduce expressions as in (\ref{eq:Omega}) but without the criterion that $\und{b}$ is a divisor string, and then to produce a single ``inverse gold relation'' by using the gold relation in the denominator and moving the integral factors into the numerators as appropriate.  Sadly, though, we have not found a way to do this without running into contradictions.
\end{remark}


\section{Maps between linear sphere complexes}
\label{se:Lmaps}

We now move towards the study of the irregular region of $H^\star(\pt)$.  Since $H^{\lamu{d}-\lamu{c}+r}=[S^{\lamu{c}},S^{r+\lamu{d}}]$, this is equivalent to studying homotopy classes of maps between spheres where neither $\lambda_{\und{c}}$ nor $\lambda_{\und{d}}$ is just an integer.  Because we allow $1$'s to appear in $\und{c}$ and $\und{d}$, one can reduce to $r\in \{0,1\}$ here.  

Recall from Section~\ref{se:smallmodels} that the complex $L(\und{c})$ is a small and convenient model for $S^{\lamu{c}}$. So one is led to study $[L(\und{c}),\Sigma^r L(\und{d})]$.  In this section we develop tools for approaching this problem, concentrating most of our work on the case where $r=0$ and $\und{c}$ and $\und{d}$ have the same length.   
From this special case one already sees both the complexities and some of the interesting consequences. 

\smallskip

Before diving into the homological algebra, we need to introduce an operation for integers that will prove useful.  

\subsection{The colon operation for integers}
Recall that if $I$ and $J$ are ideals in a commutative ring $R$ then $I\cln J=\{r\in R\,|\, rJ\subseteq I\}$.  When $R=\Z$ and $x,y\in \Z_{>0}$ it will be convenient to just write $x\cln y$ for the positive generator of $(x)\cln (y)$.  Note that this is similar to writing $(x,y)$ (the gcd) for the positive generator of the ideal $(x,y)$.  The following facts are easy exercises:

\begin{prop} 
\label{pr:colon}
Let $x,y,z\in \Z_{>0}$. 
\begin{enumerate}[(a)]
\item $x\cln y=\tfrac{x}{(x,y)}$.
\item $x\cln y=x\cln (x,y)$.
\item $x\cln 1=x$, and $1\cln x=1$.
\item $x\cln y=xz\cln yz$.
\item $(x\cln y)\cln z=x \cln (yz)$.
\item If $(y,z)=1$ then $(x\cln y)z=xz\cln y$.  
\item $x\cln (y\cln z)=\frac{xz\cln y}{z\cln y}=\frac{(xy,xz)}{(y,xz)}$.
\item $\frac{x\cln y}{x\cln (y\cln z)}=\tfrac{(y,xz)}{(x,y)(y,z)}$.
\item $\frac{x\cln y}{x\cln (yz)}=\frac{(x,yz)}{(x,y)}=(x\cln y,z)$.
\item If $z|y$ then
$[x\cln \tfrac{y}{z},z]=z(x\cln y)$.
\end{enumerate}
\end{prop}

\begin{proof}
We mostly leave these to the reader, but will prove (g) to give the general idea.  For this just write
\[ x\cln (y\cln z)=\tfrac{x}{(x,\frac{y}{(y,z)})}=\tfrac{x(y,z)}{(x(y,z),y)}=\tfrac{(xy,xz)}{(xy,xz,y)}=\tfrac{(xy,xz)}{(y,xz)}
\]
(where in the second equality we multiplied numerator and denominator by $(y,z)$) and likewise
\[ \tfrac{xz\cln y}{z\cln y}=\tfrac{  \tfrac{xz}{(xz,y)}  } {\tfrac{z}{(z,y)}}= \tfrac{x(z,y)}{(xz,y)}=\tfrac{(xz,xy)}{(xz,y)}.
\]\end{proof}

The following result will be particularly useful:

\begin{cor}
\label{co:colon}
If $d_1|d_2$ then $d_1\cln c$ divides $d_2\cln c$ and 
$(\tfrac{d_2\cln c}{d_1\cln c})\cdot d_1 = \tfrac{d_2(c,d_1)}{(c,d_2)}$.
\end{cor}

\begin{proof}
Write $d_2=d_1u$.  Then by Proposition~\ref{pr:colon}(g) we have
\[ (d_1\cln c)\cdot (u\cln (c\cln d_1))=(ud_1)\cln c=d_2\cln c.
\]
Since $u\cln(c\cln d_1)$ is an integer, $d_1\cln c$ divides $d_2\cln c$.  Proposition~\ref{pr:colon}(g) also tells us that
\[ \bigl ( \tfrac{d_2\cln c}{d_1\cln c}\bigr )d_1 = 
\bigl (\tfrac{(ud_1)\cln c}{d_1\cln c}\bigr ) d_1=
\tfrac{ (uc,ud_1)}{(c,ud_1)}\cdot d_1=
\tfrac{(uc,d_2)}{(c,d_2)}\cdot d_1=
\tfrac{(d_1uc,d_1d_2)}{(c,d_2)}=\tfrac{(d_2c,d_1d_2)}{(c,d_2)}.
\]
\end{proof}

\subsection{Maps between $L$-complexes}
\label{se:L-maps}

Recall from (\ref{eq:free-map}) that every map $f\colon F_d\ra F_e$ can be written in the form $g_d\mapsto \sum_{i=0}^{(d,e)-1} m_i (t_*)^ip_*p^* g_e$ for unique $m_i\in \Z$, and that we denote this map as $\lur{m}$.  
The sum $\sum_i m_i$ appears frequently in arguments---we saw this already in Proposition~\ref{pr:lifting-lems-2}---so it will be useful to have a name to attach to this invariant: let us call it the \dfn{normity} of the map $\lur{m}$.  This invariant depends on the choice of free generators; for example, changing the sign on $g_d$ will change the sign of the normity.  But the absolute value of the normity is choice-independent. To see this, note that $F_d(\pt)\ra F_e(\pt)$ is (after choosing generators) a map $\Z\ra \Z$; so the absolute value of the image of $1$  
is independent of any choices. Using
the generators $(\pi_d)_*g_d$ and $(\pi_e)_*g_e$, one computes that the former is sent to the  product of $d\cln e$, the normity, and the latter (e.g., use $x=1$ in the left diagram of Proposition~\ref{pr:lifting-lems-1}(a) below). 

We will often use that the normity of $I\pi R\pi\colon F_d\ra F_e$ is equal to $(d,e)$.  This follows from the fact that $(\pi_d)_*g_d$ is sent by $I\pi R\pi$ to 
$(\pi_d)_* (\pi_d)^*(\pi_e)_* g_e=d\cdot (\pi_e)^*g_e$,
and so the normity is equal to $\tfrac{d}{d\cln e}=(d,e)$. \smallskip

Maps of complexes $L(\und{c})\ra L(\und{d})$ will consist of a collection of $\langle \und{m}\rangle$ maps that fit together in prescribed ways.  Before investigating this in detail let us record some elementary facts that will be helpful.

\begin{prop}
\label{pr:lifting-lems-1}
\mbox{}\par
\begin{enumerate}[(a)]
\item The following two diagrams commute:
\[ \xymatrix{
F_{x} \ar[dr]^{(\sum_i m_i) (y\cln z) I\pi R\pi} \ar[d]_{I\pi R\pi} \\
F_y \ar[r]_{\langle \und{m} \rangle} & F_z,}
\qquad\qquad
\xymatrix{
F_x \ar[r]^{\langle \und{m}\rangle} \ar[dr]_{(\sum_i m_i)(y\cln x) I\pi R\pi} & F_y\ar[d]^{I\pi R\pi} \\
& F_z.
}
\]
(A helpful mnemonic is that the colon in these formulas is always ``the middle index $\colon$ the index opposite the $I\pi R\pi$'').  
\item The square
\[ \xymatrix{
F_{c_2}\ar[r]^{\langle \und{m}\rangle}\ar[d]_{I\pi R\pi} & F_{d_2} \ar[d]^{I\pi R\pi} \\
F_{c_1} \ar[r]_{\langle \und{Q}\rangle} & F_{d_1}
}
\]
commutes if and only if $(\sum_i Q_i)\cdot (c_1\cln d_1)=(\sum_i m_i)\cdot (d_2\cln c_2)$.
\item The diagram
\[ \xymatrixcolsep{3.2pc}\xymatrix{
& F_{d_2}\ar[d]^{I\pi R\pi}\\
F_{c_1}\ar[r]_{Q\cdot I\pi R\pi} \ar@{.>}[ur] & F_{d_1}
}
\]
admits a lifting if and only if $d_2\cln c_1$  divides $Q$.
\end{enumerate}
\end{prop}

\begin{proof}
Parts (b) and (c) follow immediately from (a).  Part (a) is  just a calculation.  For example, in the left diagram the composite is given by
\[ g_x\mapsto (\pi_x)^*(\pi_y)_*g_y \mapsto (\pi_x)^*(\pi_y)_* \sum_i m_i (t_*)^ip_*p^* g_z=\sum_i m_i (\pi_x)^*(\pi_y)_* (t_*)^i p_*p^* g_z.
\]
The relation $\pi_y\circ t=\pi_y$ implies that the $t_*$ factors
can be dropped from the formula entirely.  This gives the $(\sum_i m_i)$ part of the answer.  Next use that $\pi_y\circ p_{([y,z]\ra y)}=\pi_z\circ p_{([y,z]\ra z)}$, so that
\begin{align*}
(\pi_x)^*(\pi_y)_*(p_{([y,z]\ra y)})_*(p_{([y,z]\ra z)})^* &=
(\pi_x)^*(\pi_z)_* (p_{([y,z]\ra z)})_*(p_{([y,z]\ra z)})^*\\
&=\tfrac{[y,z]}{z}\cdot (\pi_x)^*(\pi_z)_*\\
&= \tfrac{y}{(y,z)}\cdot (\pi_x)^*(\pi_z)_*.
\end{align*}
We have therefore shown $\lur{m}\circ I\pi R\pi=(\sum_i m_i)\cdot (y\cln z) I\pi R\pi$.  

This analysis is harder to write than it is to think through: staring at the diagram
\[ \xymatrixrowsep{1pc}\xymatrix{ & \Theta_{[y,z]} \ar[r]^p\ar[d]^p & \Theta_z \\
& \Theta_y \ar[d]^{\pi} \\
\Theta_x\ar[r]^{\pi} & \Theta_1
}
\]
is generally enough to recall the argument.  

The analysis for the second triangle in (a) is entirely similar.  
\end{proof}

\begin{remark}
Note that the map $\mZ\llra{I\pi_d} F_d$ may be interpreted as $I\pi_d R\pi_1\colon F_1\ra F_d$ (since $\pi_1=\id$), and as such is an example of an ``$I\pi R\pi$'' map.  Similarly, $R\pi_d$ is also an example of such a map.  
Some of our applications of Proposition~\ref{pr:lifting-lems-1} will be to these cases, which are when one of the indices on the $F_?$ modules is equal to $1$.  
\end{remark}

Now we can proceed to construct maps of complexes $L(\und{c})\ra L(\und{d})$, where $\und{c}$ and $\und{d}$ are divisor strings of the same length.  Building such a map from degree zero upwards, we must specify integers $w_0,w_1,w_2,\ldots$ and tuples $\und{m}^{(1)}, \und{m}^{(2)},\ldots$ for the maps as follows:
\begin{equation}
\label{eq:Lmaps}
\xymatrixcolsep{4.1pc}\xymatrix{
{} \ar[r] & F_{c_2}\ar[d]^{\lur{m^{(2)}}+w_2I\pi R\pi} \ar[r]^{\id-Rt} & F_{c_2}\ar[d]^{\lur{m^{(2)}}} \ar[r]^{I\pi R\pi} & F_{c_1}\ar[d]^{\lur{m^{(1)}}+w_1I\pi R\pi} \ar[r]^{\id-Rt} & F_{c_1}\ar[d]^{\lur{m^{(1)}}} \ar[r]^{I\pi R\pi} & \mZ \ar[d]_{w_0} \\
{} \ar[r] & F_{d_2}\ar[r]_{\id-Rt} & F_{d_2}\ar[r]_{I\pi R\pi} & F_{d_1}\ar[r]_{\id-Rt} & F_{d_1}\ar[r]_{I\pi R\pi} & \mZ. \phantom{\big |}
}
\end{equation}
Proposition~\ref{pr:lifting-lems-2}(b) implies that the squares with $\id-Rt$ must be of the above form, and they commute by construction.  Let $M_i$ be the normity of $\lur{m^{(i)}}$.  Then
by Proposition~\ref{pr:lifting-lems-1}(b) the remaining squares commute if and only if
\begin{equation}
\label{eq:M-u-conds}
M_i \cdot (d_i\cln c_i)= \Bigl ( M_{i-1} + w_{i-1}(c_{i-1},d_{i-1})\Bigr )\cdot (c_{i-1}\cln d_{i-1})
\end{equation}
for all $i\geq 1$, where $M_0$ is interpreted as $0$ and 
$c_0$ and $d_0$ are interpreted as $1$.  Note that the $i=1$ condition is simply $M_1\cdot \frac{d_1}{(c_1,d_1)}=w_0$.  

All of the $w_i$ and $\und{m}^{(i)}$ together feel like an intimidating amount of information, but we can quickly make a massive reduction.  The various parts of Propositions~\ref{pr:lifting-lems-1} and \ref{pr:lifting-lems-2} only ever refer to the normity of the $\und{m}$-sequences, never to any information beyond that.  This leads to the following:

\begin{prop}
\label{pr:chain-map-reduction}
Suppose we have two maps of complexes $L(\und{c})\ra L(\und{d})$, one specified by the data $(\und{m}^{(i)},w_i)$ and the other by $(\und{m'}^{(i)},w'_i)$.  If $w_i=w'_i$ and $\sum m^{(i)}=\sum (m')^{(i)}$ for all $i$, then the two maps are chain homotopic.  
\end{prop}

\begin{proof}
By taking the difference of the two maps one reduces to the case of a map specified by data $(\und{m}^{(i)},w_i)$ where all $w_i=0$ and all $\sum m^{(i)}=0$.  One inductively constructs a chain homotopy and finds that every other map can be taken to be zero, whereas the remaining ones are automatically guaranteed by Proposition~\ref{pr:lifting-lems-2}(a).
\end{proof}

By Proposition~\ref{pr:chain-map-reduction} we can replace a chain map associated to the data $(\und{m}^{(i)},w_i)$ by the map with data $(M_ip_*p^*,w_i)$ where $M_i=\sum m^{(i)}$.  So to specify a chain homotopy class it will be enough to specify the integers $w_i$ and $M_i$ for all $i\geq 0$, subject to the conditions of (\ref{eq:M-u-conds}).  

\begin{prop}
\label{pr:L-map-null}
Suppose that $\und{c}$ and $\und{d}$ are divisor strings of the same length $k$.
The chain map specified by data $(M_i,w_i)$ is null homotopic if and only if
\begin{itemize}
\item $(c_1,d_1)|M_1$,
\item 
$(d_{i+1}\cln c_i)(d_i,c_i)=\bigl ( \frac{d_{i+1}\cln c_i}{d_i\cln c_i}\bigr )d_i$
divides $M_i+(c_i,d_i)w_i$ for $1\leq i\leq k-1$, and
\item $M_k+(c_k,d_k)w_k=0$.
\end{itemize}
(Note the second condition implies that $(c_i,d_i)|M_i$, for $1\leq i\leq k-1$.  In particular, the first condition is redundant when $k\geq 2$.)
\end{prop}

\begin{proof}
Let $K=L(\und{c})$ and $L=L(\und{d})$, and let $f\colon K\ra L$ be our map.  We attempt to construct a chain homotopy $s_i\colon K_i\ra L_{i+1}$ inductively, starting with $i=0$. These maps are shown in the diagram below:
\[ \xymatrixcolsep{4.1pc}\xymatrix{
{} \ar[r] & F_{c_2}\ar[d]^{f_4} \ar[r]^{\id-Rt} \ar@{.>}[dl]_{s_4} & F_{c_2}\ar[d]^{f_3}\ar@{.>}[dl]_{s_3} \ar[r]^{I\pi R\pi}& F_{c_1}\ar[d]^{f_2} \ar[r]^{\id-Rt} \ar@{.>}[dl]_{s_2}& F_{c_1}\ar[d]^{f_1} \ar[r]^{I\pi R\pi}\ar@{.>}[dl]_{s_1} & \mZ \ar[d]_{f_0}\ar@{.>}[dl]_{s_0} \\
{} \ar[r] & F_{d_2}\ar[r]_{\id-Rt} & F_{d_2}\ar[r]_{I\pi R\pi} & F_{d_1}\ar[r]_{\id-Rt} & F_{d_1}\ar[r]_{I\pi R\pi} & \mZ. \phantom{\big |}
}
\]
The form of each $f_i$ map is given in (\ref{eq:Lmaps}).

The map $s_0$ is a lifting across $I\pi R\pi$ of $f_0$. By Proposition~\ref{pr:lifting-lems-1}(c) (taking $\mZ=F_1$ and $f_0=w_0 I\pi_1R\pi_1$) the map $s_0$ exists if and only if $d_1|w_0$.  Since $w_0=M_1\tfrac{d_1}{(c_1,d_1)}$ by (\ref{eq:M-u-conds}), this is equivalent to $(c_1,d_1)|M_1$.

Assuming the above condition, we next consider $s_1$. This will need to be a lifting across $\id-Rt$ of the map $f_1-s_0I\pi R\pi$.  As in the proof of Proposition~\ref{pr:lifting-lems-2}(a), such a lifting exists if and only if $0=R\pi\circ (f_1-s_0 I\pi R\pi)$.
But on the right we have
\[ R\pi f_1-R\pi s_0 I\pi R\pi= f_0I\pi R\pi- f_0I\pi R\pi=0
\]
so the condition is automatically satisfied.  

Continuing in this way, $s_2$ will need to be a lifting across $I\pi_{d_1} R\pi_{d_2}$ of \[f_2-s_1(\id-Rt)=f_2-(\id-Rt)s_1=f_2-(f_1-s_0I\pi_1 R\pi_{c_1}).\] Observe $f_2-f_1=w_1I\pi_{d_1} R\pi_{c_1}$ using the the formulas for the $f_i$ maps from (\ref{eq:Lmaps}). By Proposition~\ref{pr:lifting-lems-1}(a), $s_0I\pi_{1} R\pi_{c_1}$ is a multiple of $I\pi_{d_1} R\pi_{c_1}$, and that multiple is the normity of $s_0$ times $1\cln d_1=1$. Finally note that $I\pi_{1}R\pi_{d_1}s_0=w_0$ implies the normity of $s_0$ equals $\tfrac{w_0}{d_1}=\tfrac{M_1}{(c_1,d_1)}$. Putting this all together we see we need a lift of
\[f_2-s_1(\id-Rt)=(w_1+\tfrac{M_1}{(c_1,d_1)})I\pi R\pi.\]
Proposition~\ref{pr:lifting-lems-1}(c) says the lifting exists if and only if $d_2\cln c_1$ divides $w_1+\tfrac{M_1}{(c_1,d_1)}$, in which case the normity of $s_2$ will equal
the quotient $\frac{1}{d_2\cln c_1} (w_1+\tfrac{M_1}{(c_1,d_1)}) = \frac{1}{(d_2\cln c_1)(c_1,d_1)} (M_1+w_1(c_1,d_1))$ by Proposition~\ref{pr:lifting-lems-1}(a).

The rest of the argument is exactly the same, repeating these two basic steps.  The condition for $s_i$ to exist when $i$ is odd is always satisfied automatically, but when $i$ is even it leads to a divisibility statement:
the normity of $s_{2i}$ needs to equal $\tfrac{1}{(d_{i+1}\cln c_i)(c_i,d_i)}\cdot (M_i+(c_i,d_i)w_i)$, and the divisibility statement is that this is an integer. The argument follows just as with $s_2$, and using the formula in \ref{eq:M-u-conds} to rewrite $M_{i-1}+w_{i-1}(c_{i-1},d_{i-1})$ in terms of $M_i$.  

At the top of the complex, when the $s$ terms run out, there is one extra condition for a triangle to commute: one needs $f_{2k}=s_{2k-1}(\id-Rt)$. The analysis here is exactly the same as in the generic `even' case, but here instead of a certain integer being divisible (in order to get a lifting) it must actually be equal to zero.  This is the $M_k+(c_k,d_k)w_k=0$ condition.
\end{proof}

\begin{remark}
The expression $M_i+(c_i,d_i)w_i$ that appears in the above proposition has a simple interpretation: this is precisely the normity of the map of complexes in degree $2i$. 
Here is another  interpretation:
\end{remark}

\begin{prop}
\label{pr:L-homology}
Let $L(\und{c})\ra L(\und{d})$ be the map of complexes given by the data $(\und{M},\und{w})$.  Then the map $H_{2i}L(\und{c})\ra H_{2i}L(\und{d})$ is (up to sign) multiplication by $(c_i\cln d_i)\cdot (M_i+w_i(c_i,d_i))$.  
\end{prop}

\begin{proof}
By Corollary~\ref{co:positive-cone-computation1} and
Proposition~\ref{pr:spheres-small-model}
we have calculated $H_{2i}L(\und{c})\iso \mZ/I_{c_{i+1}}$ for $i<k$, and $H_{2i}L(\und{c})\iso \mZ$ for $i=k$ (and likewise for $L(\und{d})$).  In both cases the element $\pi_*(g_{c_i})$ is a generator for the homology $\mZ$-module.  So we need to calculate the composite
\[ \mZ=F_1 \llra{I\pi} F_{c_i}\llra{f_{2i}} F_{d_i}
\]
where $f$ is our map of complexes.  By Proposition~\ref{pr:lifting-lems-1}(a)  this equals $(c_i\cln d_i)N\cdot I\pi$ where $N$ is the normity of $f_{2i}$.   But $N=M_i+w_i(c_i,d_i)$, and we are done.
\end{proof}

At this point we can describe a method for computing the group $[S^{\lambda_{\und{c}}},S^{\lambda_{\und{d}}}]$ when $\und{c}$ and $\und{d}$ have the same length $k$.  Let $Z$ be the set of tuples 
\[ (M_1,M_2,\ldots,M_{k},w_0,w_1,\ldots,w_k)\in \Z^{2k+1}
\]that satisfy the conditions (\ref{eq:M-u-conds}).
Since one of these relations is $w_0=M_1\tfrac{d_1}{(c_1,d_1)}$  we simply discard $w_0$ from now on.  
Then $Z$ will be a free abelian group of rank $k+1$.   Because of the nature of the relations in Proposition~\ref{pr:L-map-null} it is more convenient to change coordinates and set $M_i'=M_i+w_i(c_i,d_i)$.  So we are a looking at tuples $(M'_i,w_i)_{1\leq i\leq k}$ satisfying the relations
\[ \Bigr ( M_i'-w_i(c_i,d_i) \Bigr ) (d_i\cln c_i)
 = M'_{i-1}\cdot (c_{i-1}\cln d_{i-1})
\]
for $i\geq 2$.  

The conditions from Proposition~\ref{pr:L-map-null} give $k$ torsion conditions and one integral conditions.  Encode these in the map
\[ \Phi\colon Z \ra \Z/(c_1,d_1)\oplus \Bigl [ \bigoplus_{i=1}^{k-1} \Z/(\tfrac{d_{i+1}\cln c_i}{d_i\cln c_i}\cdot d_i) \Bigr ] \oplus \Z
\]
that sends $(\und{M}',\und{w})$ to $(\overline{M'}_1,\overline{M'}_1,\ldots,\overline{M'}_{k-1}, M'_k)$.
(Note that the first two coordinates are not identical, because the modular reduction is happening in different groups.  When $k\geq 2$ we can discard the $\Z/(c_1,d_1)$ summand, as previously remarked.)  The image of this map is $[S^{\lambda_{\und{c}}},S^{\lambda_{\und{d}}}]$.  In any particular case this image can be algorithmically computed (by reducing an appropriate matrix to Smith normal form), but  doing this in any kind of generality appears unwieldy.  

\begin{example}
\label{ex:H-onefold}
Let us look at the $k=1$ case.  
Then $Z$ is the set of pairs $(M_1',w_1)$, with no relations.  Our map $\Phi$ is
\[ Z\ra  \Z/(c_1,d_1)\oplus \Z, \quad (M'_1,w_1)\mapsto (\overline{M'}_1, M_1').
\]
For an element in the image to be torsion one needs $M_1'=M_1+(c_1,d_1)w_1=0$, which means $(c_1,d_1)|M_1$, and so $\overline{M'}_1=0$.  So the image contains no nonzero torsion elements and hence is isomorphic to $\Z$.  This recovers the computation $H^{\lambda_{d_1}-\lambda_{c_1}}\iso \Z$ from Proposition~\ref{pr:Sb_to_Sc}.
\end{example}

\begin{example}
\label{ex:H-twofold}
Next we look at $k=2$, so that $Z$ consists of $4$-tuples $(M'_1,M'_2,w_1,w_2)$ satisfying the linear condition
\[ (M_2'-w_2(c_2,d_2))\cdot (d_2\cln c_2) =M_1' \cdot (c_1\cln d_1).
\]
Our map is
\[ \Phi\colon Z\ra \Z/(\tfrac{d_2\cln c_1}{d_1\cln c_1}\cdot d_1)\oplus \Z, \qquad (\und{M}',\und{w})\mapsto (\overline{M'}_1,M_2').
\]
We must have $M_1'$ a multiple of $(d_2\cln c_2)\cln (c_1\cln d_1)$,
but then $w_2$ can be taken to be any integer and $M_2'$ is then determined.  In particular, $Z$ contains tuples where $M_2'$ is nonzero.  
The image of $\Phi$ is therefore the direct sum of $\Z$ and whatever torsion is in the image.  To contribute to the latter we must have $M_2'=0$ which implies $-w_2d_2=M_1'(c_1\cln d_1)$. Thus $M_1'$ is a multiple of $d_2\cln (c_1\cln d_1)$.  The torsion in $\im \Phi$ is therefore a cyclic group whose order is
\[ \tfrac{ \bigl ( \tfrac{d_2\cln c_1}{d_1\cln c_1} \bigr ) d_1}{d_2\cln (c_1\cln d_1)}=\tfrac{(d_2\cln c_1) (c_1,d_1)}{d_2\cln (c_1\cln d_1)}=
\tfrac{(c_1,d_1d_2)}{(c_1,d_2)}.
\]
where we have applied Proposition~\ref{pr:colon}(h) to simplify the colon expression in the second equality.  
So $H^{\lambda_{d_1}+\lambda_{d_2}-\lambda_{c_1}-\lambda_{c_2}}(\pt)\iso \Z \oplus \Z/(\tfrac{(c_1,d_1d_2)}{(c_1,d_2)})$.  
\end{example}

\begin{example}
For $k=3$ the group $Z$ consists of tuples $(M_1',M_2',M_3',w_1,w_2,w_3)$ satisfying the two conditions
\begin{align*}
& (M_2'-w_2(c_2,d_2))\cdot (d_2\cln c_2) 
=M_1' \cdot (c_1\cln d_1)  \\
& (M_3'-w_3(c_3,d_3))\cdot (d_3\cln c_3) =M_2' \cdot (c_2\cln d_2).
\end{align*}
The map $\Phi\colon Z \ra \Z/
(\tfrac{d_2\cln c_1}{d_1\cln c_1}\cdot d_1)\oplus\Z/
(\tfrac{d_3\cln c_2}{d_2\cln c_2}\cdot d_2)\oplus
\Z$ is given by $(\und{M'},\und{w})\mapsto (\overline{M'}_1,\overline{M'}_2, M_3')$.  The torsion part of the image consists of elements $(\overline{M'}_1,\overline{M'_2})$ where $M_2'$ is a multiple of $d_3\colon (c_2\cln d_2)$  and $d_2$ divides $M_1'(c_1\cln d_1)-M_2'(d_2\cln c_2)$.  Trying to give a closed form for the isomorphism type of this group seems to be a mess, even in the $\ell$-local setting.
\end{example}

The above methods do yield a complete calculation in the following special case:

\begin{prop}
Suppose that $\und{c}$ and $\und{d}$ are divisor strings of the same length $k$, and that $(d_id_{i-1},c_{i-1})=(d_i,c_{i-1})$ for $2\leq i\leq k$.  Then $H^{\lambda_{\und{d}}-\lambda_{\und{c}}}\iso \Z$.  [In particular, note that the conditions hold if $c_{i-1}|d_i$ for $2\leq i\leq k$.]
\end{prop}

\begin{proof}
First observe that our hypotheses imply that
\[ d_i\colon (c_{i-1}\cln d_{i-1})=\tfrac{(d_ic_{i-1},d_id_{i-1})}{(c_{i-1},d_id_{i-1})}=\tfrac{(d_ic_{i-1},d_id_{i-1})}{(c_{i-1},d_i)}=\Bigl (
\tfrac{d_i\cln c_{i-1}}{d_{i-1}\cln c_{i-1}}\Bigr )\cdot d_{i-1}.
\]
Here we used Proposition~\ref{pr:colon}(g) for the first equality, our hypothesis for the second, and Corollary~\ref{co:colon} for the third.  This identity will be important in our argument.   For convenience let $N_{i-1}$ denote this common number.  Since 
\[ N_{i-1}=\tfrac{d_i(c_{i-1},d_{i-1})}{(c_{i-1},d_i)}=(d_i\cln c_{i-1})\cdot (c_{i-1},d_{i-1})
\]
note that $(c_{i-1},d_{i-1})|N_{i-1}$.   It will be convenient to set $N_k=(c_k,d_k)$.  

Using the same notation as above, we have a free abelian group $Z$ and a map $\Phi\colon Z\ra \Z/(c_1,d_1)\oplus [\bigoplus_{i=1}^{k-1}\Z/N_i]\oplus \Z$ whose image is $H^{\lambda_{\und{d}}-\lambda_{\und{c}}}=[S^{\lambda_{\und{c}}},S^{\lambda_{\und{d}}}]$.  The torsion part of the image will consist of tuples $(\overline{M'}_1,\overline{M'_1},\ldots,\overline{M'_{k-1}},M_k')$ where $M'_k=0$ and the $M'_i$ integers satisfy the relations
\begin{equation*}
\tag{*}
(M_i'-w_i(c_i,d_i))\cdot (d_i\cln c_i) =M_{i-1}'\cdot (c_{i-1}\cln d_{i-1}).
\end{equation*}
We will show by descending induction that $M'_k=0$ implies that $M'_i$ is a multiple of $N_i$ for all $i$.  
For $i=k$ this is trivial, so assume $M'_i$ is a multiple of $N_i$ where $i\geq 2$.  Then we know $M'_i$ is also multiple of $(c_i,d_i)$. Using (*) this implies that $M_{i-1}'\cdot (c_{i-1}\cln d_{i-1})$ is a multiple of $(c_i,d_i)\cdot (d_i\cln c_i)=d_i$. Thus
\[ M_{i-1}' \in (d_i\cln(c_{i-1}\cln d_{i-1}))=(N_{i-1}),
\]
the latter equality being from the first paragraph.  This is what we wanted.

We conclude $M_k'=0$ implies that the other coordinates of $\Phi$ are all zero, which means that $\Phi$ has no torsion elements in its image.
\end{proof}

\begin{remark}[The $E$ operator]
\label{re:E-operator}
We leave the following as an extended exercise for the reader.
As above, fix divisor strings $\und{c}$ and $\und{d}$ of the same length $k$. The above analysis shows that a chain map $f\colon L(\und{c})\ra L(\und{d})$ represented by $(\und{M},\und{w})$ is torsion in $H^{\lamu{d}-\lamu{c}}$ if and only if $M_k+w_k(c_k,d_k)=0$.  This condition is equivalent to saying that $f_{2k}=A(\id-Rt)$ for some $A\colon F_{c_k}\ra F_{d_k}$, which in turn is equivalent to saying that $f$ is chain homotopic to a map that is zero in degree $2k$ (use the chain homotopy that is $A$ in the top degree and zero everywhere else).  So the torsion subgroup of $H^{\lamu{d}-\lamu{c}}$ is isomorphic to the set of chain homotopy classes of maps $L(\und{c})\ra L(\und{d})$ that are zero in the top degree.

Now assume that $\und{c}$ and $\und{d}$ are the initial segments of divisor strings $\und{c}'$ and $\und{d}'$ of length $k+1$.  A map $f\colon L(\und{c})\ra L(\und{d})$ that is zero in degree $2k$ readily extends to a map $Ef\colon L(\und{c}')\ra L(\und{d}')$ just by defining $Ef$ to be zero in degrees $2k+1$ and $2k+2$.  This extension respects chain homotopies, so that one obtains a map
\[ E\colon \tors H^{\lamu{d}-\lamu{c}}\ra \tors H^{\lamu{d'}-\lamu{c'}}.
\]
An easy argument using the techniques we have developed shows that this is actually injective.  While we do not explore this operator further here, it comes up briefly in Section~\ref{se:ell^2}.
\end{remark}

\subsection{Calculating the integral image}
As an application of our previous work in this section, we can completely calculate the image of $H^\star(\pt;\mZ)\ra H^\star(\pt;\mQ)$.  Recall that the target is $\Q[u_d^{\pm 1}: d|n, d<n]$.  Due to the multigrading, the image will be generated by certain integral multiples of the classes $\frac{u_{d_1}\cdots u_{d_s}}{u_{c_1}\cdots u_{c_k}}$, for every choice of divisors $\und{d}$ and $\und{c}$.  Our task is to determine these integer factors precisely.  Note that since $u_1=1$ we can append ones to the $c$'s or $d$'s without changing the problem, and for this reason can always assume that $\und{c}$ and $\und{d}$ have the same length.

Let $\und{c}=(c_1,\ldots,c_k)$ and $\und{d}=(d_1,\ldots,d_k)$ be two lists of positive integers.  
It will be useful to picture these as lined up next to each other.  By convention set $c_0=d_0=1$.  
Recursively define the following integers:
\[ Y_0=c_{0}, \quad Y_j=\Bigl ((c_{j}Y_{j-1}\cln d_j),c_{j}\Bigr ).
\]
This process terminates at $Y_k$, 
and we set $Y_{\und{c},\und{d}}=Y_k$.  

As this is a bit challenging to parse, let us look at specific examples:
\begin{align*}
k=1: & \quad Y_{\und{c},\und{d}}=(c_1c_{0}\cln d_1,c_1)=(c_1\cln d_1,c_1)=(c_1\cln d_1) \\
k=2: & \quad Y_{\und{c},\und{d}}= \biggl ( \Bigl( c_2(c_{1}\cln d_1,c_{1})\Bigr )\cln d_2,c_2\biggr ).
\end{align*}

Even when the $c_i$ and $d_i$ are all powers of a fixed prime $\ell$, we do not know a simple description of the integer $Y_{\und{c},\und{d}}$.  However, we do have a non-recursive description.
In the following array, consider paths that start at $1$ and where each step is either a right move (R), up-right move (UR), or down move (D):
\[ \xymatrixrowsep{1pc}\xymatrix{
 & d_1 & d_2 & \cdots & d_k \\
1 & c_1 & c_2 & \cdots & c_k
}
\]
Say that a path is \dfn{allowable} if it starts at $1$ and uses the above kinds of steps.
To each allowable path, associate the integer consisting of the product of all the elements encountered along the path.  For example, $c_1c_2\cdots c_k$ and $d_1d_2c_2c_3d_4d_5\cdots d_k$ are two such integers. We have the following formula: 

\begin{prop}
\label{pr:M-direct}
$Y_{\und{c}, \und{d}}=\tfrac{(E)}{(E,F)} $ where $E$ is the set of all integers corresponding to allowable paths that end at $c_k$, $F$ is the set of all integers corresponding to allowable paths that end at $d_k$, and all parentheses indicate gcds.  
\end{prop}

As an example, when $k=2$ we have $E=(c_1c_2,d_1c_1c_2,d_1d_2c_2,c_1d_2c_2)=(c_1c_2,d_1d_2c_2)$ and $F=(c_1d_2,d_1c_1d_2,d_1d_2)=(c_1d_2,d_1d_2)$ and so
\[ Y_{(c_1,c_2),(d_1,d_2)}=  \tfrac{(c_1c_2,\ c_2d_1d_2)}{(c_1c_2,\ c_1d_2,\ d_1d_2)}.
\]

\begin{proof}[Proof of Proposition~\ref{pr:M-direct}]
Recall that $Y_0=(1)$ and $Y_j=((c_jY_{j-1}\cln d_j),c_j)$ for all $j\geq 1$.  Let $E_j$ and $F_j$ be as in the statement of the proposition but for the truncated sequences $(c_1,\ldots,c_j)$ and $(d_1,\ldots,d_j)$.   We will prove by induction that $Y_j=\frac{(E_j)}{(E_j,F_j)}$.  
The base case $j=0$ is trivial.  Assuming the result for $Y_{j-1}$ we have
\[
c_jY_{j-1}\colon d_j=\tfrac{ \tfrac{c_j(E_{j-1})}{(E_{j-1},F_{j-1})}}{
 \Bigl ( \tfrac{c_j(E_{j-1})}{(E_{j-1},F_{j-1})}, d_{j} \Bigr )}
 =\frac{(c_jE_{j-1})}{(c_jE_{j-1},d_jE_{j-1},d_jF_{j-1})}
\]
and therefore
\begin{align*}
Y_j=(c_jY_{j-1}\cln d_j,c_j)&=
\Bigl ( \tfrac{(c_jE_{j-1})}{(c_jE_{j-1},d_jE_{j-1},d_jF_{j-1})},\ c_j\Bigr )\\
& =
\frac{(c_jE_{j-1},c_{j}^2E_{j-1},c_jd_jE_{j-1},c_jd_jF_{j-1})}
{(c_jE_{j-1},d_{j}E_{j-1},d_jF_{j-1})}
\\[0.05in]
&=\frac{(c_jE_{j-1},c_jd_jF_{j-1} )}{(c_jE_{j-1},d_jE_{j-1},d_jF_{j-1})}.
\end{align*}
In the last expression, a brief inspection shows that the elements in the numerator have the same gcd as the set $E_j$ and the elements in the denominator have the same gcd as $E_j\cup F_j$.
\end{proof}

\begin{cor}
\label{co:M=1}
Suppose that $\und{c}$ and $\und{d}$ have the same length $k$.
Then $Y_{\und{c},\und{d}}=1$ if either of the following two conditions holds:
\begin{enumerate}[(1)]
\item For each $1\leq j\leq k$ one has $c_jc_{j+1}\cdots c_k|d_jd_{j+1}\cdots d_k$, and
\item $c_{k-1}c_k|d_k$.
\end{enumerate}
\end{cor}

\begin{proof}
Consider an allowable path $P$ that ends in $d_k$.  The tail of this path will consist of some $c_i$ (possibly $c_0=1$) followed by $d_{i+1}d_{i+2}\cdots d_k$.  Let $P'$ be the path obtained from $P$ by replacing this tail with $c_ic_{i+1}\cdots c_k$.  If assumption (1) holds, the integer for $P'$ divides the integer for $P$.  So (using notation as in Proposition~\ref{pr:M-direct}) every element of $F$ is a multiple of an element of $E$,  therefore $(E,F)=(E)$ and 
$Y_{\und{c},\und{d}}=1$.

When assumption (2) holds, we argue similarly.  Any path $P$ will end in either $c_{k-1}d_k$ or $d_{k-1}d_k$.  If we construct $P'$ by replacing the ending with $c_{k-1}c_k$ in the first case and $d_{k-1}c_{k-1}c_k$ in the second, then assumption (2) implies that the integer for $P'$ divides the integer for $P$ (note that the assumption implies $c_k|d_k$). 
\end{proof}

\begin{remark}
\label{re:counter}
If $\ell$ is a prime then $\und{c}={(\ell^3,\ell^3,\ell^4)}$ and $\und{d}=(\ell,\ell,\ell^7)$ is an example where assumption (2) holds but (1) does not. 
\end{remark}
  
Say that an element of $H^\star(\pt;\mQ)$ is \dfn{integral} if it lies in the image of $H^\star(\pt;\mZ)$. We have the following results about integral elements:

\begin{prop}
\label{pr:integral-classes}
Let $\und{c}=(c_1,\ldots,c_k)$ and $\und{d}=(d_1,\ldots,d_k)$ be two lists of divisors of $n$.  Write $u_{\und{c}}=\prod_i u_{c_i}$ and likewise for $u_{\und{d}}$.  
\begin{enumerate}[(a)]
\item If $\und{c}$ and $\und{d}$ are divisor strings then a multiple $r\cdot \frac{u_{\und{d}}}{u_{\und{c}}}$ is integral  if and only if $r$ is a multiple of $Y_{\und{c},\und{d}}$.  
\item For arbitrary $\und{c},\und{d}$, a multiple $r\cdot \frac{u_{\und{d}}}{u_{\und{c}}}$ is integral 
if and only if $r$ is a multiple of $Y_{\und{c'},\und{d'}}$, where $\und{c}'$ and $\und{d}'$ are the divisors strings associated to $\und{c}$ and $\und{d}$.  
\end{enumerate}
\end{prop}

Before giving the proof of the above result, we discuss some examples.  Remember that we implicitly add $u_1$ terms in order to make the number of $c$'s and $d$'s coincide.
Using Proposition~\ref{pr:M-direct} we find that for divisor strings $\und{c}$ and $\und{d}$:

\begin{itemize}
\item 
$r\cdot \tfrac{u_d}{u_c}$ is integral if and only if $r$ is a multiple of $\tfrac{c}{(c,d)}=c\cln d$.  

\item
$r\cdot \tfrac{u_d}{u_{c_1}u_{c_2}u_{c_3}}$ is integral if and only if $r$ is a multiple of $\tfrac{(c_2c_3,dc_3)}{(c_2c_3,d)}$.  

\item
$r\cdot \tfrac{u_{d_1}u_{d_2}}{u_{c_1}u_{c_2}}$ is integral if and only if $r$ is a multiple of $\tfrac{(c_1c_2,d_1d_2c_2)}{(c_1c_2,c_1d_2,d_1d_2)}$.
\end{itemize}

Finally, as an even more specific example, 
$r\cdot \tfrac{u_2u_6}{u_4}=r\cdot \tfrac{u_2u_6}{u_1u_4}$ is integral if and only if $r$ is a multiple of
$\tfrac{4}{(4,6)}=2$.

It is tempting to try to classify precisely which elements $\tfrac{u_{\und{d}}}{u_{\und{c}}}$ are integral, but this seems not to have a nice answer.  The following result gives some idea of the landscape:

\begin{prop}\mbox{}\par
\begin{enumerate}[(a)]
\item $\tfrac{u_d}{u_c}$ is integral if and only if $c|d$.
\item The element $\tfrac{u_{d_1}u_{d_2}}{u_{c_1}u_{c_2}}$ is integral if and only if $[c_1,c_2]\Bigl |[d_1,d_2]$ and $c_1c_2|d_1d_2$.
\item Let $\und{c}$ and $\und{d}$ be divisor strings of length $3$.  Then $\tfrac{u_{d_1}u_{d_2}u_{d_3}}{u_{c_1}u_{c_2}u_{c_3}}$ is integral if and only if for all primes $\ell|c_3$ either
\[ \text{\parbox{3.5in}{ $c_1(\ell)\leq (d_1d_2)(\ell)$, and  $c_3(\ell)\leq d_3(\ell)$, and \\$(c_2c_3)(\ell)\leq (d_2d_3)(\ell)$ and $(c_1c_2c_3)(\ell)\leq (d_1d_2d_3)(\ell)$}}
\]
OR
\[ \text{\parbox{3.5in}{ $(d_1d_2)(\ell)< c_1(\ell)$, and $(d_1d_2c_3)(\ell)\leq (c_1d_3)(\ell)$, and
\\ $(d_1c_2c_3)(\ell)\leq (c_1d_3)(\ell)$ and $(c_2c_3)(\ell)\leq d_3(\ell)$.}
}
\]
(Note that Remark~\ref{re:counter} gives an example showing that the evident generalization of (a) and (b) does not hold 
here.)
\end{enumerate}
\end{prop}

\begin{proof}
Part (a) is immediate, as $Y_{c,d}=\tfrac{c}{(c,d)}$ and this equals $1$ precisely when $c|d$.  

For (b), replacing $\und{c}$ and $\und{d}$ with their derived divisor strings does not affect either integrality or the lcm or the product, so let us do that.  Proposition~\ref{pr:M-direct} plus a little arithmetic shows $Y_{\und{c},\und{d}}=(\tfrac{c_1c_2}{(c_1c_2,c_1d_2,d_1d_2)},c_2)$.  This equals $1$ if and only if for every prime $\ell|c_2$ one has $\nu_\ell(c_1c_2)\leq \nu_\ell(c_1d_2)$ and $\nu_\ell(c_1c_2)\leq \nu_\ell(d_1d_2)$.  The first is equivalent to $\nu_\ell(c_2)\leq \nu_\ell(d_2)$, and this holding for all primes $\ell|c_2$ is exactly the condition $c_2|d_2$.    Since $c_1|c_2$, the primes dividing $c_2$ are exactly the primes dividing $c_1c_2$; so the second inequality holding for all primes $\ell|c_2$ is equivalent to $c_1c_2|d_1d_2$.  These are the required conditions.    

Finally, for part (c) we use Proposition~\ref{pr:M-direct} plus some arithmetic (namely the move $(x/y,c)=(x,cy)/y$) 
to obtain
\[ Y_{\und{c},\und{d}}=\Bigl ( \tfrac{Q}{(Q,c_1c_2d_3,c_1d_2d_3,d_1d_2d_3)}, c_3\Bigr )
\]
where
$Q=(c_1c_2c_3,d_1d_2c_2c_3)=c_2c_3(c_1,d_1d_2)$.  
The condition $Y_{\und{c},\und{d}}=1$ is equivalent to saying that for every prime $\ell|c_3$, the $\ell$-adic valuation of $\tfrac{Q}{(Q,c_1c_2d_3,c_1d_2d_3,d_1d_2d_3)}$ equals zero.  But $\nu_\ell(Q)=\nu_\ell(c_2c_3)+\min\{\nu_\ell(c_1),\nu_\ell(d_1d_2)\}$ and we are asking that
\begin{equation}
\label{eq:nu-condition}
\nu_\ell(Q)\leq \min\{ \nu_\ell(c_1c_2d_3), \nu_\ell(c_1d_2d_3), \nu_\ell(d_1d_2d_3)\}.
\end{equation}
The conditions from the statement of the proposition 
just amount to analyzing the two different possibilities for the term $\min\{\nu_\ell(c_1),\nu_\ell(d_1d_2)\}$ and how this influences (\ref{eq:nu-condition}). 
\end{proof}

We now give the proof of the main integrality result:

\begin{proof}[Proof of Proposition~\ref{pr:integral-classes}]
The groups $H_{2k}(S^{\lambda_{\und{c}}})$ and $H_{2k}(S^{\lambda_{\und{d}}})$ are generated by $u_{\und{c}}$ and $u_{\und{d}}$, respectively.  Consider the map
\[ [S^{\lambda_{\und{c}}}, S^{\lambda_{\und{d}}}] \ra \Hom\Bigl (H_{2k}(S^{\lambda_{\und{c}}}),H_{2k}(S^{\lambda_{\und{d}}}) \Bigr )\iso \Hom(\Z,\Z)=\Z. 
\]
The image is precisely the ideal consisting of all coefficients $r$ for which $r\cdot \frac{u_{\und{d}}}{u_{\und{c}}}$ is integral.  
Our approach to computing this image will be to construct all possible maps between the spheres and calculate the induced maps on the top homology group.

We may replace $\und{c}$ and $\und{d}$ with their associated divisor strings without changing the homotopy type of the spheres, so (b) follows directly from (a).  From now on we assume that $\und{c}$ and $\und{d}$ are divisor strings.

As we saw in Section~\ref{se:L-maps}, a chain map $L(\und{c})\ra L(\und{d})$ can be prescribed by specifying integers $M_1,\ldots,M_k$ and $w_1,\ldots,w_k$ satisfying the conditions
\begin{equation}
\label{eq:M-u-version2}
M_j\cdot (d_j\cln c_j)= M_{j-1}\cdot (c_{j-1}\cln d_{j-1})+w_{j-1} c_{j-1}
\end{equation}
for $2\leq j \leq k-1$.  By Proposition~\ref{pr:L-homology}
the induced map on the top homology is multiplication by
$(c_k\cln d_k)M_k + c_kw_k$. So our first goal is to show this is a multiple of $Y_{\und{c}, \und{d}}=Y_k$. Recall the integers $Y_j$ are defined recursively by 
\[
Y_0=1, \quad Y_j=((c_jY_{j-1}\cln d_j), d_j).
\]
In particular, note $Y_1$ simplifies to $Y_1=(c_1\cln d_1)$. 

All we do now is analyze the conditions (\ref{eq:M-u-version2}) inductively.  The first says that $M_2$ is in the ideal
\[ ((c_1\cln d_1),c_1)\cln (d_2\cln c_2)=(c_1\cln d_1)\cln (d_2\cln c_2)=Y_1\cln (d_2\cln c_2).
\]
Using this fact, the next condition then says that $M_3$ is in the ideal
\begin{align*}
\Bigl (  (c_2\cln d_2)((c_1\cln d_1)\cln (d_2\cln c_2)), c_2 \Bigr )\cln (d_3\cln c_3)&=\Bigl ( ((c_1\cln d_1)c_2)\cln d_2,c_2 
\Bigr ) \cln (d_3\cln c_3)\\
&=Y_2\cln (d_3\cln c_3)
\end{align*}
where in the second equality we have used Proposition~\ref{pr:colon}(g).  Proceeding inductively, we find that $M_j$ is in the ideal $Y_{j-1}\cln (d_j\cln c_j)$ for all $j\geq 2$.  

Finally the induced map on $H_{2k}$ is multiplication by
$(c_k\cln d_k)M_k+c_kw_k$, which can be any element of the ideal
\[ \Bigl (    (c_k\cln d_k)(Y_{k-1}\cln (d_k\cln c_k)),c_k \Bigr )=((Y_{k-1}c_k)\cln d_k,c_k)=Y_k=Y_{\und{c}, \und{d}}
\]
where in the first equality we have again used Proposition~\ref{pr:colon}(g) to simplify the colon expression, and the last equality is from the definition of the $Y_j$ numbers. This shows that if $r\cdot u_{\und{d}}/u_{\und{c}}$ is integral then  $r$ is a multiple of $Y_{\und{c}, \und{d}}$. But in fact all of these steps are reversible. As long as $M_k$ is chosen to be in $Y_{k-1}\cln (d_k\cln c_k)$, then one can recursively choose $M_j$ and $w_j$ (descending in $j$) in order to make a map of complexes. Thus any multiple of $Y_{\und{c}, \und{d}}$ gives rise to such an integral class as well.
\end{proof}

\subsection{The classification of units}
We end this section by classifying the units in $H^\star(pt;\mZ)$. We start with a lemma:
\begin{lemma}
\label{le:M=M}
Suppose that $c_1,\ldots,c_k$ and $d_1,\ldots,d_k$ are divisor strings.
\begin{enumerate}[(a)]
\item  If $Y_{\und{c},\und{d}}=1$ then $c_k|d_k$.
\item If $Y_{\und{c},\und{d}}=1=Y_{\und{d},\und{c}}$ then
$\und{c}=\und{d}$.  
\end{enumerate}
\end{lemma}

\begin{proof}
For (a),
it suffices to prove the result after $\ell$-localization for each prime $\ell$.  So we can fix $\ell$ and replace each $c_i$ and $d_i$ with their $\ell$-adic parts $c_i(\ell)$ and $d_i(\ell)$.

Using the setup from Proposition~\ref{pr:M-direct}, if $P$ is an allowable path write $[P]$ for its associated integer.  Say that path $P$ is a divisor of path $Q$ if $[P]$ divides $[Q]$. Since all of the integers are powers of $\ell$, the condition $Y_{\und{c},\und{d}}=1$ is equivalent to saying that every allowable path that ends at $d_k$ is a multiple of some allowable path that ends at $c_k$.  Let $\sigma$ be the smallest allowable path that ends at $c_{k-1}$, and let $\sigma'$ be the concatenation of $\sigma$ with $d_k$.    Then $\sigma'$ is a multiple of some allowable path $S$ that ends at $c_k$, and such a path must necessarily end in $c_{k-1}c_k$.  So the part of $S$ that ends at $c_{k-1}$ is a multiple of $\sigma$, by our choice of $\sigma$. Hence $[S]$ is a multiple of $[\sigma]c_k$.  
But $[\sigma']=[\sigma]d_k$ is a multiple of $[S]$, and so $c_k|d_k$.

For (b), the proof is by induction on $k$.  For the case $k=1$, if $\tfrac{c}{(c,d)}=1=\tfrac{d}{(c,d)}$ then $c|d$ and $d|c$, so $c=d$.  For the general case, it follows from (a) (applied twice) that $c_k=d_k$.  
Now use that if $c_k=d_k$ then $Y_{\und{c},\und{d}}=Y_{(c_1,\ldots,c_{k-1}),(d_1,\ldots,d_{k-1})}$. This is easy to prove directly (for divisor strings) using Proposition~\ref{pr:M-direct}, but it is also an immediate consequence of Proposition~\ref{pr:integral-classes}.  Finally if $\und{c}'$ and $\und{d}'$ denote the strings with the last element removed then we now have $Y_{\und{c}',\und{d}'}=1=Y_{\und{d}',\und{c}'}$ and so induction gives that $\und{c}'=\und{d}'$.
\end{proof}

\begin{prop}
\label{pr:units}
Every homogeneous unit in $H^\star(\pt;\mZ)$ is, up to sign, a product of the $\chi_{a,b}$-classes and their inverses.
\end{prop}

\begin{proof}
Let $u\in H^{\lambda_{\und{d}}-\lambda_{\und{c}}+m}(\pt;\mZ)$ be a unit, for some lists of divisors $d_1,\ldots,d_k$ and $c_1,\ldots,c_s$ and some $m\in \Z$.  First note that we must have $m=2s-2k$, otherwise the group would be torsion and so could not contain a unit. By adding ones to either the $\und{d}$-sequence or $\und{c}$-sequence (according to whether $m>0$ or $m<0$) we can assume $m=0$.  

By multiplying $u$ by appropriate $\chi$-classes we can convert $\und{d}$ and $\und{c}$ into their associated divisor strings.
So we further reduce to the case where $\und{d}$ and $\und{c}$ are divisor strings, and we will prove $u=\pm 1$.  

Applying the map $H^\star(\pt;\mZ)\ra H^\star(\pt;\mQ)$, Proposition~\ref{pr:integral-classes} says that $u$ will map to $P Y_{\und{c},\und{d}}\cdot \tfrac{u_{\und{d}}}{u_{\und{c}}}$ for some integer $P$.  Likewise, $u^{-1}$ will map to $QY_{\und{d},\und{c}}\cdot \tfrac{u_{\und{c}}}{u_{\und{d}}}$ for some integer $Q$.  Then $1=PQY_{\und{c},\und{d}}\cdot Y_{\und{d},\und{c}}$ and since the four factors are integers this can only happen if $P,Q\in \{1,-1\}$ and 
$Y_{\und{c},\und{d}}=1=Y_{\und{d},\und{c}}$.
By Lemma~\ref{le:M=M} the latter condition implies that $\und{c}=\und{d}$.  Hence $u\in H^0(\pt;\mZ)=\Z$, and being a unit we must have $u=\pm 1$.
\end{proof}

\section{The irregular region}
\label{se:irregular}

In this section we explore a few aspects of the irregular region of $H^\star(\pt)$.  
We prove Proposition~\ref{pr:irregular} from the introduction, which mostly amounts to pulling together ideas we have already developed.  We then expand further on some of those techniques, and close by examining in detail the case $n=\ell^2$.\medskip

\begin{proof}[Proof of Proposition~\ref{pr:irregular}]
For (a) consider the cofiber sequence $\Sigma^2 S^{\lamu{c'}}\llra{u_{c_1}} S^{\lamu{c}}\ra \mZ/I_{c_1}$ and apply $[\blank,\Sigma^k S^{\lamu{d}}]$.  One obtains the exact sequence
\[ \ucD(\Sigma^{-1}\mZ/I_{c_1},\Sigma^kS^{\lamu{d}}) \lla \uH^{\lamu{d}-\lamu{c'}+(k-2)} \llla{\cdot u_{c_1}} \uH^{\lamu{d}-\lamu{c}+k} \lla \ucD(\mZ/I_{c_1},\Sigma^k S^{\lamu{d}}).
\]
Since $\mZ/I_{c_1}$ has a free resolution of length $3$, if $k\geq 4$ then the groups on the two ends are zero.

The proof of (b) is similar, using the cofiber sequence $\Sigma^2 S^{\lamu{d'}}\ra S^{\lamu{d}} \ra \mZ/I_{d_1}$.  
This yields the exact sequence
\[ \ucD(S^{\lamu{c}},\Sigma^{k-1}\mZ/I_{d_1}) \ra \uH^{\lamu{d'}-\lamu{c}+(k+2)}\llra{\cdot u_{d_1}} \uH^{\lamu{d}-\lamu{c}+k} \lra \ucD(S^{\lamu{c}},\Sigma^k \mZ/I_{d_1}).
\]
If $k<0$ then the groups on the two ends will be zero, since $S^{\lamu{c}}$ is concentrated in non-negative degrees. 

For (c) we use the K\"unneth spectral sequence (see \cite[Section 6]{DHone}).  Note that $H^{\lamu{d}-\lamu{c}+k}=H_{-k}(S^{-\lamu{c}}\bbox S^{\lamu{d}})$.   Recall that the homology of $S^{-\lamu{c}}$ has an $I_{c_q}$ in degree $-2q$ and $\mZ/I_{c_{q-j}}$ in degree $-2q+2j-1$ for $1\leq j\leq q-1$.  Likewise, $S^{\lamu{d}}$ has homology only in even degrees, given by $\mZ/I_{d_j}$ in degree $2j-2$ ($1\leq j\leq s$) and $\mZ$ in degree $2s$.  In the K\"unneth spectral sequence we have
\begin{itemize}
\item box products $I_{c_q}\bbox \mZ/I_{d_j}$ and $I_{c_q}\bbox \mZ$, all contributing in even total degree,
\item $\uTor_2(I_{c_q},\mZ/I_{d_j})$ terms, also contributing in even total degree (all other $\Tor_i$ are zero here),
\item box products $\mZ/I_{c_i}\bbox \mZ/I_{d_j}$ and $\mZ/I_{c_i}\bbox \mZ$, all in odd total degree, and
\item $\uTor_3(\mZ/I_{c_i},\mZ/I_{d_j})$ terms, contributing to even total degree.  
\end{itemize}
Only the third types of terms are in odd total degree, and these are all in filtration degree zero of the spectral sequence.  So for $k$ odd we find that $\uH^{\lamu{d}-\lamu{c}+k}$ is a quotient of the sum of these terms, and therefore $\uH^{\lamu{d}-\lamu{c}+k}$ is generated by products of the negative $\gamma_{\und{c}}$-classes and the $(au)_{\und{d}}$-classes.  

For (d) let $\und{c}'=(c_1,\ldots,c_s)$.  The inclusion of complexes $S^{\lamu{c'}}\inc S^{\lamu{c}}$ is a model for the homotopy class $A=a_{c_{s+1}}a_{c_{s+2}}\cdots a_q$.  Let $Q$ denote the quotient complex, which is nonzero only in degrees $2s+1$ through $2q$.  Applying $[\blank,S^{\lamu{d}}]$ gives the exact sequence
\[ [\Sigma^{-1}S^{\lamu{c}},S^{\lamu{d}}]\lla [\Sigma^{-1}Q,S^{\lamu{d}}]  \llla{g} H^{\lamu{d}-\lamu{c'}} \llla{\cdot A} H^{\lamu{d}-\lamu{c}} \lla [Q,S^{\lamu{d}}]=0.
\]
The group on the right is zero for degree reasons, and the group on the left is torsion by Proposition~\ref{pr:rationalization}.  We also know already that the group $H^{\lamu{d}-\lamu{c'}}$ is isomorphic to the direct sum of $\Z$ and a torsion group, since by 
Proposition~\ref{pr:rationalization} the rationalization is $\Q$.  We will show that $[\Sigma^{-1}Q,S^{\lamu{d}}]\iso \Z$.  From this it follows from the exact sequence that the map labelled $g$ is nonzero and that the kernel is precisely the torsion subgroup of the domain, which will prove (d).   The calculation of $[\Sigma^{-1}Q,S^{\lamu{d}}]$ is easy to do by hand: chain maps look like
\[ \xymatrix{
\cdots \ar[r] & F_{c_{s+1}}\ar[r]^{\id-Rt} & F_{c_{s+1}}\ar[d]^f \\
&& F_{d_s} \ar[r]^{\id-Rt} & F_{d_s} \ar[r] & F_{d_{s-1}}\ar[r] & \cdots 
}
\]
and to be a map of complexes one must have $(\id-Rt)f=0$, which by Proposition~\ref{pr:lifting-lems-2}(b) is equivalent to $f=uI\pi R\pi$ for some $u\in \Z$.  There are no possible nontrivial chain homotopies, and so $[\Sigma^{-1}Q,S^{\lamu{d}}]\iso \Z$.

The proof of (e) is similar to (d) but with a couple of changes.  Let $\und{d'}=(d_1,\ldots,d_q)$.  We use the linear models for the spheres and the cofiber sequence $S^{\lamu{d'}}\inc S^{\lamu{d}} \ra Q$ where $Q$ is the quotient.  Note that $Q$ is nonzero only in degrees $2q+1$ through $2s$.  The inclusion on the left represents the homotopy class $A=a_{d_{q+1}}\cdots a_{d_s}$.  Applying $[S^{\lamu{c}},\blank]$ to the cofiber sequence gives
\[ [S^{\lamu{c}},\Sigma^{-1}S^{\lamu{d}}]\lra [S^{\lamu{c}},\Sigma^{-1}Q] \lra H^{\lamu{d'}-\lamu{c}} \llra{\cdot A} H^{\lamu{d}-\lamu{c}} \lra [S^{\lamu{c}},Q]=0.
\]
The group on the right is zero for degree reasons, and the group on the left is torsion by Proposition~\ref{pr:rationalization}.  So we will be done once we show that $[S^{\lamu{c}},\Sigma^{-1}Q]$ is isomorphic to $\mZ$.   But this is easy by direct calculation: a map of complexes looks like
\[ \xymatrix{
\cdots \ar[r] & F_{d_{q+1}}\ar[r]^{\id-Rt} & F_{d_{q+1}} \\
&& F_{c_q}\ar[r]^{\id-Rt}\ar[u]^{\lur{m}} & F_{c_q}\ar[r]  & \cdots
}
\]
and such a map is null homotopic if and only if there exist $h_0,h_1\colon F_{c_q}\ra F_{d_{q+1}}$ such that $\lur{m}=(\id-Rt)h_1+h_0(\id-Rt)$.  But this expression also equals $(\id-Rt)(h_1+h_0)$, and so the map is null homotopic if and only if $\lur{m}$ lifts 
across $\id-Rt$.  By Proposition~\ref{pr:lifting-lems-2} this lift exists if and only if $\sum_i m_i=0$.  This shows that $\lur{m}\mapsto \sum_i m_i$ gives an embedding $[S^{\lamu{c}},\Sigma^{-1}Q]\ra \Z$, and the image is all of $\Z$ since $\und{m}$ can be chosen arbitrarily.  
\end{proof}

\begin{cor}
\label{co:comm}
$H^\star(\pt)$ is a commutative ring.
\end{cor}

\begin{proof}
Let $x\in H^\star$ and $y\in H^\star$ be homogeneous elements.  If the fixed-dimension of either $|x|$ or $|y|$ is even then we know $xy=yx$ by (\ref{eq:skew-comm}).  So assume both of those fixed dimensions are odd.  By Proposition~\ref{pr:irregular}(c), we reduce to the case where each of $x$ and $y$ is a product of $au$-classes with negative $\gamma$-classes.  But the $au$-classes commute with everything (since $|a|$ and $|u|$ have even fixed dimension) and the product of two negative $\gamma$-classes is zero by Proposition~\ref{pr:neg-cone-products}.  So in this case $xy$ and $yx$ are both zero, hence $xy=yx$.  
\end{proof}

\subsection{Isotypical regions}

Proposition~\ref{pr:irregular} dealt with the idea of reducing a particular homogeneous component of $H^\star(\pt)$ to another one via multiplication by $u$- or $a$-classes.  
This idea can be further exploited in various directions.  The first of these says that multiplication by the largest $a$-class is almost always an isomorphism:

\begin{prop}
For any multi-index $\beta$, the map $\uH^\beta \llra{\cdot a_n} \uH^{\beta+\lambda_n}$ is an isomorphism except possibly when $|\beta|\in \{-2,-1,0\}$.  When $|\beta|=-2$ the map is injective and the image is the torsion subgroup of the target, and when $|\beta|=-1$ the map is surjective.
When $|\beta|=0$ the map is surjective, but also injective on torsion classes.
\end{prop}

\begin{proof}
We start with the cofiber sequence $\mZ\llra{a_n} S^{\lambda_n}\ra \Sigma X_n$, then box with $S^{\beta}$ to obtain
\[ S^\beta \ra S^{\beta+\lambda_n}\ra S^{1+\beta}\bbox X_n.
\]
This gives the long exact sequence
\[\cdots \lra \uH_1(S^{1+\beta}\bbox X_n)\lra \uH^\beta \llra{\cdot a_n} \uH^{\beta+\lambda_n} \lra \uH_0(S^{1+\beta}\bbox X_n) \lra \cdots
\]
Recall from Lemma~\ref{le:suspxb} that $S^{1+\beta}\bbox X_n\he S^{1+\dim \beta}\bbox X_n$, and so the homology is equal to $I_n$ in degree $1+\dim \beta$ and $\mZ$ in degree $2+\dim\beta$.  In particular, the groups on the two ends are zero if $0,1\notin \{1+\dim \beta,2+\dim \beta\}$, or equivalently $\dim \beta\notin \{0,-1,-2\}$.  

If $\dim \beta=-2$ then our exact sequence becomes
$0 \lra \uH^{\beta}\llra{\cdot a_n} \uH^{\beta+\lambda_n} \lra \mZ$.  But $\uH^{\beta}$ is torsion, and so its image in $\uH^{\beta+\lambda_n}$ must coincide with the entire torsion subgroup.

If $\dim \beta=-1$ then $\uH^{\beta+\lambda_n}$ is torsion, whereas $\uH_0(S^{1+\beta}\bbox X_n)\iso I_n$.  So the map from the former to the latter must be zero, and hence multiplication by $a_n$ is surjective.

Finally, if $\dim \beta=0$ then our exact sequence  becomes
\[
\uH^{\beta-1+\lambda_n}\lra I_n\lra \uH^\beta \llra{\cdot a_n} \uH^{\beta+\lambda_n} \lra 0.
\]
The Mackey functor $\uH^{\beta-1+\lambda_n}$ is torsion and $I_n$ is torsion-free, so the left map is zero.  It follows at once that multiplication by $a_n$ is injective on torsion classes.
\end{proof}

Note that multiplication by $a_n$ increases the geometric dimension by two, and so separates $H^\star$ into odd and even strands. 
The following diagram nicely describes the situation:
\begin{equation}
\label{eq:slices}
\xymatrixcolsep{2.5pc}\xymatrix{
\cdots\ar[r]^-\iso & (\dim \beta=-3) \ar[r]^-\iso &
(\dim \beta=-1) \ar@{->>}[r] & (\dim \beta=1) \ar[r]^-\iso & \cdots
 \\
\cdots \ar[r]^-\iso & (\dim \beta=-2) \ar@{ >->}[r]_{tors-iso} & (\dim \beta=0) \ar@{->>}[r]_{tors-inj} & (\dim \beta=2)\ar[r]^-\iso & \cdots 
}
\end{equation}
Thanks to the isomorphisms, once one knows the slices $\dim \beta=r$ for $r\in \{-1,0,1,2\}$  one knows the entirely of $H^\star(\pt)$.  So if $\Div_n$ is $D$-dimensional, this can be thought of as reducing to a $(D-1)$-dimensional calculation.  

\begin{cor}
If $\dim \beta$ is odd then any class in $H^\beta$ is infinitely-divisible by $a_n$.  If $\dim \beta$ is even then any nonzero torsion class $x\in H^\beta$ has the property that $a_n^kx\neq 0$ for all $k\geq 1$.
\end{cor}

\begin{proof}
Follows immediately from the diagram (\ref{eq:slices}).
\end{proof}

This idea can be continued a bit further.  Multiplication by a $u_d$-class preserves the geometric dimension, so it acts on each of the slices in (\ref{eq:slices}). These classes are also isomorphisms in a large range, though the exact range varies with $d$.  Here is one result in this direction (see (\ref{eq:dim_ell}) for the definition of $\dim_{\hat{\ell}} \beta$):

\begin{prop}
\label{pr:u-iso-range}
Let $\ell$ be a prime and let $\beta$ be any multi-index.  Then $H^{\beta}\llra{\cdot u_\ell} H^{\beta+\lambda_\ell-2}$ is an isomorphism if $\dim_{\hat{\ell}} \beta\notin \{2,3\}$.  Moreover, the map is surjective when $\dim_{\hat{\ell}} \beta=3$ and injective when $\dim_{\hat{\ell}} \beta=2$.
\end{prop}

\begin{proof}
Start with the cofiber sequence $\Sigma^2\mZ \llra{u_\ell} S^{\lambda_\ell}\lra \mZ/I_\ell$ and box with $S^{\beta-2}$ to obtain $S^\beta\llra{u_{\ell}} S^{\lambda_{\ell}+\beta-2} \ra S^{\beta-2}\bbox \mZ/I_\ell$.  By Proposition~\ref{pr:sphere-box-Z/I} we know
$S^{\beta-2}\bbox \mZ/I_\ell \he \Sigma^{\dim_{\hat{\ell}}\beta -2} \mZ/I_\ell$.  Since this complex only has homology in degree $\dim_{\hat{\ell}}\beta-2$, the result now follows immediately.
\end{proof}

Proposition~\ref{pr:u-iso-range} effectively lets us cut down the difficulty of calculating $H^\star(\pt)$ by one further dimension.  

\subsection{The example of \mdfn{$G=C_{\ell^2}$}}
\label{se:ell^2}

The fundamental representations are $1$, $\lambda_\ell$ and $\lambda_{\ell^2}$, so we have a tri-graded ring.  The reductions described in the last section essentially reduce this to a one-dimensional computation, which is manageable.  
We depict the result in Figure~\ref{fig:C_lsquared} below.
Writing down this chart, and explaining how to navigate it, will be our goal in this section.  

We start with a list of the isotypical regions: parts (1) and (2) are just specializations of the general results from the last section.  

\begin{prop}
\label{pr:iso-ranges}
For the map with domain $H^{r\lambda_\ell + k\lambda_{\ell^2}+m}=H^\beta$:
\begin{enumerate}[(1)]
\item Multiplication by $a_{\ell^2}$ is an isomorphism if $\dim \beta\notin \{0,-1,-2\}$.  It is surjective if $\dim \beta\in \{0,-1\}$, and if $\dim \beta=-2$ it is injective and the image is precisely the torsion subgroup of the target.
\item Multiplication by $u_\ell$ is an isomorphism when $m\notin \{2,3\}$.  It is injective when $m=2$, and surjective when $m=3$.
\item Multiplication by $u_{\ell^2}$  is an isomorphism if either
\[ \text{ $r\geq 0$ and (\ $4\leq m$ or $m\leq 1-2r$\ )}
\] OR
\[ \text{ $r<0$ and  (\ $m\leq 1 $ or $m\geq 4-2r$\ ).}
\]
\end{enumerate}
\end{prop}

\begin{proof}
Only part (3) is new, and for this 
 one uses $\Sigma^2\mZ\llra{u_{\ell^2}} S^{\lambda_{\ell^2}}\ra \mZ/I_{\ell^2}$ together with $S^{\lambda_{\ell^2}}\bbox \mZ/I_{\ell^2}\he \mZ/I_{\ell^2}$.  We leave the details to the reader. 
\end{proof}

\vspace{0.1in}

Part (1) of the above result implies that all of the groups $H^{\beta}$ can be deduced from the values on the four planes $\dim \beta\in \{-1,0,1,2\}$.  In fact, the plane $\dim\beta=1$ turns out to be entirely zero, though this is not clear initially. 
In Figure~\ref{fig:C_lsquared}
we will draw a system of charts---one for each $r$---where each chart depicts the four lines of intersection with the above planes.

How to read the chart:
\begin{enumerate}[(i)]
\item $\square=\Z$, $\bullet=\Z/\ell$, and 
\begin{tikzpicture}[scale=1.9]
\fill(0,0) circle (0.02);
\draw(0,0) circle (0.05);
\end{tikzpicture}
$=\Z/\ell^2$.  
\item The charts show the groups $H^{r\lambda_\ell+k\lambda_{\ell^2}+m}=H^\beta$.  The $r$-value separates the different charts.  In each chart, the line with the squares consists of the groups where $\dim \beta=0$, and
the line immediately below it is $\dim \beta=-1$.
The two lines on either side of these (when present) are $\dim \beta=-2$ and $\dim \beta=1$.
The line $\dim \beta=-2$ only shows up in the charts  near the bottom of the page, and the line $\dim\beta=1$ consists only of the zero groups.

\item Note that $m$ is constant along the vertical lines.  
\item Moving around the diagram:
\begin{itemize}
\item Within any chart, moving up one unit (up two $\beta$ lines) is a $+\lambda_{\ell^2}$ move.

\item Moving right two units in a chart (e.g. from a square to the next one) is a $-\lambda_{\ell^2}+2$ move.

\item Moving from a spot on the ``squares'' line to the spot to the right of it on the line below is a $-\lambda_{\ell^2}+1$ move.  
\item Moving from a spot in chart $r$ to the same spot in chart $r+1$ (below) is a $+\lambda_\ell-\lambda_{\ell^2}$ move.  
\end{itemize}

\item To find a given group, e.g. $H^{2\lambda_\ell+\lambda_{\ell^2}-4}$, in the diagram: the coefficient of $\lambda_\ell$ says what chart to look at (here $r=2$) and then the constant coefficient indicates the appropriate vertical line (here $m=-4$).  The geometric dimension (in the example, $4+2-4=2$) tells what horizontal line to look at in the chart (so we see $H^{2\lambda_\ell+\lambda_{\ell^2}-4}\cong \Z/\ell^2$).  

\item Multiplication by the elements $u_\ell$, $u_{\ell^2}$, $a_{\ell^2}$, and $a_\ell$ behave as follows:
\begin{itemize}
\item Multiplication by $u_{\ell^2}$ stays within the given chart and moves left two units.
\item Multiplication by $u_\ell$ moves one chart down, and left two units.
\item 
Multiplication by $a_{\ell^2}$ stays within the given chart and moves up one unit.

\item Multiplication by $a_\ell$ moves down one chart and then up one unit.
\end{itemize}
Prototypes for each of these multiplications are shown in blue in the center of the diagram, emanating from the class $1$.  

\item Vertical lines indicate multiplication by $a_{\ell^2}$.  Arrows indicate that the pattern of such multiplications repeats infinitely, either above or below; for example, note the families $u_\ell (a_{\ell^2})^k$ and $\tfrac{u_\ell\gamma_{\ell^2}}{(a_{\ell^2})^k}$.  
\item Horizontal arrows (right to left) denote multiplication by $u_{\ell^2}$.  If the arrow has an integer label $k$ this means that multiplication of the tail by $u_{\ell^2}$ yields $k$ times the head.  For example, $u_{\ell^2}\cdot u_{[\ell\cln \ell^2]}=\ell\cdot u_\ell$.   Note that we have not drawn horizontal arrows when they can be readily deduced from others, e.g. the undrawn arrow pointing from $u_\ell a_{\ell^2}$ to $u_\ell u_{\ell^2}a_{\ell^2}$.   The three horizontal arrows labeled $X$ indicate that $u_{\ell^2}$ times the tail of the arrow is a nontrivial linear combination of the two basis elements in the target (more on this below). 
\item Note the $u_\ell$-periodicity: the picture to the left of $m=2$ consists of the lines $m=0$ and $m=1$ repeated over and over (with a corresponding downward shift as one moves to the left), and likewise the picture to the right of $m=2$ consists of the $m=2$ and $m=3$ lines repeated over and over (with the corresponding upward shift as one moves to the right).  So the actual computations necessary are only in the range $0\leq m\leq 3$. 
\item Note the class $\alpha\in H^{2\lambda_\ell-2\lambda_{\ell^2}}$.  This group is $\Z\oplus \Z/\ell$ by Example~\ref{ex:H-twofold} and $\alpha$ is the unique torsion class having the property that $a_{\ell^2}\alpha=a_{\ell^2}(u_{[\ell\cln \ell^2]})^2$. Note that the appearance of $\alpha$ opens the door to a formula such as
\[ u_{\ell^2}\cdot \tfrac{\ell^2(u_\ell)^3}{(u_{\ell^2})^2}=1\cdot (u_{[\ell\cln \ell^2]})^2+(??)\alpha.
\]
The $1$ coefficient is forced by rationalization.  One then deduces that $??=-1$, since that is the only way the equation is compatible with $a_{\ell^2}\cdot \tfrac{\ell^2(u_\ell)^2}{(u_{\ell^2})^3}=0$ (see the chart).  This `hidden extension' was observed in \cite{Z}; it is labelled $X$ in the chart.

\item The class $\alpha$ generates an entire family of elements via the operator $E$ from Remark~\ref{re:E-operator}.  

\item Observe the relations
\[ u_{\ell^2}\cdot \tfrac{\gamma_{\ell^2}}{u_\ell}=\tfrac{a_\ell \gamma_\ell}{a_{\ell^2}}, \qquad
u_{\ell^2}\cdot \tfrac{a_\ell\gamma_\ell}{u_\ell a_{\ell^2}}=0,\qquad\text{and}\qquad
u_{\ell^2}\cdot \tfrac{\gamma_\ell}{u_\ell}=0
\]
that can be found in the upper right corner of the chart.  These are all readily derived from the relations given in Section~\ref{se:magic} (or the rigorous versions in earlier sections), and doing so is an informative exercise.  For the third one, it helps to first divide the fraction on the left-hand-side by $a_{\ell^2}$ and prove the relation in that form instead. 
\end{enumerate}

\begin{remark}
As mentioned earlier, for $G=C_{\ell^2}$ one has $H^{\beta}=0$ whenever $\dim \beta=1$.  This phenomenon does not hold for larger $G$: for example, $H^{\lambda_{\ell^4}+\lambda_{\ell}-2\lambda_{\ell^2}+1}\neq 0$ for $G=C_{\ell^4}$, with the class $a_{\ell^4}u_\ell\gamma_{\ell^2}$ a nonzero element.  This class must support an infinite number of $a_{\ell^4}$-multiplications.  
\end{remark}

We should explain how this chart was computed.  First, via $u_{\ell^2}$-periodicity we only need to compute the lines $m\in \{0,1,2,3\}$.  Next, find the positive and negative cones in the picture: these are wedge-like sectors in the bottom-left and top-right, and we know the groups here by the general results established earlier.  Using $u_{\ell^2}$-periodicity we then know the top parts of the $m=0,1$ lines and the bottom parts of the $m=2,3$ lines.  This is already a big piece of the entire picture.   For the remaining groups, we made use of the long exact sequences for multiplication by $u_\ell$ and $u_{\ell^2}$ as well as some ad hoc chain-level analysis to settle some of the extension problems.  The methods are not hard, but not particularly enlightening. We hope some future reader will find a more illuminating approach.

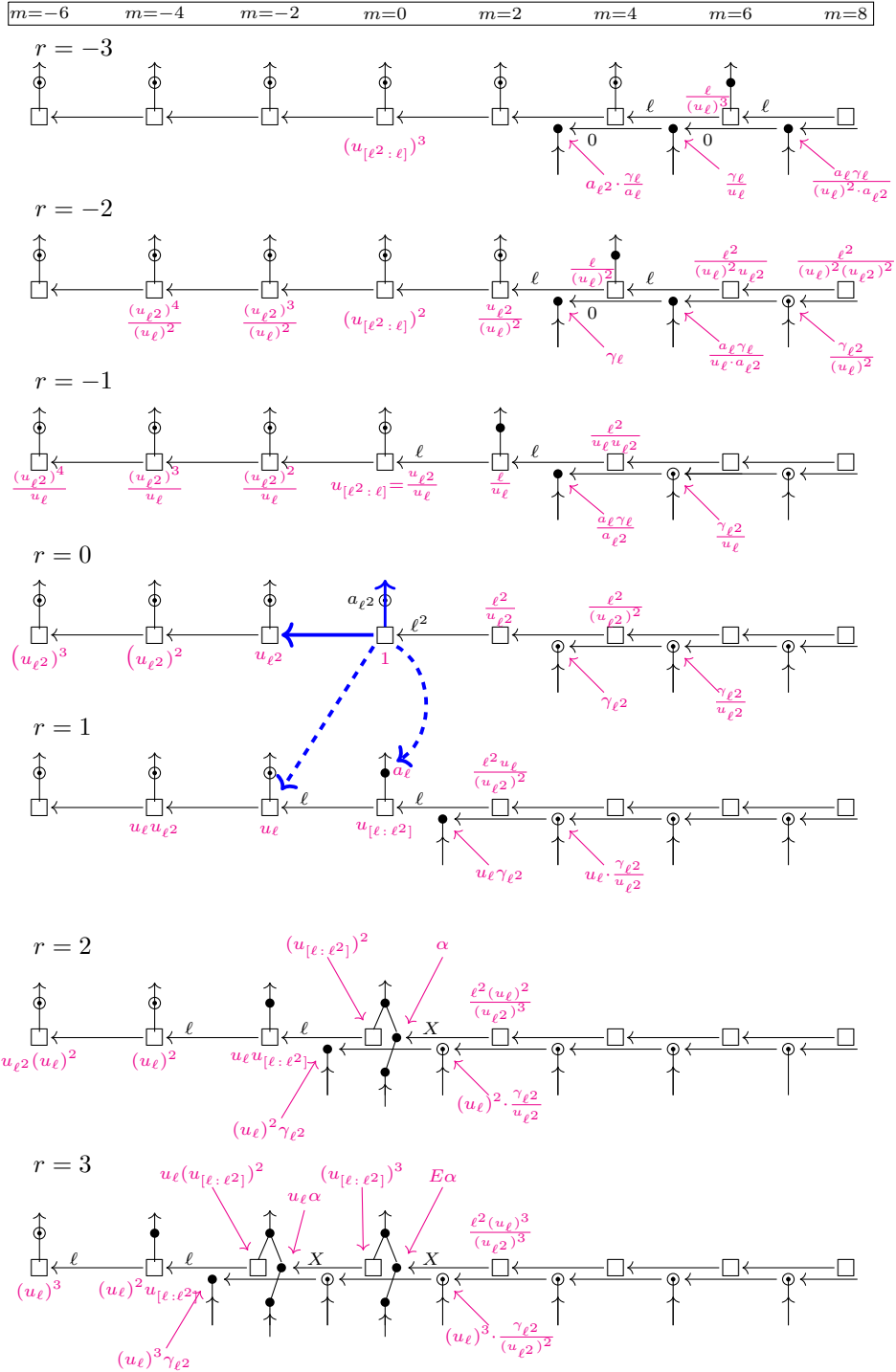
\begin{figure}[p]
\caption{$H^{r\lambda_\ell+k\lambda_{\ell^2}+m}$ for $G=C_{\ell^2}$}
\label{fig:C_lsquared}
\vspace{0.2in}

\begin{tikzpicture}[scale=1.6]
\draw(-3.27,7.0)--(4.24,7.0)--(4.24,6.8)--(-3.27,6.8)--(-3.27,7.0);
\draw(-3,6.9) node{$\scriptstyle{m=-6}$};
\draw(-2,6.9) node{$\scriptstyle{m=-4}$};
\draw(-1,6.9) node{$\scriptstyle{m=-2}$};
\draw(0,6.9) node{$\scriptstyle{m=0}$};
\draw(1,6.9) node{$\scriptstyle{m=2}$};
\draw(2,6.9) node{$\scriptstyle{m=4}$};
\draw(3,6.9) node{$\scriptstyle{m=6}$};
\draw(4,6.9) node{$\scriptstyle{m=8}$};
\draw(-2.7,6.6) node{${r=-3}$};
\draw (-3,6) node{$\square$};
\draw (-2,6) node{$\square$};
\draw (-1,6) node{$\square$};
\draw (0,6) node{$\square$};
\draw (1,6) node{$\square$};
\draw (2,6) node{$\square$};
\draw (3,6) node{$\square$};
\draw (4,6) node{$\square$};
\fill (3,6.3) circle (0.04);
\draw[->] (3,6.05)--(3,6.48);
\fill (2,6.3) circle (0.02);
\draw (2,6.3) circle (0.05);
\draw[->] (2,6.05)--(2,6.48);
\fill (1,6.3) circle (0.02);
\draw (1,6.3) circle (0.05);
\draw[->] (1,6.05)--(1,6.48);
\fill (0,6.3) circle (0.02);
\draw (0,6.3) circle (0.05);
\draw[->] (0,6.05)--(0,6.48);
\fill (-1,6.3) circle (0.02);
\draw (-1,6.3) circle (0.05);
\draw[->] (-1,6.05)--(-1,6.48);
\fill (-2,6.3) circle (0.02);
\draw (-2,6.3) circle (0.05);
\draw[->] (-2,6.05)--(-2,6.48);
\fill (-3,6.3) circle (0.02);
\draw (-3,6.3) circle (0.05);
\draw[->] (-3,6.05)--(-3,6.48);
\fill (1.5,5.9) circle (0.04);
\draw (1.5,5.5)--(1.5,5.92);
\draw[->] (1.5,5.5)--(1.5,5.75);
\fill (2.5,5.9) circle (0.04);
\draw (2.5,5.5)--(2.5,5.92);
\draw[->] (2.5,5.5)--(2.5,5.75);
\fill (3.5,5.9) circle (0.04);
\draw (3.5,5.5)--(3.5,5.92);
\draw[->] (3.5,5.5)--(3.5,5.75);
\draw[->] (-2.1,6)--(-2.9,6);
\draw[->] (-1.1,6)--(-1.9,6);
\draw[->] (-0.1,6)--(-0.9,6);
\draw[->] (0.9,6)--(0.1,6);
\draw[->] (1.9,6)--(1.1,6);
\draw[->] (2.9,6)--(2.1,6);
\draw[->] (3.9,6)--(3.1,6);
\draw (2.3,6.1) node{$\scriptstyle{\ell}$};
\draw (3.3,6.1) node{$\scriptstyle{\ell}$};
\draw[->](2.4,5.9)--(1.6,5.9);
\draw (1.8,5.8) node{$\scriptstyle{0}$};
\draw (2.8,5.8) node{$\scriptstyle{0}$};
\draw[->](3.4,5.9)--(2.6,5.9);
\draw[->](4.1,5.9)--(3.6,5.9);
\draw[magenta](2.8,6.15)  node{$\scriptstyle{\frac{\ell}{(u_\ell)^3}}$};
\draw[magenta](0,5.73) node{$\scriptstyle{
(u_{[\ell^2\cln \ell]})^3} $};

\draw[magenta](2.0,5.4) node{$\scriptstyle{a_{\ell^2}\cdot \frac{\gamma_\ell}{a_\ell}}$};
\draw[->,magenta] (1.9,5.5) -- (1.6,5.8);

\draw[magenta](3.05,5.4) node{$\scriptstyle{\frac{\gamma_\ell}{u_\ell}}$};
\draw[->,magenta](2.9,5.55)--(2.6,5.8);
\draw[magenta](4.05,5.4) node{$\scriptstyle{\frac{a_\ell\gamma_\ell}{(u_\ell)^2\cdot a_{\ell^2}}}$};
\draw[->,magenta](3.9,5.55)--(3.6,5.8);
\draw(-2.7,5.2) node{${r=-2}$};
\draw (-3,4.5) node{$\square$};
\draw (-2,4.5) node{$\square$};
\draw (-1,4.5) node{$\square$};
\draw (0,4.5) node{$\square$};
\draw (1,4.5) node{$\square$};
\draw (2,4.5) node{$\square$};
\draw (3,4.5) node{$\square$};
\draw (4,4.5) node{$\square$};
\fill (2,4.8) circle (0.04);
\draw[->] (2,4.55)--(2,4.98);
\fill (1,4.8) circle (0.02);
\draw (1,4.8) circle (0.05);
\draw[->] (1,4.55)--(1,4.98);
\fill (0,4.8) circle (0.02);
\draw (0,4.8) circle (0.05);
\draw[->] (0,4.55)--(0,4.98);
\fill (-1,4.8) circle (0.02);
\draw (-1,4.8) circle (0.05);
\draw[->] (-1,4.55)--(-1,4.98);
\fill (-2,4.8) circle (0.02);
\draw (-2,4.8) circle (0.05);
\draw[->] (-2,4.55)--(-2,4.98);
\fill (-3,4.8) circle (0.02);
\draw (-3,4.8) circle (0.05);
\draw[->] (-3,4.55)--(-3,4.98);
\fill (1.5,4.4) circle (0.04);
\draw (1.5,4.0)--(1.5,4.42);
\draw[->] (1.5,4.0)--(1.5,4.25);
\fill (2.5,4.4) circle (0.04);
\draw (2.5,4.0)--(2.5,4.42);
\draw[->] (2.5,4.0)--(2.5,4.25);
\fill (3.5,4.4) circle (0.02);
\draw (3.5,4.4) circle (0.05);
\draw (3.5,4.0)--(3.5,4.42);
\draw[->] (3.5,4.0)--(3.5,4.25);
\draw[->] (-2.1,4.5)--(-2.9,4.5);
\draw[->] (-1.1,4.5)--(-1.9,4.5);
\draw[->] (-0.1,4.5)--(-0.9,4.5);
\draw[->] (0.9,4.5)--(0.1,4.5);
\draw[->] (1.9,4.5)--(1.1,4.5);
\draw[->] (2.9,4.5)--(2.1,4.5);
\draw[->] (3.9,4.5)--(3.1,4.5);
\draw (1.3,4.6) node{$\scriptstyle{\ell}$};
\draw (2.3,4.6) node{$\scriptstyle{\ell}$};
\draw[->](2.4,4.4)--(1.6,4.4);
\draw (1.8,4.3) node{$\scriptstyle{0}$};
\draw[->](3.4,4.4)--(2.6,4.4);
\draw[->](4.1,4.4)--(3.6,4.4);
\draw[magenta](1.8,4.65)  node{$\scriptstyle{\frac{\ell}{(u_\ell)^2}}$};
\draw[magenta](1,4.23) node{$\scriptstyle{\frac{u_{\ell^2}}{(u_\ell)^2}}$};
\draw[magenta](0,4.23) node{$\scriptstyle{
(u_{[\ell^2\cln \ell]})^2} $};
\draw[magenta](-1,4.23) node{$\scriptstyle{\frac{(u_{\ell^2})^3}{(u_\ell)^2}  }$};
\draw[magenta](-2,4.23) node{$\scriptstyle{
\frac{ (u_{\ell^2})^4}{(u_\ell)^2}}   $};

\draw[magenta](3,4.725) node{$\scriptstyle{\frac{\ell^2}{(u_\ell)^2 u_{\ell^2}}}$};

\draw[magenta](4,4.725) node{$\scriptstyle{\frac{\ell^2}{(u_\ell)^2 (u_{\ell^2})^2}}$};

\draw[magenta](2.0,3.9) node{$\scriptstyle{\gamma_\ell}$};
\draw[->,magenta] (1.9,4.0) -- (1.6,4.3);
\draw[magenta](3.05,3.9) node{$\scriptstyle{\frac{a_\ell\gamma_\ell}{u_\ell\cdot a_{\ell^2}}}$};
\draw[->,magenta](2.9,4.05)--(2.6,4.3);

\draw[magenta](4.05,3.9) node{$\scriptstyle{\frac{\gamma_{\ell^2}}{(u_\ell)^2}}$};
\draw[->,magenta](3.9,4.05)--(3.6,4.3);

\draw(-2.7,3.7) node{$r=-1$};
\draw (-3,3) node{$\square$};
\draw (-2,3) node{$\square$};
\draw (-1,3) node{$\square$};
\draw (0,3) node{$\square$};
\draw (1,3) node{$\square$};
\draw (2,3) node{$\square$};
\draw (3,3) node{$\square$};
\draw (4,3) node{$\square$};
\fill (1,3.3) circle (0.04);
\draw[->] (1,3.05)--(1,3.48);
\fill (0,3.3) circle (0.02);
\draw (0,3.3) circle (0.05);
\draw[->] (0,3.05)--(0,3.48);
\fill (-1,3.3) circle (0.02);
\draw (-1,3.3) circle (0.05);
\draw[->] (-1,3.05)--(-1,3.48);
\fill (-2,3.3) circle (0.02);
\draw (-2,3.3) circle (0.05);
\draw[->] (-2,3.05)--(-2,3.48);
\fill (-3,3.3) circle (0.02);
\draw (-3,3.3) circle (0.05);
\draw[->] (-3,3.05)--(-3,3.48);
\fill (1.5,2.9) circle (0.04);
\draw (1.5,2.5)--(1.5,2.92);
\draw[->] (1.5,2.5)--(1.5,2.75);
\fill (2.5,2.9) circle (0.02);
\draw (2.5,2.9) circle (0.05);
\draw (2.5,2.5)--(2.5,2.92);
\draw[->] (2.5,2.5)--(2.5,2.75);
\fill (3.5,2.9) circle (0.02);
\draw (3.5,2.9) circle (0.05);
\draw (3.5,2.5)--(3.5,2.92);
\draw[->] (3.5,2.5)--(3.5,2.75);
\draw[->] (-2.1,3)--(-2.9,3);
\draw[->] (-1.1,3)--(-1.9,3);
\draw[->] (-0.1,3)--(-0.9,3);
\draw[->] (1.9,3)--(1.1,3);
\draw[->] (2.9,3)--(2.1,3);
\draw[->] (0.9,3)--(0.1,3);
\draw[->] (3.9,3)--(3.1,3);
\draw (0.3,3.1) node{$\scriptstyle{\ell}$};
\draw (1.3,3.1) node{$\scriptstyle{\ell}$};
\draw[->](2.4,2.9)--(1.6,2.9);
\draw[->](3.1,2.9)--(2.6,2.9);
\draw[->](3.4,2.9)--(2.6,2.9);
\draw[->](4.1,2.9)--(3.6,2.9);
\draw[magenta](1,2.8)  node{$\scriptstyle{\frac{\ell}{u_\ell}}$};
\draw[magenta](0,2.8) node{$\scriptstyle{u_{[\ell^2\cln \ell]}=\frac{u_{\ell^2}}{u_\ell}}  $};
\draw[magenta](-1,2.8) node{$\scriptstyle{\frac{(u_{\ell^2})^2}{u_\ell}  }$};
\draw[magenta](-2,2.8) node{$\scriptstyle{
\frac{ (u_{\ell^2})^3}{u_\ell}}   $};
\draw[magenta](-3,2.8) node{$\scriptstyle{
\frac{ (u_{\ell^2})^4}{u_\ell}}   $};
\draw[magenta](2,3.22) node{$\scriptstyle{\frac{\ell^2}{u_\ell u_{\ell^2}}}$};
\draw[magenta](2,2.4) node{$\scriptstyle{\frac{a_\ell \gamma_\ell}{a_{\ell^2}}}$}; 
\draw[->,magenta](1.9,2.55)--(1.6,2.8);
\draw[magenta](3,2.35) node{$\scriptstyle{ \frac{\gamma_{\ell^2}}{u_{\ell}} }$};
\draw[->,magenta] (2.9,2.5) -- (2.6,2.8);

\draw(-2.8,2.2) node{$r=0$};
\draw (-3,1.5) node{$\square$};
\draw (-2,1.5) node{$\square$};
\draw (-1,1.5) node{$\square$};
\draw (0,1.5) node{$\square$};
\draw (1,1.5) node{$\square$};
\draw (2,1.5) node{$\square$};
\draw (3,1.5) node{$\square$};
\draw (4,1.5) node{$\square$};
\fill (0,1.8) circle (0.02);
\draw (0,1.8) circle (0.05);
\draw[->,blue,line width=0.015in] (0,1.57)--(0,1.98);
\draw (-0.2,1.78) node{$\scriptstyle{a_{\ell^2}}$};
\fill (-1,1.8) circle (0.02);
\draw (-1,1.8) circle (0.05);
\draw[->] (-1,1.55)--(-1,1.98);
\fill (-2,1.8) circle (0.02);
\draw (-2,1.8) circle (0.05);
\draw[->] (-2,1.55)--(-2,1.98);
\fill (-3,1.8) circle (0.02);
\draw (-3,1.8) circle (0.05);
\draw[->] (-3,1.55)--(-3,1.98);

\fill (1.5,1.4) circle (0.02);
\draw (1.5,1.4) circle (0.05);
\draw (1.5,1.0)--(1.5,1.42);
\draw[->] (1.5,1.0)--(1.5,1.25);
\fill (2.5,1.4) circle (0.02);
\draw (2.5,1.4) circle (0.05);
\draw (2.5,1.0)--(2.5,1.42);
\draw[->] (2.5,1.0)--(2.5,1.25);
\fill (3.5,1.4) circle (0.02);
\draw (3.5,1.4) circle (0.05);
\draw (3.5,1.0)--(3.5,1.42);
\draw[->] (3.5,1.0)--(3.5,1.25);
\draw[->] (-2.1,1.5)--(-2.9,1.5);
\draw[->] (-1.1,1.5)--(-1.9,1.5);
\draw[->,blue,line width=0.02in] (-0.1,1.5)--(-0.9,1.5);
\draw[->] (1.9,1.5)--(1.1,1.5);
\draw[->] (2.9,1.5)--(2.1,1.5);
\draw[->] (0.9,1.5)--(0.1,1.5);
\draw[->] (3.9,1.5)--(3.1,1.5);
\draw (0.3,1.6) node{$\scriptstyle{\ell^2}$};
\draw[->](2.4,1.4)--(1.6,1.4);
\draw[->](3.4,1.4)--(2.6,1.4);
\draw[->](4.1,1.4)--(3.6,1.4);

\draw[->,blue,line width=0.02in,dashed] (-0.1,1.4)--(-0.9,0.15);
\draw[->,blue,line width=0.02in,dashed] (0.1,1.4) to [out=330,in=30] (0.1,0.4);

\draw[magenta](0,1.3)  node{$\scriptstyle{1}$};
\draw[magenta](-1,1.3) node{$\scriptstyle{u_{\ell^2}}$};
\draw[magenta](-2,1.3) node{$(\scriptstyle{u_{\ell^2})^2}$};
\draw[magenta](-3,1.3) node{$(\scriptstyle{u_{\ell^2})^3}$};
\draw[magenta](1,1.72) node{$\scriptstyle{\frac{\ell^2}{u_{\ell^2}}}$};
\draw[magenta](2,1.72) node{$\scriptstyle{\frac{\ell^2}{(u_{\ell^2})^2}}$};
\draw[magenta](2,0.9) node{$\scriptstyle{\gamma_{\ell^2}}$};
\draw[->,magenta](1.9,1.0)--(1.6,1.3);
\draw[magenta](3,0.9) node{$\scriptstyle{ \frac{\gamma_{\ell^2}}{u_{\ell^2}} }$};
\draw[->,magenta] (2.9,1.0) -- (2.6,1.3);

\draw(-2.8,0.7) node{$r=1$};
\fill (0,0.3)  circle (0.04);
\draw[->] (0,0.05)--(0,0.48);
\fill (-1,0.3) circle (0.02);
\draw (-1,0.3) circle (0.05);
\draw[->] (-1,0.05)--(-1,0.48);
\fill (-2,0.3) circle (0.02);
\draw (-2,0.3) circle (0.05);
\draw[->] (-2,0.05)--(-2,0.48);
\fill (-3,0.3) circle (0.02);
\draw (-3,0.3) circle (0.05);
\draw[->] (-3,0.05)--(-3,0.48);
\draw (-3,0) node{$\square$};
\draw (-2,0) node{$\square$};
\draw (-1,0) node{$\square$};
\draw (0,0) node{$\square$};
\draw (1,0) node{$\square$};
\draw (2,0) node{$\square$};
\draw (3,0) node{$\square$};
\draw (4,0) node{$\square$};
\draw[->] (-2.1,0)--(-2.9,0);
\draw[->] (-1.1,0)--(-1.9,0);
\draw[->] (-0.1,0)--(-0.9,0);
\draw[->] (3.9,0)--(3.1,0);
\draw (-0.7,0.08) node{$\scriptstyle{\ell}$};
\draw[->] (0.9,0)--(0.1,0);
\draw (0.3,0.08) node{$\scriptstyle{\ell}$};
\draw[->] (1.9,0)--(1.1,0);
\draw[->] (2.9,0)--(2.1,0);
\fill (0.5,-0.1) circle (0.04);
\draw (0.5,-0.5)--(0.5,-0.1);
\draw[->] (0.5,-0.5)--(0.5,-0.25);
\fill (1.5,-0.1) circle (0.02);
\draw (1.5,-0.1) circle (0.05);
\draw (1.5,-0.5)--(1.5,-0.08);
\draw[->] (1.5,-0.5)--(1.5,-0.25);
\fill (2.5,-0.1) circle (0.02);
\draw (2.5,-0.1) circle (0.05);
\draw (2.5,-0.5)--(2.5,-0.08);
\draw[->] (2.5,-0.5)--(2.5,-0.25);
\fill (3.5,-0.1) circle (0.02);
\draw (3.5,-0.1) circle (0.05);
\draw (3.5,-0.5)--(3.5,-0.08);
\draw[->] (3.5,-0.5)--(3.5,-0.25);
\draw[->] (1.4,-0.1)--(0.6,-0.1);
\draw[->] (2.4,-0.1)--(1.6,-0.1);
\draw[->](3.4,-0.1)--(2.6,-0.1);
\draw[->](4.1,-0.1)--(3.6,-0.1);
\draw[magenta](0,-0.2)  node{$\scriptstyle{u_{[\ell\cln \ell^2]}}$};
\draw[magenta](-1,-0.2) node{$\scriptstyle{u_\ell}$};
\draw[magenta](-2,-0.2) node{$\scriptstyle{u_\ell u_{\ell^2}}$};
\draw[magenta](1,0.3) node{$\scriptstyle{\frac{\ell^2 u_\ell}{(u_{\ell^2})^2}}$};
\draw[magenta](1,-0.6) node{$\scriptstyle{u_\ell\gamma_{\ell^2}}$};
\draw[->,magenta](0.9,-0.5)--(0.6,-0.2);
\draw[magenta](2,-0.6) node{$\scriptstyle{u_\ell\cdot \frac{\gamma_{\ell^2}}{u_{\ell^2}} }$};
\draw[->,magenta] (1.9,-0.5) -- (1.6,-0.2);
\draw[magenta] (0.15,0.3) node{$\scriptstyle{a_\ell}$};
%
\draw(-2.8,-1.2) node{$r=2$};
\draw (-3,-2) node{$\square$};
\draw (-2,-2) node{$\square$};
\draw (-1,-2) node{$\square$};
\draw (-0.1,-2) node{$\square$};
\fill (0.1,-2) circle (0.04);
\draw (1,-2) node{$\square$};
\draw (2,-2) node{$\square$};
\draw (3,-2) node{$\square$};
\draw (4,-2) node{$\square$};
\fill(0,-2.3) circle(0.04);
\draw (0,-2.3) -- (0.1,-2);
\draw[->] (0,-2.5)--(0,-2.4);
\draw (0,-2.6)--(0,-2.3);
\draw[->] (-2.1,-2)--(-2.9,-2);
\draw[->] (-1.1,-2)--(-1.9,-2);
\draw (-1.7,-1.92) node{$\scriptstyle{\ell}$};
\draw[->] (-0.2,-2)--(-0.9,-2);
\draw (-0.7,-1.92) node{$\scriptstyle{\ell}$};
\draw[->] (0.9,-2)--(0.18,-2);
\draw (0.4,-1.92) node{$\scriptstyle{X}$};
\draw[->] (1.9,-2)--(1.1,-2);
\draw[->] (2.9,-2)--(2.1,-2);
\draw[->] (3.9,-2)--(3.1,-2);
\fill (0,-1.7)  circle (0.04);
\draw (-0.1,-1.93)--(0,-1.7);
\draw (0.1,-1.95)--(0,-1.7);
\draw[->] (0,-1.7)--(0,-1.5);
\fill (-1,-1.7) circle (0.04);
\draw[->] (-1,-1.95)--(-1,-1.52);
\fill (-2,-1.7) circle (0.02);
\draw (-2,-1.7) circle (0.05);
\draw[->] (-2,-1.95)--(-2,-1.52);
\fill (-3,-1.7) circle (0.02);
\draw (-3,-1.7) circle (0.05);
\draw[->] (-3,-1.95)--(-3,-1.52);

\fill (-0.5,-2.1) circle (0.04);
\draw (-0.5,-2.5)--(-0.5,-2.1);
\draw[->] (-0.5,-2.5)--(-0.5,-2.25);
\fill (0.5,-2.1) circle (0.02);
\draw (0.5,-2.1) circle (0.05);
\draw (0.5,-2.5)--(0.5,-2.08);
\draw[->] (0.5,-2.5)--(0.5,-2.25);
\fill (1.5,-2.1) circle (0.02);
\draw (1.5,-2.1) circle (0.05);
\draw (1.5,-2.5)--(1.5,-2.08);
\fill (2.5,-2.1) circle (0.02);
\draw (2.5,-2.1) circle (0.05);
\draw (2.5,-2.5)--(2.5,-2.08);
\fill (3.5,-2.1) circle (0.02);
\draw (3.5,-2.1) circle (0.05);
\draw (3.5,-2.5)--(3.5,-2.08);
\draw[->] (1.5,-2.5)--(1.5,-2.25);
\draw[->] (2.5,-2.5)--(2.5,-2.25);

\draw[->] (0.4,-2.1)--(-0.4,-2.1);
\draw[->] (1.4,-2.1)--(0.6,-2.1);
\draw[->] (2.4,-2.1)--(1.6,-2.1);
\draw[->](3.4,-2.1)--(2.6,-2.1);
\draw[->](4.1,-2.1)--(3.6,-2.1);

\draw[magenta](-0.5,-1.2)  node{$\scriptstyle{(u_{[\ell\cln \ell^2]})^2 }$};
\draw[->,magenta](-0.5,-1.3)--(-0.19,-1.85);
\draw[magenta](0.5,-1.2) node{$\scriptstyle{\alpha}$};
\draw[->,magenta](0.5,-1.3)--(0.2,-1.88);
\draw[magenta](-1,-2.2) node{$\scriptstyle{u_\ell u_{[\ell\cln \ell^2]}}$};
\draw[magenta](-2,-2.2) node{$\scriptstyle{(u_\ell)^2}$};
\draw[magenta](-3,-2.2) node{$\scriptstyle{u_{\ell^2}(u_\ell)^2}$};
\draw[magenta](1,-1.7) node{$\scriptstyle{\frac{\ell^2 (u_\ell)^2}{(u_{\ell^2})^3}}$};
\draw[magenta](-1,-2.8) node{$\scriptstyle{(u_\ell)^2\gamma_{\ell^2}}$};
\draw[->,magenta](-0.9,-2.7)--(-0.6,-2.2);
\draw[magenta](1,-2.6) node{$\scriptstyle{(u_\ell)^2\cdot \frac{\gamma_{\ell^2}}{u_{\ell^2}} }$};
\draw[->,magenta] (0.9,-2.5) -- (0.6,-2.2);
%
%
\draw(-2.8,-3.1) node{$r=3$};
\draw (-3,-4) node{$\square$};
\draw (-2,-4) node{$\square$};
\draw (-1.1,-4) node{$\square$};
\fill (-0.9,-4) circle (0.04);
\draw (-0.1,-4) node{$\square$};
\fill (0.1,-4) circle (0.04);
\draw (1,-4) node{$\square$};
\draw (2,-4) node{$\square$};
\draw (3,-4) node{$\square$};
\draw (4,-4) node{$\square$};
\fill(-1,-4.3) circle(0.04);
\draw (-1,-4.3) -- (-0.9,-4);
\draw[->] (-1,-4.5)--(-1,-4.4);
\draw (-1,-4.6)--(-1,-4.3);
\fill(0,-4.3) circle(0.04);
\draw (0,-4.3) -- (0.1,-4);
\draw[->] (0,-4.5)--(0,-4.4);
\draw (0,-4.6)--(0,-4.3);
\fill (0,-3.7)  circle (0.04);
\draw (-0.1,-3.93)--(0,-3.7);
\draw (0.1,-3.95)--(0,-3.7);
\draw[->] (0,-3.7)--(0,-3.5);
\fill (-1,-3.7)  circle (0.04);
\draw (-1.1,-3.93)--(-1,-3.7);
\draw (-0.9,-3.95)--(-1,-3.7);
\draw[->] (-1,-3.7)--(-1,-3.5);
\fill (-2,-3.7) circle (0.04);
\draw[->] (-2,-3.95)--(-2,-3.52);
\fill (-3,-3.7) circle (0.02);
\draw (-3,-3.7) circle (0.05);
\draw[->] (-3,-3.95)--(-3,-3.52);
\draw[->] (-2.1,-4)--(-2.9,-4);
\draw[->] (-1.2,-4)--(-1.9,-4);
\draw[->] (-0.2,-4)--(-0.8,-4);
\draw[->] (0.9,-4)--(0.2,-4);
\draw[->] (1.9,-4)--(1.1,-4);
\draw[->] (2.9,-4)--(2.1,-4);
\draw[->] (3.9,-4)--(3.1,-4);
\draw (-2.7,-3.92) node{$\scriptstyle{\ell}$};
\draw (-1.7,-3.92) node{$\scriptstyle{\ell}$};
\draw (-0.6,-3.92) node{$\scriptstyle{X}$};
\draw (0.4,-3.92) node{$\scriptstyle{X}$};

\draw[->](-0.6,-4.1)--(-1.4,-4.1);
\draw[->](0.4,-4.1)--(-0.4,-4.1);
\draw[->](1.4,-4.1)--(0.6,-4.1);
\draw[->](2.4,-4.1)--(1.6,-4.1);
\draw[->](3.4,-4.1)--(2.6,-4.1);
\draw[->](4.1,-4.1)--(3.6,-4.1);
\fill (-1.5,-4.1) circle (0.04);
\draw (-1.5,-4.5)--(-1.5,-4.1);
\draw[->] (-1.5,-4.5)--(-1.5,-4.25);
\fill (-0.5,-4.1) circle (0.02);
\draw (-0.5,-4.1) circle (0.05);
\draw (-0.5,-4.5)--(-0.5,-4.08);
\draw[->] (-0.5,-4.5)--(-0.5,-4.25);
\fill (0.5,-4.1) circle (0.02);
\draw (0.5,-4.1) circle (0.05);
\draw (0.5,-4.5)--(0.5,-4.08);
\draw[->] (0.5,-4.5)--(0.5,-4.25);
\fill (1.5,-4.1) circle (0.02);
\draw (1.5,-4.1) circle (0.05);
\draw (1.5,-4.5)--(1.5,-4.08);
\draw[->] (1.5,-4.5)--(1.5,-4.25);
\fill (2.5,-4.1) circle (0.02);
\draw (2.5,-4.1) circle (0.05);
\draw (2.5,-4.5)--(2.5,-4.08);
\draw[->] (2.5,-4.5)--(2.5,-4.25);
\fill (3.5,-4.1) circle (0.02);
\draw (3.5,-4.1) circle (0.05);
\draw (3.5,-4.5)--(3.5,-4.08);
\draw[->] (3.5,-4.5)--(3.5,-4.25);

\draw[magenta](-1.5,-3.2)  node{$\scriptstyle{u_\ell (u_{[\ell\cln \ell^2]})^2 }$};
\draw[->,magenta](-1.5,-3.3)--(-1.19,-3.85);
\draw[magenta](-0.7,-3.4) node{$\scriptstyle{u_\ell \alpha}$};
\draw[->,magenta](-0.7,-3.5)--(-0.8,-3.88);
\draw[magenta](-2.05,-4.2) node{$\scriptstyle{(u_\ell)^2 u_{[\ell\cln\! \ell^2]}}$};
\draw[magenta](-3,-4.2) node{$\scriptstyle{(u_\ell)^3}$};
\draw[magenta](1,-3.7) node{$\scriptstyle{\frac{\ell^2 (u_\ell)^3}{(u_{\ell^2})^3}}$};
\draw[magenta](-2,-4.8) node{$\scriptstyle{(u_\ell)^3\gamma_{\ell^2}}$};
\draw[->,magenta](-1.9,-4.7)--(-1.6,-4.2);
\draw[magenta](1,-4.63) node{$\scriptstyle{(u_\ell)^3\cdot \frac{\gamma_{\ell^2}}{(u_{\ell^2})^2} }$};
\draw[->,magenta] (0.9,-4.5) -- (0.6,-4.2);

\draw[magenta](-0.2,-3.2)  node{$\scriptstyle{(u_{[\ell\cln \ell^2]})^3 }$};
\draw[->,magenta](-0.2,-3.3)--(-0.19,-3.85);
\draw[magenta](0.5,-3.2) node{$\scriptstyle{E \alpha}$};
\draw[->,magenta](0.5,-3.3)--(0.2,-3.88);
\end{tikzpicture}
\end{figure}


\section{Change of groups}
\label{se:change}

Throughout this paper we have worked with $C_n$ and $\mZ_{C_n}$ directly, though often the methods have involved localizing at a prime $\ell$.  One can envision a reduction from $C_n$ to the smaller group $C_{n(\ell)}$.  Here we explain how to approach this using the standard change-of-group functors.\smallskip

Fix a prime $\ell$.  Recall that $n(\ell)$ is the $\ell$-primary part of $n$,  and $n(\hat{\ell})=\frac{n}{n(\ell)}$.
We have the extension of groups
\[ 0 \ra C_{n(\hat{\ell})} \inc C_n \llra{f} C_{n(\ell)}  \ra 0.
\]
We will consider the restriction functor $\mRes_f\colon \mZ_{C_{n(\ell)}}\MMod \ra \mZ_{C_n}\MMod$, as described for example in \cite[Section 7]{DHone}.  
By \cite[Remark 7.4]{DHone} we have 
$\mRes_f(F_{\Theta_{\ell^e}})=F_{\Theta_{\ell^e}}$.
Note that the two free functors, though written the same, are slightly different: the first is a $\mZ_{C_{n(\ell)}}$-module and the second is a $\mZ_{C_n}$-module.  It follows that $\mRes_f(S^{\lambda_{\ell^e}})=S^{\lambda_{\ell^e}}$, and since $\mRes_f$ preserves the box product (\cite[Proposition 7.5]{DHone}) we get
\[ \mRes_f(S^{\lambda_{\ell^{e_1}}}\bbox \cdots \bbox
S^{\lambda_{\ell^{e_s}}} )\iso
S^{\lambda_{\ell^{e_1}}}\bbox \cdots \bbox S^{\lambda_{\ell^{e_s}}}.
\]
Once again, note that the spheres on the left and the right are slightly different: on the left they denote complexex of $\mZ_{C_{n(\ell)}}$-modules, and on the right they denote  complexes of $\mZ_{C_n}$-modules.  

For convenience write $\lambda_{\ell^{e_\bdot}}$ for $\lambda_{\ell^{e_1}}+\cdots+\lambda_{\ell^{e_s}}$.  Given sequences $(e_1,\ldots,e_s)$ and $(k_1,\ldots,k_t)$ we
now observe the following chain of isomorphisms:

\begin{align*}
[S^{\lambda_{\ell^{e_\bdot}}}, \Sigma^m S^{\lambda_{\ell^{k_\bdot}}}]_{C_n}&=
[\mRes_f(S^{\lambda_{\ell^{e_\bdot}}})\, , \, \Sigma^m \mRes_f(S^{\lambda_{\ell^{k_\bdot}}})]_{C_n}\\
&=
[S^{\lambda_{\ell^{e_\bdot}}} \, , \, \Sigma^m \mInd_f\mRes_f(S^{\lambda_{\ell^{k_\bdot}}})]_{C_{n(\ell)}}\\
&= 
[S^{\lambda_{\ell^{e_\bdot}}}, \Sigma^m S^{\lambda_{\ell^{k_\bdot}}}]_{C_{n(\ell)}}
\end{align*}
where the second equality uses the adjunction and the third equality uses \cite[Proposition 7.3(c)]{DHone}.
Restating this in words, computing homotopy classes in $\cD(\mZ_{C_n})$ between $\ell$-adic spheres reduces to the analogous computation in $\cD(\mZ_{C_{n(\ell)}})$.  In fact, the map
\[ \mRes_f \colon [S^{\lambda_{\ell^{e_\bdot}}},
\Sigma^m S^{\lambda_{\ell^{k_\bdot}}}]_{C_{n(\ell)}} \ra
[S^{\lambda_{\ell^{e_\bdot}}},
\Sigma^m S^{\lambda_{\ell^{k_\bdot}}}]_{C_{n}}
\]
is the isomorphism, and because $\mRes_f$ preserves the box products these maps assemble into an isomorphism of multigraded rings.  So when restricted to $\ell$-adic gradings, $H^\star_{C_n}(\pt)$ is isomorphic to $H^\star_{C_{n(\ell)}}(\pt)$.  

Next let us consider what happens away from the $\ell$-adic gradings.  To any abelian group $A$ we can associate the diagram of localizations
\[ \xymatrix{
A_{(2)}\ar[dr] & A_{(3)}\ar[d] & A_{(5)}\ar[dl] & \cdots \ar[dll] \\
& A_{(0)}
}
\]
where $A_{(\ell)}=A\tens \Z_{(\ell)}$ and $A_{(0)}=A\tens \Q$.  Write $\lim_\ell A_{(\ell)}$ for the limit of this diagram, and observe the canonical map $A\ra \lim_\ell A_{(\ell)}$.  It is easy to see that this is an isomorphism whenever $A$ is finitely-generated ($A=\bigoplus_\ell \Z/\ell\Z$ is a counterexample to the more general statement).  For example, when $A$ is finitely-generated and torsion then $A_{(0)}=0$ and the limit is just the direct sum of the $\ell$-primary torsion components of $A$.  

If $C$ and $D$ are bounded complexes of finitely-generated free abelian groups then the above observation implies that we have the isomorphism 
\begin{equation}
\label{eq:ell-reduction}
[C,D] \llra{\iso} \lim_\ell\, [C,D]_{(\ell)}.
\end{equation}
But we can also observe that
\begin{align*}
[C,D]_{(\ell)}=H_0(\Hom(C,D))\tens \Z_{(\ell)}&=H_0(\Hom(C,D)\tens \Z_{(\ell)} )\\
&\iso H_0(\Hom(C\tens \Z_{(\ell)},D\tens \Z_{(\ell)})) \\
&= [C_{(\ell)},D_{(\ell)}].
\end{align*}
So in fact we can rewrite (\ref{eq:ell-reduction}) as 
\begin{equation}
\label{eq:ell-reduction2}
[C,D] \llra{\iso} \lim_{\ell}\, [C_{(\ell)},D_{(\ell)}].
\end{equation}
The same argument works if $C$ and $D$ are bounded complexes of finitely-generated free $\mZ_G$-modules.  So if $(b_1,\ldots,b_s)$ and $(c_1,\ldots,c_t)$ are divisors of $n$ then we have the isomorphism
\[ [S^{\lambda_{\und{b}}}\,,\, \Sigma^m\, S^{\lambda_{\und{c}}}]_{C_n} \llra{\iso}
\lim_\ell\, 
[S^{\lambda_{\und{b}}}_{(\ell)}\, , \, \Sigma^m\, S^{\lambda_{\und{c}}}_{(\ell)} ]_{C_n}.
\]
We have seen in Proposition~\ref{pr:ell-local-1} that $S^{\lambda_d}_{(\ell)}   \he S^{\lambda_{d(\ell)}}_{(\ell)}$, and this yields
\[ 
[S^{\lambda_{\und{b}}}\,,\, \Sigma^m\, S^{\lambda_{\und{c}}}]_{C_n} \iso
\lim_\ell\, 
[S^{\lambda_{\und{b(\ell)}}}_{(\ell)}\, , \, \Sigma^m\, S^{\lambda_{\und{c(\ell)}}}_{(\ell)} ]_{C_n} \iso
\lim_\ell \,
[ S^{\lambda_{\und{b(\ell)}}}\, , \, \Sigma^m S^{\lambda_{\und{c(\ell)}}}]_{(\ell),C_{n(\ell)}}.
\]
This reduces the computation of the $H^\star_{C_n}(\pt)$ groups to the case of cyclic groups of prime-power order.  For $m\neq 2(s-t)$ the groups in the above isomorphism are all torsion, and the inverse limit is just a direct sum.  When $m=2(s-t)$ there is a single $\Z$ summand that is $\ell$-adically patched across the inverse limit.  

Tracking the multiplicative structure across the above isomorphisms is a little troublesome.  One needs to make use of the elements $u_{[d:d(\ell)]}$ and $u_{[d(\ell):d]}$, which are units after $\ell$-localization but not before.   
We do not pursue this further here.


\bibliographystyle{amsalpha}

\begin{thebibliography}{JTTW}

\bibitem[AMR]{AMR} D. Ayala, A. Mazel-Gee, N. Rozenblyum, {\em Derived Mackey functors and $C_{p^n}$-equivariant cohomology\/}, 2021 preprint, arXiv: 2105.02456.

\bibitem[BD]{BD} S. Basu and P. Dey, {\em Equivariant cohomology for cyclic groups\/}, 
Int. Math. Res. Not. IMRN 2025, Paper No. rnaf150, 41pp.

\bibitem[BG1]{BG1} S. Basu and S. Ghosh, {\em Computations in $C_{pq}$-Bredon cohomology\/}, Math. Z. {\bf 293} (2019), no. 3--4, 1443--1487.

\bibitem[BG2]{BG2} S. Basu and S. Ghosh, {\em Equivariant cohomology for cyclic groups of square-free order\/},
Res. Math. Sci. {\bf 11} (2024), no. 2, Paper No. 29, 33pp.

\bibitem[BG3]{BG3} S. Basu and S. Ghosh, {\em Non-trivial extensions in equivariant cohomology with constant coefficients\/}, 2021 preprint, arXiv: 2108.12763.

\bibitem[DV]{DV} B. Deng and M. Voineagu, {\em $RO(C_2\times C_2)$-graded cohomology ring of a point and applications\/}, 2025 preprint, arXiv: 2502: 00956.

\bibitem[D]{D} D. Dugger, {\em Coherence for invertible objects and multi-graded homotopy rings\/}, Algebr. Geom. Topol. {\bf 14} (2014), 1055--1106.

\bibitem[DH]{DHone} D. Dugger and C. Hazel, {\em Mackey homological algebra over cyclic groups\/}, preprint, 2026.

\bibitem[G1]{G1} N. Georgakopoulos, {\em The $RO(C_4)$ integral homology of a point\/}, 2019 preprint, arXiv: 1912:06758.

\bibitem[HHR1]{HHR1} M.~A. Hill, M.~J. Hopkins, and D.~C. Ravenel, {\em On the nonexistence of elements of Kervaire invariant one\/}, Ann. of Math. (2) {\bf 184} (2016), no. 1, 1--262.

\bibitem[HHR2]{HHR2} M.~A. Hill, M.~J. Hopkins, and D.~C. Ravenel, {\em The slice spectral sequence for certain $RO(C_{p^n})$-graded suspensions of $H\mZ$\/}, Bol. Soc. Mat. Mex. (3) {\bf 23} (2017), no. 1, 289--317.

\bibitem[HHR3]{HHR3} M.~A. Hill, M.~J. Hopkins, and D.~C. Ravenel, {\em Equivariant stable homotopy theory and the Kervaire invariant one problem\/}, New Math. Monogr. {\bf 40}, Cambridge University Press, Cambridge, 2021.

\bibitem[HK]{HK} J. Holler and I. Kriz, {\em On $RO(G)$-graded ``ordinary'' cohomology where $G$ is a power of $\Z/2$\/}, Algebr. Geom. Topol. {\bf 17} (2017), 741--763.

\bibitem[KL]{KL} I. Kriz and Y. Lu, {\em On the $RO(G)$-graded coefficients of dihedral equivariant cohomology\/}, Math. Res. Lett. {\bf 27} (2020), no. 4, 1109--1128.

\bibitem[Li]{Li} Y. Liu, {\em Computing equivariant homology with a splitting method\/}, Topology Appl. {\bf 334} (2023), Paper No. 108554, 57pp.

\bibitem[Lu]{Lu} Y. Lu, {\em On the $RO(G)$-graded coefficients of $Q_8$-equivariant cohomology\/}, Topology Appl. {\bf 307} (2022), Paper No. 107921, 15pp.

\bibitem[SS]{SS} S. Schwede and B. Shipley, {\em Equivalences of monoidal model categories\/}, Algebr. Geom. Topol. {\bf 3} (2003), no.~1, 287--334.

\bibitem[Ya1]{Ya1} G. Yan, {\em The $RO(C_{2^n})$-graded homotopy of $H\mZ$ through generalized Tate squares\/}, Matematica {\bf 4} (2025), no. 2, 364--418.

\bibitem[Ya2]{Ya2} G. Yan, {\em The Green functor structure of $RO(D_{2p})$-graded cohomology of a point\/}, 2023 preprint, arXiv: 2310.16118.

\bibitem[Z]{Z} M. Zeng, {\em Equivariant Eilenberg-MacLane spectra in cyclic $p$-groups\/}, preprint, 2017, arXiv:1710.01769.

\end{thebibliography}

\end{document}